\documentclass{amsart}
\usepackage[all,cmtip]{xy}
\usepackage{enumerate}
\usepackage{amssymb}
\usepackage{color}

\newcommand{\ie}{{\em i.e.\ }}
\newcommand{\eg}{{\em e.g.\ }}

\newtheorem{theorem}[subsection]{Theorem}
\newtheorem*{thm}{Theorem}

\newtheorem{lemma}[subsection]{Lemma}
\newtheorem{proposition}[subsection]{Proposition}
\newtheorem{corollary}[subsection]{Corollary}

\theoremstyle{definition}
\newtheorem{definition}[subsection]{Definition}

\newtheorem{notation}[subsection]{\textbf{Notation}}

\newtheorem{example}[subsection]{Example}
\newtheorem{examples}[subsection]{Examples}

\newtheorem{remark}[subsection]{\textbf{Remark}}

\numberwithin{equation}{section}

\newcommand{\N}{\mathbf{N}}
\newcommand{\Z}{\mathbf{Z}}

\newcommand{\ko}{\: , \;}

\newcommand{\ul}{\underline}
\newcommand{\we}{\wedge}
\newcommand{\che}{\vee}
\renewcommand{\tilde}[1]{\widetilde{#1}}

\newcommand{\ra}{\rightarrow}

\newcommand{\arr}[1]{\stackrel{#1}{\rightarrow}}

\newcommand{\opname}[1]{\operatorname{\mathsf{#1}}}

\newcommand{\Mod}{\opname{Mod}\nolimits}

\newcommand{\Tria}{{\opname{Tria}}}

\newcommand{\id}{\mathbf{1}}

\renewcommand{\L}{\mathbf{L}}

\newcommand{\im}{\opname{im}\nolimits}
\renewcommand{\ker}{\opname{ker}\nolimits}
\newcommand{\colim}{\colim}

\newcommand{\can}{\opname{can}}

\newcommand{\op}[1]{\opname{#1}\nolimits}

\newcommand{\ca}{{\mathcal A}}
\newcommand{\cb}{{\mathcal B}}
\newcommand{\cc}{{\mathcal C}}
\newcommand{\cd}{{\mathcal D}}
\newcommand{\ce}{{\mathcal E}}
\newcommand{\cF}{{\mathcal F}}
\newcommand{\cg}{{\mathcal G}}
\newcommand{\ch}{{\mathcal H}}

\newcommand{\cm}{{\mathcal M}}
\newcommand{\cn}{{\mathcal N}}

\newcommand{\cp}{{\mathcal P}}
\newcommand{\cR}{{\mathcal R}}

\newcommand{\cs}{{\mathcal S}}
\newcommand{\ct}{{\mathcal T}}

\newcommand{\cx}{{\mathcal X}}
\newcommand{\cy}{{\mathcal Y}}
\newcommand{\cz}{{\mathcal Z}}

\newcommand{\Hom}{\opname{Hom}}
\newcommand{\RHom}{\opname{RHom}}
\newcommand{\REnd}{\opname{REnd}}
\newcommand{\End}{\opname{End}}

\newcommand{\Ext}{\opname{Ext}}

\newcommand{\uni}{\mathbf{U}}

\newcommand{\tria}{\opname{tria}}

\newcommand{\per}{\opname{per}}
\newcommand{\thick}{\opname{thick}}

\makeindex
\begin{document}
\title{Generalized tilting theory}

\author{Pedro Nicol\'as}
\address{Pedro Nicol\'as, Universidad de Murcia, Facultad de Educaci\'on, Campus de Espinardo, 30100 Murcia, ESPA\~NA}
\email{pedronz@um.es}

\author{Manuel Saor\'in}
\address{Manuel Saor\'in, Departamento de Matem\'aticas\\
Universidad de Murcia, Aptdo. 4021\\
30100 Espinardo, Murcia\\
SPAIN}
\email{msaorinc@um.es}

\thanks{The authors are supported by research projects from the Ministerio de Econom\'ia y Competitividad of Spain (MTM2013-46837-P and MTM2016-77445-P) and from the Fundaci\'on 'S\'eneca' of Murcia (19880/GERM/15), both with a part of FEDER funds.}

\date{} 

\maketitle

\begin{center}
Dedicated to the memory of Michael Butler and Dieter Happel.
\end{center}

\tableofcontents

\section{Introduction}

The concept of tilting module has its roots in the work of  Bernstein, Gelfand and Ponomarev \cite{BGP} on reflection functors. They used these functors to construct all the indecomposable modules over a representation-finite hereditary algebra starting from the simple ones. The reflection functors were interpreted and generalized in more categorical terms in \cite{APR}, producing the first example of tilting modules, today called {\em APR-tilted modules}, although the term ``tilting module'' did not appear yet. Brenner and Butler \cite{BrennerButler} were the first to use the term  and gave a first list of conditions that a module over a finite dimensional algebra has to satisfy in order to be a tilting module. Alternative more polished lists, still in the context of finite dimensional algebras over a field, were given shortly after by Happel and Ringel \cite{HR} and by Bongartz \cite{Bo} and they laid the foundation of what is today called Tilting Theory. In essence, it was a generalization of Morita Theory. While this studied equivalences between categories of modules over two algebras (or rings), what is today known as the Tilting theorem \cite{ColbyFuller} states that if $T$ is a tilting module over the algebra $A$ and $B=\End_A(T)$, then 
$(\ct,\cF)=(\op{ker}\op{Ext}_A^1(T,?),\op{Ker}\Hom_A(T,?))$  and $(\cx,\cy)=(\op{ker}(?\otimes_BT),\op{ker}\op{Tor}_1^B(?,T))$ are torsion theories in $\Mod A$ and $\Mod B$, respectively, and there are equivalences of categories $\Hom_A(T,?):\ct\arr{\sim}\cy:?\otimes_BT$ and 
$\op{Ext}_A^1(T,?):\cF\arr{\sim}\cx:\op{Tor}_1^B(?,T)$.

The original tilting modules were of projective dimension less than or equal to $1$. A few years later, Miyashita \cite{Miy} generalized the concept allowing modules of any finite projective dimension and over any ring, for which a generalization of the Tilting theorem was still valid. The most frequently used terminology nowadays, that we shall follow in this paper, calls the modules introduced
by Miyashita {\em classical $n$-tilting modules} (see Definition~\ref{tilting module}), which coincide  in the case $n=1$ with the ones introduced by Brenner-Butler and Happel-Ringel.

The tilting modules of Miyashita, and hence their precursors, are modules which have a finite projective resolution with finitely generated terms. In particular, they are always finitely generated. It seems that Colpi and Trlifaj \cite{CT} were the first to generalize the notion of (1-)tilting module, as introduced by
Brenner-Butler and Happel-Ringel, to not necessarily finitely generated modules. Later on, Angeleri-H\"{u}gel and Coelho \cite{Ang-Coelho} did the same with the concept of Miyashita and used for the first time the term ``(n-)tilting module'' as it is used nowadays (see Definition~\ref{tilting module}).

As Morita Theory had its contravariant version, the so-called Morita duality, parallel to Tilting Theory, in its finitely and infinitely generated versions, a contravariant version developed throughout the years and the concept of cotilting module entered the scene. It seems that it was Colby \cite{Colby} the first to use the term to name what is today known as classical 1-cotilting module. The general version, for arbitrary (finite) injective dimension and not necessarily finitely generated modules, seems to have been first used in \cite{Ang-Coelho}.

It is rare to see an algebraic theory that has been more fruitful than Tilting Theory, convening to consider within it also its contravariant version. It has impregnated several fields of Mathematics. Without intention to be exhaustive, we just mention a few. In the representation theory of finite dimensional algebras the theory gave rise to new classes of algebras (\eg  tilted, piecewise hereditary, tubular, quasi-tilted, etc.) whose representation type (finite, tame or wild) and structure of their modules categories could be effectively determined by applying the corresponding knowledge over path algebras of quivers (see \cite{HR}, \cite{Ri}, \cite[Chapters VIII, IX]{ASS}, \cite{Si-Skow}, \cite{HRS1}, \cite{HRS2}). It also helped to classify homologically well-behaved subcategories (\eg resolving, contra- and co-variantly finite, torsion, cotorsion, etc) either of the category $\op{mod}A$ of finitely generated modules or of the category $\Mod A$ of all  modules (see \cite{AR}, \cite{Ang-Coelho}, \cite{KS}, \cite{BK}). It helped also to tackle from a new angle long-standing homological open questions, like the finitistic dimension  and the Nakayama conjectures (see \cite[\S IV.3]{BR}, \cite{Ang-Trl}).

The result in \cite{Ang-Trl} about the finitistic dimension conjecture is a particular case of a more general one which shows one of the amazing phenomena of (co)tilting modules, namely, that many aspects of the category $\op{mod}A$ are controlled by infinitely generated (co)tilting modules and this fact goes beyond
finite dimensional or Artin algebras. Indeed, roughly speaking, the mentioned result (see \cite{AHT}) states that, for every ring $A$, the tilting modules in $\Mod A$ parametrize all resolving subcategories of $\op{mod}A$ consisting of modules which admit a finite projective resolution with finitely generated terms. Other
examples of the interaction between infinitely generated (co)tilting modules and properties of $\op{mod}A$ are given in the classification of tilting and cotilting modules over hereditary algebras finite dimensional algebras (see \cite{BK} and \cite{Ang-San}). Still in the field of module theory over arbitrary rings, let us mention an important result, namely, that each $n$-cotilting module is pure-injective (see \cite{Baz}, for the case $n=1$, and \cite{Sto} for the general case). In this form cotilting modules connect with the Ziegler spectrum of a ring and the notion of definable subcategory, having a great impact, among other things,
on the study of cotorsion pairs in modules categories (see \cite{GT}).

Apart from its original habitat of module theory and representation theory, tilting theory has been used in the representation of algebraic groups and Lie algebras (see \cite{D}), in  Algebraic Geometry in order to study  the category of coherent sheaves on a weighted projective line and, more generally, to study
hereditary categories with finite-dimensional $\Hom$ and $\op{Ext}$ spaces (see \cite{GL}, \cite{GL2}, \cite{Meltz} and \cite{L}) and it has been fundamental in the categorification process of the study of cluster algebras, as defined by Fomin and Zelevinsky in \cite{FZ}, a concept that appears in fields as diverse as Poisson geometry, discrete dynamical systems, Algebraic Geometry and the study of polyhedra (see \cite{K-clI} and \cite{K-clII} for excellent surveys on the topic).

One of the fundamental consequences of Tilting Theory, which is the most interesting for us in this paper, is that it was the precursor of the so-called Morita theory for derived categories. Indeed, Happel \cite{Happel} proved that if $T$ is a classical tilting module (see Definition~\ref{tilting module}) over a finite dimensional algebra $A$ and $B=\End_A(T)$ is its endomorphism algebra, then the derived functors of $\Hom$ and the tensor product induce an equivalence $\cd^bA\arr{\sim}\cd^bB$ between the bounded derived categories of the corresponding  categories of finitely generated modules. The result inspired Rickard and Keller (see \cite{Rickard} and \cite{KellerDDC}) who gave necessary and sufficient conditions for two (not necessarily finite dimensional) algebras to have equivalent derived categories. As a consequence, we have:

\begin{thm}
Let $A$ be an ordinary algebra, $T$ be a right $A$-module and $B=\End_A(T)$ be its endomorphism algebra. The following assertions are equivalent:
\begin{itemize}
\item[1)] $T$ is a classical tilting $A$-module. 
\item[2)] The functor $\RHom_A(T,?):\cd A\ra \cd B$ is an equivalence of categories. 
\item[3)] The functor $T\otimes^\L_A?:\cd(A^\text{op})\ra\cd(B^\text{op})$ is an equivalence of categories.
\end{itemize}
\end{thm}

It seems natural to weaken the hypotheses on the functors in assertions 2 and 3 of the theorem, by requiring only that they be fully faithful, and try to see what sort of modules we do get. Another natural question arises also, namely, the connection of these fully faithful hypotheses with the general concept of
(infinitely generated) tilting module. A recent result by Bazzoni-Mantese-Tonolo \cite{BazzoniManteseTonolo} points in this direction and shows that the so-called
good tilting modules (see Definition~\ref{tilting module}) provide an example of modules $T$ for which the functors
$\RHom_A(T,?):\cd A\ra\cd B$ and $T\otimes^\L_A?:\cd(A^\text{op})\ra\cd(B^\text{op})$ are fully faithful. The result was extended by Yang \cite{Yang} to the context of derived categories of dg categories over a field, by showing that, for a natural substitute of the notion of good tilting module in this context, the corresponding $\RHom$ functor was fully faithful and presented the original derived category as part of a recollement of derived categories of dg categories. When coming back to an ordinary algebra $A$ and a good tilting module $T$, with $B=\End_A(T)$ as its endomorphism algebra, Yang's result implies that $\cd B$ is a recollement of $\cd A$ and the derived category of a dg algebra. Recent papers by Chen and Xi \cite{Chen-Xi1, Chen-Xi2} give conditions under which this dg algebra can be replaced by an ordinary algebra.

The first main result of this paper, Theorem~\ref{teor.derived tensor fully faithful}, characterizes the situation where, for a $\cb$-$\ca$-bimodule $T$ over dg categories $\cb$ and $\ca$, the functor $?\otimes^\L_\cb T:\cd\cb\ra\cd\ca$ is fully faithful (and hence induces a semiorthogonal decomposition of $\cd\ca$, since it has a right adjoint).

The second main result of this paper, Theorem~\ref{special localization result}, characterizes, up to quasi-isomorphism of dg categories, the situation where we, furthermore, have another dg $\ca$-$\cb$-bimodule $T'$ such that $T'\otimes^\L_{\cb}T$ is quasi-isomorphic to $\ca$ (as a dg
$\ca$-$\ca$-bimodule, \ie when $T$ has a left tensor inverse $T'$). In that case we can express $\cd\cb$ as a recollement of $\cd\ca$ and the derived category of another dg category.

Our results, when applied to ordinary algebras, show the relationship of the appearing bimodules with the concept of good tilting module and, more generally, with tilting theory. Also for ordinary algebras, our results cover and extend some results in \cite{Chen-Xi1, Chen-Xi2}.

Now we show some consequences, in the world of ordinary algebras, of our main results.

Let $k$ be a commutative ring, $A$ an ordinary $k$-algebra and assume in this paragraph that $T$ a right $A$-module such that $\op{Ext}^p_A(T,T)=0$ for all $p>0$ and $B=\End_A(T)$. Then, after Corollary~\ref{generalized tilting and recollements}, there exists a dg algebra $C$ and a recollement
\[\xymatrix{
\cd A\ar@/_2.5pc/[d]\ar@/_-2.5pc/[d] \\
\cd B\ar[u]^{j^{*}}\ar@/_2.5pc/[d]\ar@/_-2.5pc/[d] \\
\cd C\ar[u]^{i_*}
}
\]
such that $j^{*}=?\otimes^\L_{B}T$, if and only if there exists an exact sequence in $\Mod A$
\[
0\ra A\ra T^0\ra T^{1}\ra\dots\ra T^{n}\ra 0,
\]
where each $T^i$ is a direct summand of a finite direct sum of copies of $T$. In case $B$ is $k$-flat, the dg algebra $C$ can be chosen so that there is a homological epimorphism of dg algebras $f:B\ra C$ such that $i_{*}=f_{*}$ is the restriction of scalars along $f$. Also, after Corollary~\ref{recollement and bimodule}, there is a recollement
\[\xymatrix{
\cd B\ar@/_2.5pc/[d]_{j_{!}}\ar@/_-2.5pc/[d] \\
\cd A\ar[u]\ar@/_2.5pc/[d]\ar@/_-2.5pc/[d] \\
\cy\ar[u]
}
\]
with $j_{!}=?\otimes^\L_{B}T$ if and only if $T$ admits a finite projective resolution with finitely generated terms as a right $A$-module.

Assuming $\Ext^p_{A}(T,T^{(\alpha)})=0$ for each cardinal $\alpha$ and each integer $p>0$, Corollary~\ref{some characterization of good tilting} characterizes in terms of the derived functors $\RHom_{A}(T,?)$ and $?\otimes^\L_{B}T$ when $T$ is a good tilting module. Assuming $A$ is right hereditary, Corollary~\ref{1tilting inducing ff} characterizes, in terms of the categories of modules over $A$ and $B$, when $\RHom_{A}(T,?):\cd A\ra\cd B$ is fully faithful.

Concerning results about absence of communion between `tiltingness' of $T$ and nice properties of $\RHom_A(T,?)$, Theorem~\ref{main result ordinary case} says that (with $B=\End_A(T)$):
\begin{itemize}
\item[1)] There are right $A$-modules $T$ satisfying the following conditions:
\begin{itemize}
\item[a)] $\op{pd}_{A}T\leq 1$,
\item[b)] the functor $\RHom_{A}(T,?):\cd A\ra\cd B$ is fully faithful, preserves compact objects and $B$ is in its essential image,
\item[c)] $\Ext_A(T,T^{(\N)})\neq 0$ and so $T$ is not a tilting module.
\end{itemize}
\item[2)] There are right $A$-modules $T$ satisfying the following conditions:
\begin{itemize}
\item[a)] $\op{pd}_{A}T\leq 1$,
\item[b)] the functor $\RHom_{A}(T,?):\cd A\ra\cd B$ is fully faithful,
\item[c)] $\Ext^{1}_{A}(T,T)\neq 0$.
\end{itemize}
\item[3)] There are right $A$-modules $T$ satisfying the following conditions:
\begin{itemize}
\item[a)] $T$ is a $1$-tilting module,
\item[b)] $\RHom_{A}(T,?):\cd A\ra\cd B$ is not fully faithful.
\end{itemize}
\end{itemize}

We use sections \ref{Notation and terminology}, \ref{Basics of categories}, \ref{The derived category of a dg category} and \ref{Adjunctions} both to fix notation and to introduce the non-expert readers into the subject. 
\bigskip

\textbf{Acknowledgments}
\bigskip

The authors thank the referee for his careful reading of the first version of
this article and his helpful remarks. 

\section{Notation and terminology}\label{Notation and terminology}

Concerning set-theoretical technicalities we follow the example of \cite[\S II.1.3]{GelfandManin}, \cite[\S I.1]{Gabriel}, \cite{GabrielPopescu}, \cite{MacLane}, \cite{Rickard},\dots based on Grothendieck's viewpoint \cite{SGA4}. We add the following axiom to ZFC: 
\bigskip

\textbf{Universe axiom:} Every set is an element of a universe.
\bigskip

A universe is a set closed under all possible operations on abstract sets (see a definition in \cite{SGA4} or \cite{Gabriel}). After adding the universe axiom to ZFC, we fix a universe $\uni$ `big enough' (\ie containing all the rings considered, etc.: it turns out that this universe will be a model of ZFC). We say that a set $S$ is {\em small} if $S\in\uni$. For us, if $\cc$ is a category, then its objects form a set, and each morphism space $\Hom_{\cc}(M,N)$ is a small set. We then say that a category $\cc$ is {\em small} if its objects form a small set. We speak of {\em small} (co)products when we form the (co)product of a small set of objects.
\bigskip

Let $\cc$ be an additive category and let $\cs$ be a set of objects of $\cc$. Then:

\begin{itemize}
\item $\op{add}_{\cc}(\cs)$ (or $\op{add}(\cs)$, if the $\cc$ is understood) denotes the full subcategory of $\cc$ formed by direct summands of finite coproducts of objects of $\cs$. When $\cs$ has only one element $M$ we write $\op{add}_{\cc}(M)$ instead of $\op{add}_{\cc}(\{M\})$.
\item $\op{Add}_{\cc}(\cs)$ (or $\op{Add}(\cs)$, if the $\cc$ is understood) denotes the full subcategory of $\cc$ formed by direct summands of small coproducts of objects of $\cs$. When $\cs$ has only one element $M$ we write $\op{Add}_{\cc}(M)$ instead of $\op{Add}_{\cc}(\{M\})$.
\item $\cs^\perp$ is the full subcategory of $\cc$ formed by those objects $M$ such that $\Hom_{\cc}(S,M)=0$ for each $S\in\cs$.
\item $^{\perp}\cs$ is the full subcategory of $\cc$ formed by those objects $M$ such that $\Hom_{\cc}(M,S)=0$ for each $S\in\cs$.
\end{itemize}

Given a triangulated category, a {\em full triangulated subcategory} is a full additive subcategory closed under shifts and extensions, a {\em thick subcategory} is a full triangulated subcategory closed under direct summands, and a {\em localizing subcategory} is a full triangulated subcategory closed under small coproducts. Let $\cd$ be a triangulated category and let $\cs$ be a set of objects of $\cc$. Then:

\begin{itemize}
\item $\op{tria}_{\cd}(\cs)$ (or $\op{tria}(\cs)$, if the $\cd$ is understood) is the smallest full triangulated subcategory of $\cd$ containing $\cs$. When $\cs$ has only one element $M$ we write $\op{tria}_{\cd}(M)$ instead of $\op{tria}_{\cd}(\{M\})$.
\item $\op{Tria}_{\cd}(\cs)$ (or $\op{Tria}(\cs)$, if the $\cd$ is understood) is the smallest localizing subcategory of $\cd$ containing $\cs$. When $\cs$ has only one element $M$ we write $\op{Tria}_{\cd}(M)$ instead of $\op{Tria}_{\cd}(\{M\})$.
\item $\thick_{\cd}(\cs)$ (or $\thick(\cs)$, if the $\cd$ is understood) is the smallest thick subcategory of $\cd$ containing $\cs$. When $\cs$ has only one element $M$ we write $\thick_{\cd}(M)$ instead of $\thick_{\cd}(\{M\})$.
\end{itemize}

We will always assume that there is a fixed ground ring $k$ (commutative, associative, with identity) and that our categories are $k$-linear. We will also assume that the tensor product $\ca\otimes\cb$ of two dg categories $\ca$ and $\cb$ is made over $k$. The existence of such a ring is important for example in Notation~\ref{one-functors}, Notation~\ref{two-functors}, Lemma~\ref{defining triangulated 2-functors}, Proposition~\ref{relation among dualities}, Notation~\ref{B-dual}, Lemma~\ref{B-duality and isos}, Theorem~\ref{special localization result}, \dots where we assume certain dg categories satisfy certain properties relative to $k$.

\begin{definition}\label{tilting module}
Let $A$ be an ordinary algebra, and let $T$ be a right $A$-module. Consider the following conditions:
\begin{itemize}
\item[a)] There is an exact sequence $0\ra P_{n}\ra P_{n-1}\ra\dots\ra P_{1}\ra P_{0}\ra T\ra 0$ in $\Mod A$ where the modules $P_{i}$ are projective.
\item[a')] There is an exact sequence $0\ra P_{n}\ra P_{n-1}\ra\dots\ra P_{1}\ra P_{0}\ra T\ra 0$ in $\Mod A$ where the modules $P_{i}$ are finitely generated projective.
\item[b)] There is an exact sequence $0\ra A\ra T^{0}\ra T^{1}\ra\dots\ra T^{m}\ra 0$ in $\Mod A$ where the modules $T^{i}$ are in $\op{Add}(T)$.
\item[b')] There is an exact sequence $0\ra A\ra T^{0}\ra T^{1}\ra\dots\ra T^{m}\ra 0$ in $\Mod A$ where the modules $T^{i}$ are in $\op{add}(T)$.
\item[c)] $\Ext^{p}_{A}(T,T^{(\alpha)})=0$ for each $p>0$ and each cardinal $\alpha$.
\end{itemize}
We say that $T$ is an {\em $n$-tilting} module if it satisfies a), b) and c). We say that it is a {\em classical $n$-tilting} module if it satisfies a'), b) and c). We say that $T$ is a {\em good $n$-tilting} module if it satisfies a), b') and c). We say that a module is {\em tilting} (resp. {\em classical tilting}, {\em good tilting}) if it is $n$-tilting (resp. classical $n$-tilting, good $n$-tilting) for some $n\geq 1$.
\end{definition}

\begin{remark}\label{terminology infinitely generated tilting}
To emphasize the fact that a tilting module is not classical, sometimes the literature speaks of an {\em infinitely generated} tilting module.
\end{remark}

\section{Basics of categories}\label{Basics of categories}

\subsection{Localizations}

Recall Proposition I.1.3. of \cite{GabrielZisman}:

\begin{proposition}\label{GabrielZisman result}
Let 
\[\xymatrix{ \cd\ar@<1ex>[d]^{R} \\
\cc\ar@<1ex>[u]^{L}
}
\]
be an adjoint pair of functors between arbitrary categories, and let $\cs$ be the set of morphisms $u$ of $\cc$ such that $L(u)$ is an isomorphism. The following conditions are equivalent:
\begin{itemize}
\item[1)] $R$ is fully faithful.
\item[2)] The counit of the adjunction is an isomorphism of functors.
\item[3)] The functor $\cc\left[\cs^{-1}\right]\ra \cd$ induced by $L$ is an equivalence, where $\cc\left[\cs^{-1}\right]$ is the {\em category of fractions of $\cc$ for $\cs$} (see \cite[Chapter One, \S 1]{GabrielZisman}).
\end{itemize}
\end{proposition}

\subsection{Generalities about triangulated categories}

The following notion comes from Algebraic Topology:

\begin{definition}
Let $\cd$ be a triangulated category. A \emph{localization} of $\cd$ is a triple $(L,\alpha,\eta)$ where $(L,\alpha)$ is a triangle endofunctor of $\cd$ and $\eta:\id\ra L$ is a natural transformation such that:
\begin{itemize}
\item[a)] $L\eta:L\arr{\sim}L^2$ is an isomorphism,
\item[b)] $L\eta=\eta L$,
\item[c)] $\eta$ commutes with the shift functor, \ie for each $M\in\cd$ the following diagram is commutative:
\[\xymatrix{\Sigma M\ar[r]^{\Sigma\eta_{M}}\ar[dr]_{\eta_{\Sigma M}} & \Sigma(LM) \\
& L(\Sigma M)\ar[u]_{\alpha_{M}}^{\wr}
}
\] 
 \end{itemize}
The localization is said to be \emph{smashing} if $L$ preserves small coproducts. A full subcategory $\cx$ of $\cd$ is said to be {\em smashing} if it is the kernel of a smashing localization functor of $\cd$. A full subcategory $\cz$ of $\cd$ is said to be {\em co-smashing} if it becomes a smashing subcategory when regarded inside the opposite triangulated category $\cd^\text{op}$.
\end{definition}

\begin{remark}
Let $\cd$ be a triangulated category with small coproducts. One can prove (see \cite[Proposition 4.4.3]{NicolasTesis}) that a full triangulated subcategory $\cx$ of $\cd$ is smashing if and only if the inclusion functor $\cx\ra\cd$ admits a right adjoint which preserves small coproducts. Dually, if $\cd$ is a triangulated category with small products, then a full triangulated subcategory $\cz$ of $\cd$ is co-smashing if and only if the inclusion functor $\cz\ra\cd$ admits a left adjoint which preserves small products.
\end{remark}

The following notion comes from Algebraic Geometry:

\begin{definition}
A {\em ($2$-steps) semi-orthogonal decomposition} of a triangulated category $\cd$ is a pair $(\cx,\cy)$ of strictly full triangulated subcategories of $\cd$ such that:
\begin{enumerate}
\item $\Hom_{\cd}(\cx,\cy)=0$,
\item $\cd=\op{tria}_{\cd}(\cx\cup\cy)$.
\end{enumerate}
\end{definition}

The following lemma characterizes the second condition in the former definition:

\begin{lemma}
Let $\cd$ be a triangulated category and let $(\cx,\cy)$ be a pair of strictly full triangulated subcategories of $\cd$ such that $\Hom_{\cd}(\cx,\cy)=0$. The following properties are equivalent:
\begin{enumerate}
\item $\cd=\op{tria}_{\cd}(\cx\cup\cy)$.
\item For each $M\in\cd$ there exists a triangle
\[X\ra M\ra Y\ra\Sigma X
\]
of $\cd$ with $X\in\cx$ and $Y\in\cy$.
\item The inclusion $x:\cx\ra\cd$ has a right adjoint (denoted $\tau_{\cx}$).
\item The inclusion $y:\cy\ra\cd$ has a left adjoint (denoted $\tau^{\cy}$).
\end{enumerate}
\end{lemma}

\begin{remark}
According to Beligiannis-Reiten \cite{BR}, Bondal and Kapranov introduced a concept in \cite{BondalKapranov} which was latter called {\em semi-orthogonal decomposition} by Reiten-Van den Bergh in \cite{RVdB}. 
\end{remark}

Inspired by a notion in Module Theory, we can define the following:

\begin{definition}
A {\em (triangulated) torsion torsionfree(=TTF) triple} in a triangulated category $\cd$ is a triple $(\cx,\cy,\cz)$ such that both $(\cx,\cy)$ and $(\cy,\cz)$ are semi-orthogonal decompositions.
\end{definition}

The notion of `recollement' was introduced by A.~A.~Beilinson, J.~Bernstein and P.~Deligne in their work \cite{BeilinsonBernsteinDeligne} in Algebraic Analysis. 

\begin{definition}
Let $\cd\ko \cd_{F}$ and $\cd_{U}$ be triangulated categories. Then $\cd$ is a \emph{recollement of $\cd_{F}$ and $\cd_{U}$}, diagrammatically expressed by
\[\xymatrix{
\cd_{U}\ar@/_2.5pc/[d]_{j_{!}}\ar@/_-2.5pc/[d]^{j_{*}} \\
\cd\ar[u]^{j^{!}=j^{*}}\ar@/_2.5pc/[d]_{i^{*}}\ar@/_-2.5pc/[d]^{i^{!}} \\
\cd_{F}\ar[u]^{i_{*}=i_{!}}
}
\]
if there exist six triangle functors which satisfies the following four conditions:
\begin{enumerate}[R1)]
\item $(i^*,i_{*}=i_{!},i^!)$ and $(j_{!},j^!=j^*,j_{*})$ are adjoint triples.
\item $i^!j_{*}=0$ (and thus $j^!i_{!}=j^*i_{*}=0$ and $i^*j_{!}=0$).
\item $i_{*}\ko j_{!}$ and $j_{*}$ are full embeddings (and thus $i^*i_{*}\cong i^!i_{!}\cong \id_{\cd_{F}}$).
\item For every object $M$ of $\cd$ there exist two triangles,
\[i_{!}i^!M\ra M\ra j_{*}j^*M\ra (i_{!}i^!M)[1]
\]
and
\[j_{!}j^!M\ra M\ra i_{*}i^*M\ra (j_{!}j^!M)[1],
\]
in which the morphisms $i_{!}i^!M\ra M$ etc. are the corresponding adjunction morphisms.
\end{enumerate}
In this case we say that the data 
\[(\cd_{F},\cd_{U}, i^*,i_{*}=i_{!},i^!, j_{!},j^!=j^*,j_{*})
\] 
is a \emph{d\'{e}collement} of $\cd$.
Two d\'{e}collements of $\cd$,
\[\xymatrix{
\cd_{U}\ar@/_2.5pc/[d]_{j_{!}}\ar@/_-2.5pc/[d]^{j_{*}} \\
\cd\ar[u]^{j^{!}=j^{*}}\ar@/_2.5pc/[d]_{i^{*}}\ar@/_-2.5pc/[d]^{i^{!}} \\
\cd_{F}\ar[u]^{i_{*}=i_{!}}
}
\]
and
\[\xymatrix{
\cd'_{U}\ar@/_2.5pc/[d]_{j'_{!}}\ar@/_-2.5pc/[d]^{j'_{*}} \\
\cd\ar[u]^{j'^{!}=j'^{*}}\ar@/_2.5pc/[d]_{i'^{*}}\ar@/_-2.5pc/[d]^{i'^{!}} \\
\cd'_{F}\ar[u]^{i'_{*}=i'_{!}}
}
\]
are \emph{equivalent}\index{d\'{e}collement!equivalent} if 
\[(\im(i_{*}),\im(j_{*}),\im(j_{!}))=(\im(i'_{*}),\im(j'_{*}),\im(j'_{!})),
\] 
where by $\im(i_{*})$ we mean the essential image of $i_{*}$, and analogously with the other functors.
\end{definition}

\begin{theorem}
Let $\cd$ be a triangulated category with small coproducts.
\begin{enumerate}
\item The assignment
\[L\mapsto (\ker(L),\op{im}(L))=(\ker(L),\ker(L)^\perp)
\]
underlies a bijection between equivalence classes of localization functors and semi-orthogonal decompositions. The inverse map sends $(\cx,\cy)$ to the composition
\[\xymatrix{
L:\cd\ar[r]^{\tau^{\cy}} & \cy\ar[r]^{y} & \cd.
}
\]
Moreover, if $\cd$ is {\em perfectly generated} (see \cite[Definition 1]{Krause}, for example $\cd$ can be compactly generated, see Definition~\ref{def:compactly generated}), then under this bijection smashing localizations correspond to semi-orthogonal decompositions $(\cx,\cy)$ which fit into a TTF triple $(\cx,\cy,\cz)$.
\item The assignment
\[(\cd_{F},\cd_{U}, i^*,i_{*}=i_{!},i^!, j_{!},j^!=j^*,j_{*})\mapsto (j_{!}(\cd_{U}),i_{*}(\cd_{F}),j_{*}(\cd_{U}))
\]
yields a bijection between equivalence classes of d\'ecollements of $\cd$ and TTF triples on $\cd$. The inverse map takes the TTF triple $(\cx,\cy,\cz)$ to the class of the d\'ecollement
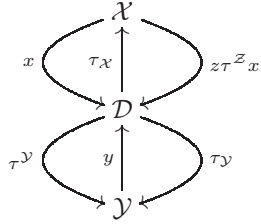
\[\xymatrix{
\cx\ar@/_2.5pc/[d]_{x}\ar@/_-2.5pc/[d]^{z\tau^{\cz}x} \\
\cd\ar[u]^{\tau_{\cx}}\ar@/_2.5pc/[d]_{\tau^\cy}\ar@/_-2.5pc/[d]^{\tau_{\cy}} \\
\cy\ar[u]^{y}
}
\]
\end{enumerate}
\end{theorem}
\begin{proof}
(1) The fact that $(\ker(L),\op{im}(L))$ is a semi-orthogonal decomposition can be found in implication $(5)\Rightarrow (4)$ of the proof of \cite[Proposition 4.4.3]{NicolasTesis}.

The proof of the bijection between smashing localizations and TTF triples can be found in \cite[Proposition 4.4.14]{NicolasTesis}.

(2) See Proposition 4.2.4 and Proposition 4.2.5 of \cite{NicolasTesis}.
\end{proof}

\begin{lemma}\label{two-sided fully faithful adjoint}
Let $F:\cd\ra\cd'$ be a triangle functor between triangulated categories and suppose that $F$ has a left adjoint  $L$ and a right adjoint $R$. Then $L$ is
fully faithful if, and only if, so is $R$. In such case,  the triple $(\op{im}(L),\op{ker}(F),\op{im}(R))$ is a triangulated TTF triple in $\cd$.
\end{lemma}
\begin{proof}
By the duality principle, it is enough to prove one of the implications.  We assume that $R$ is fully faithful. We denote by
$\mu: \id_{\cd'}\ra FL$ (resp. $\lambda: \id_{\cd}\ra RF$) and $\varepsilon: LF\ra \id_{\cd}$ (resp. $\delta: FR\ra \id_{\cd'}$) the unit and the counit of the adjunction $(L,F)$ (resp. $(F,R)$). Due to the fully faithful condition of $R$, we know that $\delta$ is an isomorphism.

We need to prove that $\mu(D')$ is an isomorphism, for each $D'\in\cd'$. This is equivalent to prove that $\mu(F(D))$ is an isomorphism, for each $D\in\cd$, because $\delta$ is an isomorphism and, therefore, each object of $\cd'$ is in the essential image of $F$. Using the adjunction identity $F(\varepsilon(D))\circ\mu(F(D))=\id_{F(D)}$, our task is reduced to check that $F(\varepsilon(D))$ is an isomorphism, for each $D\in\cd$.

Given $D$ as above, we have a commutative diagram of morphisms of functors
\[
\xymatrix{
\Hom_\cd(D,R(?))\ar[r]^{\varepsilon(D)^{\che}} & \Hom_\cd(LF(D),R(?))\ar[r]^{\sim} & \Hom_{\cd'}(F(D),FR(?))\ar[d]^{\delta^\we} \\
\Hom_{\cd'}(F(D),?)\ar[u]^{\wr}\ar[rr]^{\id} & & \Hom_{\cd'}(F(D),?).
}
\]
Note that $\delta^\we$ is an isomorphism since so is $\delta$. It follows that $\varepsilon(D)^{\che}$ is an isomorphism, for each $D\in\cd$. We then complete to a triangle
\[
LF(D)\arr{\varepsilon(D)} D\ra X\ra \Sigma LF(D)
\]
and get that $\Hom_\cd(X,R(?)):\cd'\ra \Mod\Z$ is the zero functor. By adjunction, it follows that
$\Hom_{\cd'}(F(X),?)=0$, which is equivalent to say that $F(\Sigma^{p}X)=0$ for each $p\in\Z$. This implies that $F(\varepsilon(D))$ is an isomorphism, as desired.

In order to prove that $(\op{im}(L),\op{ker}(F),\op{im}(R))$ is a triangulated TTF triple one just has to prove the existence of suitable triangles, since the vanishing conditions are pretty clear. To construct the triangles one has to use the unit and the counit of the adjunctions, the relations between them, and the fact that the unit (resp. the counit) is an isomorphism if (and only if) the left (resp. right adjoint) is fully faithful (see Proposition~\ref{GabrielZisman result} and its dual).
\end{proof}

\begin{definition}
An object $M$ of a triangulated category $\cd$ is {\em compact} if the functor $\Hom_{\cd}(M,?):\cd\ra\Mod\Z$ commutes with small coproducts, \ie if $\{X_{i}\}_{i\in I}$ is a small family such that $\coprod_{i\in I}X_{i}$ exists in $\cd$, then the canonical morphism
\[\can:\coprod_{i\in I}\Hom_{\cd}(M,X_{i})\ra\Hom_{\cd}(M,\coprod_{i\in I}X_{i})
\]
is an isomorphism.
\end{definition}

\begin{definition}\label{def:symmetric generators}
Let $\cd$ be a triangulated category. A nonempty set $\cs$ of objects of $\cd$ is a {\em generating set} for $\cd$ if the following holds: 
\begin{itemize}
\item[1)] if an object $D\in\cd$ satisfies $\Hom_{\cd}(S,D)=0$ for each $S\in\cs$, then $D=0$.
\end{itemize}
A generating set $\cs$ is {\em symmetric} if there exists a nonempty set $\ct$ of objects of $\cd$ with the following property:
\begin{itemize}
\item[2)] for any morphism $X\ra Y$ the induced map $\Hom_{\cd}(S,X)\ra\Hom_{\cd}(S,Y)$ is surjective for every $S\in\cs$ if and only if $\Hom_{\cd}(Y,T)\ra\Hom_{\cd}(X,T)$ is injective for every $T\in\ct$.
\end{itemize}
\end{definition}

\begin{definition}\label{def:compactly generated}
A triangulated category $\cd$ is said to be {\em compactly generated} if it has small coproducts and a generating set formed by compact objects.
\end{definition}

\begin{remark}\label{compactly implies symmetrically}
A compactly generated triangulated category has a set of symmetric generators (see \cite{Krause}).
\end{remark}

\begin{corollary}\label{smashing in bijection with cosmashing}
Let $\cd$ be a triangulated category with small coproducts and with a symmetric set of generators. The assignments $\cx\mapsto\cx^{\perp\perp}$ and $\cz\mapsto\ ^{\perp\perp}\cz$ define mutually inverse bijections between:
\begin{itemize} 
\item[1)] The set of (compactly generated) smashing subcategories of $\cd$.
\item[2)] The set of (compactly generated) co-smashing subcategories of $\cd$.
\end{itemize}
\end{corollary}
\begin{proof}
Let $\cs$ and $\ct$ be as in Definition~\ref{def:symmetric generators}. It is easy to check that $\cs$ is a set of {\em perfect generators} for $\cd$ (see Definition 1 of \cite{Krause}). Therefore, after \cite[Theorem A]{Krause}, we know that any contravariant functor $\cd\ra\Mod\Z$ sending triangles to exact sequences and sending small coproducts to small products is representable. One can use this to prove that $\cd$ has small products, and so $\cd^{\text{op}}$ has small coproducts. Note that $\ct$ is a set of symmetric generators of $\cd^\text{op}$. Therefore, we conclude that $\cd^\text{op}$ has also a set of perfect generators. Now we can use \cite[Corollary 2.4]{NicolasSaorinParametrizing} to deduce the existence of a bijection between the sets formed by:
\begin{itemize}
\item[-] the smashing subcategories of $\cd$,
\item[-] the triangulated TTF triples on $\cd$ (which are precisely the triangulated TTF triples on $\cd^\text{op}$),
\item[-] the smashing subcategories of $\cd^\text{op}$.
\end{itemize}
Moreover, when $(\cx,\cy,\cz)$ is a triangulated TTF triple, there is an equivalence of triangulated categories $\cx\simeq\cz$. Therefore the bijection restricts to
a bijection between the smashing subcategories which are compactly generated when considered as triangulated categories and the co-smashing subcategories which are also compactly generated when considered as triangulated categories.
\end{proof}

\begin{lemma}\label{lemma:description of tria(S) y thick(S)}
Let $\cd$ be a triangulated category and let $\cs$ be any set of objects of $\cd$. For an object $M$ of $\cd$, the following assertions are equivalent:
\begin{itemize}
\item[1)] $M$ belongs to $\op{tria}_\cd(\cs)$. 
\item[2)] There are triangles in $\cd$
\[
M_{i-1}\ra M_i\ra S_i\ra \Sigma M_{i-1}\ko i=1\ko\dots\ko r
\]
such that
\begin{enumerate}
\item[a)] $M_0=S_0$ and $S_1\ko\dots\ko S_r$ are such that $\Sigma^{n_{i}}S_i\in\cs$, for some  $n_i\in\Z$. 
\item[b)] $M_r=M$.
\end{enumerate}
\end{itemize}
Moreover, $\thick_{\cd}(\cs)$ consists of the direct summands of objects in $\tria_\cd(\cs)$.
\end{lemma}
\begin{proof}
The equivalence of assertions 1) and 2) follows from \cite[Example 1.3.11]{BeilinsonBernsteinDeligne}. The description of the objects in
$\thick_\cd(\cs)$ follows from \cite[Lemma 2.2.2]{BondalVanDenBergh}.
\end{proof}

\begin{corollary} \label{cor.preservation of thickness}
Let $F:\cd\ra\cd'$ be a triangle functor between triangulated categories and let $\cs$ be any set of objects in $\cd$. Then we have inclusions
\[F(\tria_\cd(\cs))\subseteq \tria_{\cd'}(F\cs),
\]
\[ F(\thick_\cd(\cs))\subseteq \thick_{\cd'}(F\cs)
\]
and, when $F$ preserves small coproducts, also 
\[F(\Tria_\cd(\cs))\subseteq \Tria_{\cd'}(F\cs).
\]
\end{corollary}
\begin{proof}
The inclusions for $\tria$ and $\thick$ follow from Lemma~\ref{lemma:description of tria(S) y thick(S)}. For the case of $\Tria$, let $\tilde{\cs}$ be the set of objects $M$ of $\Tria_{\cd}(\cs)$ such that $FM\in\Tria_{\cd'}(F\cs)$. Put $\cc=\tria_{\cd}(\tilde{\cs})$. We have already proved that $F\cc$ is contained in $\Tria_{\cd'}(F\cs)$. It is clear that $\cc$ contains $\cs$ and is contained in $\Tria_{\cd}(\cs)$. Moreover, $\cc$ is closed under small coproducts. Indeed, if $\{M_{i}\}_{i\in I}$ is a small set of objects of $\cc$, then $F(\coprod_{I}M_{i})\cong\coprod_{I}F(M_{i})\in\Tria_{\cd'}(F\cs)$, and so $\coprod_{I}M_{i}\in\tilde{\cs}$ which, of course, implies $\coprod_{I}M_{i}\in\cc$. Therefore, $\cc=\Tria_{\cd}(\cs)$.
\end{proof}

\begin{definition} 
If $\cc_1\ko\dots\ko \cc_{n}\ko \cc$ are categories, then a functor 
\[F: \cc_{1}\times\dots\times\cc_{n}\ra\cc
\] 
is called an {\em $n$-functor}. In case $n=2$ it is called a {\em bifunctor}.
\end{definition}

\begin{definition}\label{def.triangulated n-functor}
Let $(\cd_{1}, \Sigma_{1})\ko\dots\ko (\cd_{n}, \Sigma_{n})\ko (\cd, \Sigma)$ be triangulated categories. A $n$-functor 
\[F: \cd_{1}\times\dots\times\cd_{n}\ra\cd
\] 
is said to be a {\em triangle $n$-functor} if it is a triangle functor in each variable. That is, for each index $i\in\{1,\dots ,n\}$ and each object $(M_1,...,M_{i-1},M_{i+1},...,M_n)$ of $\cd_{1}\times\dots\times\cd_{i-1}\times\cd_{i+1}\times\dots\times\cd_{n}$, there is a morphism of functors 
\[\tau_{i}: F(M_1,...,M_{i-1},?,M_{i+1},...,M_n)\circ \Sigma_{i}\ra\Sigma\circ F(M_1,...,M_{i-1},?,M_{i+1},...,M_n)
\]
such that the pair 
\[ (F(M_1,...,M_{i-1},?,M_{i+1},...,M_n), \tau_{i})
\]
is a triangle functor from $\cd_{i}$ to $\cd$.

If $F$ and $G$ are triangle $n$-functors $\cd_{1}\times\dots\times\cd_{n}\ra\cd$, then a {\em morphism of triangle $n$-functors} 
\[f:F\ra G
\] 
is a morphism of $n$-functors such that, for each index $i\in\{1,\dots ,n\}$ and each object $(M_1,...,M_{i-1},M_{i+1},...,M_n)$ of $\cd_{1}\times\dots\times\cd_{i-1}\times\cd_{i+1}\times\dots\times\cd_{n}$, the induced morphism of functors
\[
f(M_1,...,M_{i-1},?,M_{i+1},...,M_n)
\]
is a morphism of triangle functors. 
\end{definition}

\begin{corollary}\label{cor:point-stable form a thick subcategory}
Let $\cd_{1}\ko\dots\ko\cd_{n}\ko\cd$ be triangulated categories and let
\[F\ko G: \cd_{1}\times\dots\times\cd_{n}\ra\cd 
\] 
be triangle $n$-functors. Let $f:F\ra G$ be a morphisms of triangle $n$-functors. 
\begin{itemize}
\item[1)] For each index $i\in\{1,\dots ,n\}$ and each object $(M_1,...,M_{i-1},M_{i+1},...,M_n)$ of $\cd_{1}\times\dots\times\cd_{i-1}\times\cd_{i+1}\times\dots\times\cd_{n}$, the set $\cc_{(M_1,..,M_{i-1},M_{i+1},..,M_n)}$ of objects $M$ of $\cd_{i}$ for which $f(M_1,...,M_{i-1},M,M_{i+1},...,M_n)$ is an isomorphism is a thick
subcategory of $\cd_i$. 
\item[2)] The set $\cc$ of objects $M$ of $\cd_{i}$ for which the morphism of triangle $(n-1)$-functors
\[
f(?,...,?,M,?,...,?):F(?,...,?,M,?,...,?)\ra G(?,...,?,M,?,...,?)
\]
is an isomorphism, is a thick subcategory of $\cd_i$.
\end{itemize}
\end{corollary}
\begin{proof}
The set of objets  $\cc$ is precisely the intersection of the sets 
\[
\cc_{(M_1,..,M_{i-1},M_{i+1},...,M_n)}
\]
when the $(M_1,...,M_{i-1},M_{i+1},...,M_n)$ varies on 
$\cd_{1}\times\dots\times\cd_{i-1}\times\cd_{i+1}\times\dots\times\cd_{n}$. Since the
intersection of thick subcategories is again a thick subcategory, the proof is reduced to check that if $f:F\ra G$ is a morphism of triangle functors
$\cd'\ra\cd$ between triangulated categories, then the set $\cc'$ of objects $M'$ of $\cd'$ for which $f(M'):F(M')\ra G(M')$ is an isomorphism, is a thick subcategory of $\cd'$. Bearing in mind that the class $\cc'$ is clearly closed under taking direct summands and shifts, the result is a direct
consequence of \cite[Proposition 1.1.20]{Neeman2001}.
\end{proof}

\subsection{(Co)reflective objects}

Let
\[
\xymatrix{
\cm\ar@<1ex>[d]^{R} \\
\cn\ar@<1ex>[u]^{L}
}
\]
be an adjoint pair of functors between arbitrary categories. We denote by $\sigma$ the unit and by $\tau$ the counit.

\begin{definition}\label{def: (co)reflective}
We denote by $\op{Coref}$ the full subcategory of $\cm$ formed by the {\em coreflective objetcs}, \ie those objects $M$ such that the counit $\tau_{M}:LRM\arr{\sim}M$ is an isomorphism. We denote by $\op{Ref}$ the full subcategory of $\cn$ formed by the {\em reflective objects}, \ie those objects $N$ such that the unit $\sigma_{N}:N\arr{\sim} RLN$ is an isomorphism. 
\end{definition}

\begin{lemma}\label{equivalence between reflective and coreflective objects}
\begin{itemize}
\item[1)] The functors $L$ and $R$ induce mutually quasi-inverse equivalences between $\op{Coref}$ and $\op{Ref}$:
\[\xymatrix{\cm\ar@<1ex>[d]^{R} && \op{Coref}\ar@{_(->}[ll]\ar@<1ex>[d]^{R}_{\wr}\\
\cn\ar@<1ex>[u]^{L} && \op{Ref}\ar@{_(->}[ll]\ar@<1ex>[u]^{L}
}
\]
\item[2)] $R$ is fully faithful if and only if $\tau$ is an isomorphism (\ie $\op{Coref}=\cm$).
\item[3)] $L$ is fully faithful if and only if $\sigma$ is an isomorphism (\ie $\op{Ref}=\cn$).
\end{itemize}
\end{lemma}
\begin{proof}
1) Use the well-known equations involving the unit and the counit of the adjunction.

2)-3) See for instance \cite[Lemma 1.2.1]{NicolasTesis}.
\end{proof}

As a consequence of Corollary~\ref{cor:point-stable form a thick subcategory} we have:

\begin{corollary}\label{(co)reflective objects form a thick subcategory}
If $\cm$ and $\cn$ are triangulated categories and $L$ and $R$ are triangle functors, then $\op{Coref}$ and $\op{Ref}$ are thick subcategories.
\end{corollary}

\section{The derived category of a dg category}\label{The derived category of a dg category}

\subsection{Basic definitions}

For the notions of dg category, opposite dg category, right dg module, left dg module and dg bimdule we refer to \cite{KellerDDC,KellerDGC}. Given a dg category $\ca$, the category denoted by $\text{Dif}\ \ca$ in \cite{KellerDDC} is denoted by $\cc_{dg}\ca$ in \cite{KellerDGC}. We will follow the notation of this last article. In particular, given a dg category $\ca$ and an object $A$ of $\ca$, the right dg $\ca$-module $A^\we$ is defined to be the covariant dg functor
\[
A^\we=\Hom_{\ca}(?,A):\ca^\text{op}\ra\cc_{dg}k,
\]
where the base ring $k$ is regarded as a dg category with only one object. Similarly, the left dg $\ca$-module $A^\che$ is defined to be the covariant dg functor
\[
A^\che=\Hom_{\ca}(A,?):\ca\ra\cc_{dg}k.
\]

\subsection{Structured Hom-spaces}

If $\ca$ and $\cb$ are dg categories, $T$ is a dg $\cb$-$\ca$-bimodule and $M$ is a right dg $\ca$-module, then the notation
\[\Hom_{\cc_{dg}\ca}(T,M)
\]
is meaningless, because $T$ is not an object of $\cc_{dg}\ca$. Nevertheless, we use this kind of notation throughout the article, and we will explain here what we mean by this.

\begin{lemma}\label{structures of modules on Hom}
Let $\ca$ and $\cb$ be dg categories, $T$ a dg $\cb$-$\ca$-bimodule, $N$ a left dg $\cb$-module and $M$ a right dg $\ca$-module. Then:
\begin{enumerate}
\item We write $\Hom_{\cc_{dg}\ca}(T,M)$ to denote the following right dg $\cb$-module:
\[\cb^\text{op}\ra\cc_{dg}k\ko B\mapsto \Hom_{\cc_{dg}\ca}(T(?,B),M).
\]
It induces a covariant dg functor
\[\Hom_{\cc_{dg}\ca}(T,?): \cc_{dg}\ca\ra\cc_{dg}\cb.
\]
\item We write $\Hom_{\cc_{dg}\ca}(M,T)$ to denote the following left dg $\cb$-module:
\[\cb\ra\cc_{dg}k\ko B\mapsto \Hom_{\cc_{dg}\ca}(M,T(?,B)).
\] 
It induces a contravariant dg functor
\[\Hom_{\cc_{dg}\ca}(?,T): \cc_{dg}\ca\ra\cc_{dg}(\cb^\text{op}).
\]
\item We write $\Hom_{\cc_{dg}(\cb^\text{op})}(N,T)$ to denote the following right dg $\ca$-module:
\[\ca^\text{op}\ra\cc_{dg}k\ko A\mapsto \Hom_{\cc_{dg}(\cb^\text{op})}(N,T(A,?)).
\] 
It induces a contravariant dg functor
\[\Hom_{\cc_{dg}(\cb^\text{op})}(?,T): \cc_{dg}(\cb^\text{op})\ra \cc_{dg}\ca.
\]
\item We write $\Hom_{\cc_{dg}(\cb^\text{op})}(T,N)$ to denote the left dg $\ca$-module:
\[\ca\ra\cc_{dg}k\ko A\mapsto\Hom_{\cc_{dg}(\cb^\text{op})}(T(A,?),N).
\] 
It induces a covariant dg functor
\[\Hom_{\cc_{dg}(\cb^\text{op})}(T,?): \cc_{dg}(\cb^\text{op})\ra\cc_{dg}(\ca^\text{op}).
\]
\end{enumerate} 
\end{lemma}

\begin{remark}
It is important to notice that for each $A\in\ca$ we have an isomorphism $\Hom_{\cc_{dg}\ca}(A^\we,T)\cong T(A,?)$ in $\cc_{dg}(\cb^\text{op})$ and for each $B\in\cb$ we have an isomorphism $\Hom_{\cc_{dg}(\cb^\text{op})}(B^\che,T)\cong T(?,B)$ in $\cc_{dg}\ca$.
\end{remark}

\subsection{Triangulated categories associated to a dg category}

Let $\ca$ be a dg category, and consider its associated dg category $\cc_{dg}\ca$ of right dg $\ca$-modules. 

\subsubsection{The category up to homotopy}

Its corresponding category $Z^0(\cc_{dg}\ca)$ of $0$th-cocycles will be denoted by $\cc \ca$. The objects are again the right dg $A$-modules, and the morphisms are morphisms $f$ of $\cc_{dg}\ca$ homogeneous of degree $0$ and compatible with the differentials.  It turns out that $\cc\ca$ is a Frobenious category \cite{Keller1996}, where the conflations are given by short exact sequences
\[0\ra L\arr{f}M\arr{g} N\ra 0 
\]
of right dg $\ca$-modules, where $f$ is a section in $\cg\ca$ and $g$ is a retraction in $\cg\ca$, and $\cg\ca$ is the graded category of  graded right modules over  $\ca$ (see in \cite{KellerDDC} further information). In this situation, $f$ is said to be an {\em inflation} and $g$ is said to be a {\em deflation}.

The corresponding projective-injective modules are precisely the contractible ones. Hence, the associated stable category $\ul{\cc\ca}$, which is triangulated, is the quotient of the additive category $\cc\ca$ by the two-sided ideal formed by those morphisms factoring through a contractible module, or, in other words, those morphisms which are null-homotopic. This triangulated category will be denoted by $\ch\ca$, and it is said to be the {\em category of right dg $\ca$-modules up to homotopy}.

\subsubsection{The unbounded derived category} 

The full subcategory $\cn$ of $\ch\ca$ formed by those modules which are acyclic is a triangulated subcategory. The corresponding triangle quotient (see \cite{Verdier1996, Neeman2001}), $\ch\ca/\cn$, is denoted by $\cd\ca$, and it is said to be the {\em unbounded derived category of right dg $\ca$-modules}. 

\subsubsection{Unbounded resolutions}\label{Unbounded resolutions}

Let $M$ be a right dg $\ca$-module such that 
\[
\Hom_{\ch\ca}(M,N)=0
\]
for any acyclic $N$. In this case $M$ is said to be {\em $\ch$-projective} or {\em homotopically projective}. The full triangulated subcategory of $\ch\ca$ formed by the $\ch$-projective modules is denoted by $\ch_{\bf p}\ca$.

Similarly, let $M$ be a right dg $\ca$-module such that 
\[
\Hom_{\ch\ca}(N,M)=0
\] 
for any acyclic $N$. In this case $M$ is said to be {\em $\ch$-injective} or {\em homotopically injective}. The full triangulated subcategory of $\ch\ca$ formed by the $\ch$-injective modules is denoted by $\ch_{\bf i}\ca$.

For any module $M\in\ch\ca$ there are triangles (unique, up to a unique isomorphism extending the indentity in $M$),
\[{\bf p}_{\ca}M\arr{\pi} M\ra {\bf a}_{\ca}M\ra \Sigma{\bf p}_{\ca}M
\]
and
\[{\bf a'}_{\ca}M\ra M\arr{\iota}{\bf i}_{\ca}M\ra\Sigma{\bf a'}_{\ca}M,
\]
in $\ch\ca$ such that ${\bf p}_{\ca}M$ is $\ch$-projective (is said to be the {\em $\ch$-projective resolution} of $M$), ${\bf a}_{\ca}M$ and ${\bf a'}_{\ca}M$ are acyclic and ${\bf i}_{\ca}M$ is $\ch$-injective (is said to be the {\em $\ch$-injective resolution} of $M$). 

\begin{remark}\label{HprojHinjQsi}
We can use the long exact sequence of homology to prove that both $\pi$ and $\iota$ are quasi-isomorphisms.
\end{remark}

The map $M\mapsto{\bf p}_{\ca}M$ underlies a functor 
\[
{\bf p}_{\ca}:\ch\ca\ra\ch_{\bf p}\ca
\] 
which is right adjoint to the inclusion. This functor factors through the quotient $q:\ch\ca\ra\cd\ca$ to give a triangle equivalence
\[\xymatrix{
\ch\ca\ar[r]^{q}\ar[d]_{{\bf p}_{\ca}}&\cd\ca\ar[dl]_{\sim}^{{\bf p}_{\ca}} \\
\ch_{\bf p}\ca &
}
\]
Similarly, the map $M\mapsto{\bf i}_{\ca}M$ underlies a functor 
\[
{\bf i}_{\ca}:\ch\ca\ra\ch_{\bf i}\ca
\] 
which is left adjoint to the inclusion. This functor factors through the quotient $q:\ch\ca\ra\cd\ca$ to give a triangle equivalence
\[\xymatrix{
\ch\ca\ar[r]^{q}\ar[d]_{{\bf i}_{\ca}}&\cd\ca\ar[dl]_{\sim}^{{\bf i}_{\ca}} \\
\ch_{\bf i}\ca &
}
\]

As proved in \cite{KellerDDC}, any $\ch$-projective module in $\ch\ca$ is isomorphic to a module $P$ which admits a filtration
\[0=P_{-1}\ra P_{0}\ra P_{1}\ra\dots\ra P_{n-1}\ra P_{n}\ra\dots\ra P
\]
such that
\begin{itemize}
\item[a)] Each arrow $P_{n-1}\ra P_{n}$ is an inflation in $\cc\ca$.
\item[b)] Each quotient $P_{n}/P_{n-1}$ is a small coproduct of shifts of copies of modules of the form $A^\we\ko A\in\ca$.
\item[c)] $P$ is the colimit in $\cc\ca$ of this filtration.
\end{itemize}

\subsubsection{The perfect derived category}

The full subcategory $\per\ca$ of $\cd\ca$ formed by those modules which are compact objects in $\cd\ca$ is said to be the {\em perfect derived category of right dg $\ca$-modules}.

It is not difficult to prove that a module $M$ is compact if and only if its homotopically projective resolution ${\bf p}_{\ca}M$ can be taken to be a direct summand of a module $P$ admiting a finite filtration
\[0=P_{-1}\ra P_{0}\ra P_{1}\ra\dots\ra P_{n-1}\ra P_{n}=P
\]
such that
\begin{itemize}
\item[a)] Each arrow $P_{n-1}\ra P_{n}$ is an inflation in $\cc\ca$.
\item[b)] Each quotient $P_{n}/P_{n-1}$ is a finite coproduct of shifts of copies of modules of the form $A^\we\ko A\in\ca$.
\end{itemize}

Thus $\per\ca=\op{thick}_{\cd\ca}(\{A^\we\}_{A\in\ca})$.

\subsection{Derived functors}

Let $\ca$ and $\cb$ be dg categories, and let $T$ be a dg $\cb$-$\ca$-bimodule. As said in Lemma~\ref{structures of modules on Hom}, we have two covariant dg functors,
\[\Hom_{\cc_{dg}\ca}(T,?):\cc_{dg}\ca\ra\cc_{dg}\cb
\]
and
\[\Hom_{\cc_{dg}(\cb^\text{op})}(T,?):\cc_{dg}(\cb^\text{op})\ra\cc_{dg}(\ca^\text{op}),
\]
and two contravariant dg functors,
\[\Hom_{\cc_{dg}\ca}(?,T):\cc_{dg}\ca\ra\cc_{dg}(\cb^\text{op})
\]
and
\[\Hom_{\cc_{dg}(\cb^\text{op})}(?,T):\cc_{dg}(\cb^\text{op})\ra\cc_{dg}\ca.
\]
We can describe the triangles in the categories up to homotopy in terms of mapping cones. This is useful to check that the former functors induce triangle functors between the corresponding categories up to homotopy:
\[\Hom_{\cc_{dg}\ca}(T,?):\ch\ca\ra\ch\cb,
\]
\[\Hom_{\cc_{dg}(\cb^\text{op})}(T,?):\ch(\cb^\text{op})\ra\ch(\ca^\text{op}),
\]
\[\Hom_{\cc_{dg}\ca}(?,T):\ch\ca\ra\ch(\cb^\text{op}),
\]
and
\[\Hom_{\cc_{dg}(\cb^\text{op})}(?,T):\ch(\cb^\text{op})\ra\ch\ca.
\]
If $T$ is arbitrary, then these functor need not preserve acyclic objects, and so they do not induce triangle functors between the corresponding derived categories. Nevertheless, we can use $\ch$-injective and $\ch$-projective resolution to construct `approximation' to those induced functors. Namely, we can consider the following compositions:
\[\xymatrix{
\RHom_{\ca}(T,?):\cd\ca\ar[r]^{\hspace{1cm}{\bf i}_{\ca}} & \ch\ca\ar[rr]^{\Hom_{\cc_{dg}\ca}(T,?)} && \ch\cb\ar[r]^{q} & \cd\cb,
}
\]
\[\xymatrix{
\RHom_{\cb^\text{op}}(T,?):\cd(\cb^\text{op})\ar[r]^{\hspace{1cm}{\bf i}_{\cb^\text{op}}} &\ch(\cb^\text{op})\ar[rr]^{\Hom_{\cc_{dg}(\cb^\text{op})}(T,?)} && \ch(\ca^\text{op})\ar[r]^{q} & \cd(\ca^\text{op}),
}
\]
\[\xymatrix{
\RHom_{\ca}(?,T):\cd\ca\ar[r]^{\hspace{1cm}{\bf p}_{\ca}}&\ch\ca\ar[rr]^{\Hom_{\cc_{dg}\ca}(?,T)} && \ch(\cb^\text{op})\ar[r]^{q} &\cd(\cb^\text{op}),
}
\]
and
\[\xymatrix{
\RHom_{\cb^\text{op}}(?,T):\cd(\cb^\text{op})\ar[r]^{\hspace{1cm}{\bf p}_{\cb^\text{op}}}&\ch(\cb^\text{op})\ar[rr]^{\Hom_{\cc_{dg}(\cb^\text{op})}(?,T)} && \ch(\ca)\ar[r]^{q} & \cd\ca.
}
\]
We use the notation `$\RHom$' to express that these functors between derived categories are the {\em right derived} version of the corresponding functors between the categories up to homotopy (see for example \cite{Keller1996} or \cite{Deligne}). 

Starting with the bimodule $T$, we can also define another two dg covariant functors (see \cite{KellerDDC}):
\[?\otimes_{\cb}T: \cc_{dg}\cb\ra\cc_{dg}\ca
\]
and
\[T\otimes_{\ca}?:\cc_{dg}(\ca^\text{op})\ra\cc_{dg}(\cb^\text{op}).
\]

\begin{remark}\label{tensoring with free modules}
It is important to notice that for each $B\in\cb$ we have an isomorphism $B^\we\otimes_{\cb}T\cong T(?,B)$ in $\cc_{dg}\ca$ and for each $A\in\ca$ we have an isomorphism $T\otimes_{\ca}A^\che\cong T(A,?)$ in $\cc_{dg}(\cb^\text{op})$.
\end{remark}

One can check that they preserve conflations, and so they induce triangle functors between the corresponding categories up to homotopy:
\[?\otimes_{\cb}T: \ch\cb\ra\ch\ca
\]
and
\[T\otimes_{\ca}?:\ch(\ca^\text{op})\ra\ch(\cb^\text{op}).
\]
Again, these functors need not to preserve acyclic objects, but still we can use $\ch$-projective resolutions to define functors between the corresponding derived categories,
\[\xymatrix{
?\otimes^\L_{\cb}T: \cd\cb\ar[r]^{\hspace{1cm}{\bf p}_{\cb}} & \ch\cb\ar[rr]^{?\otimes_{\cb}T} &&\ch\ca\ar[r]^{q} & \cd\ca
}
\]
and
\[\xymatrix{
T\otimes^\L_{\ca}?:\cd(\ca^\text{op})\ar[r]^{\hspace{1cm}{\bf p}_{\ca^\text{op}}} & \ch(\ca^\text{op})\ar[rr]^{T\otimes_{\ca}?} &&\ch(\cb^\text{op})\ar[r]^{q} & \cd(\cb^\text{op}).
}
\]
We use the symbol `$\L$' to express the fact that these functors are the {\em left derived} version of the corresponding functors between the categories up to homotopy (see for example \cite{Keller1996} or \cite{Deligne}).

\begin{remark}
Derived functors can be also defined by a universal property as explained e.g. in \cite[\S\S\ 8.4-8.5]{Hirschhorn}. However, we shall make use of the particular construction of the derived functors, and this is why we present them as compositions in which localizations and resolutions are involved. 
\end{remark}

\section{Adjunctions}\label{Adjunctions}

These six functors between derived categories are organized in three couples of adjoint functors:
\[\xymatrix{\cd\ca\ar@<1ex>[d]^{\RHom_{\ca}(T,?)} \\
\cd\cb,\ar@<1ex>[u]^{?\otimes^\L_{\cb}T}
}
\hspace{1cm}
\xymatrix{\cd(\cb^\text{op})\ar@<1ex>[d]^{\RHom_{\cb^\text{op}}(T,?)} \\
\cd(\ca^\text{op}),\ar@<1ex>[u]^{T\otimes^\L_{\ca}?}
}
\hspace{1cm}
\xymatrix{(\cd\ca)^{\text{op}}\ar@<1ex>[d]^{\RHom_{\ca}(?,T)} \\
\cd(\cb^\text{op}).\ar@<1ex>[u]^{\RHom_{\cb^\text{op}}(?,T)}
}
\]

\begin{lemma}\label{rem.relations between RHom and tensor}
Let $\ca$ and $\cb$ be small dg categories and let $T$ be a dg $\cb$-$\ca$-bimodule. Then:
\begin{itemize}
\item[1)] For all $B\ko B'\in\cb$ there is an isomorphism 
\[
\RHom_\ca(T(?,B),T)(B')\arr{\sim}\RHom_\ca(T,T(?,B'))(B)
\]
in $\cd k$.  
\item[2)] For all $A\ko A'\in\ca$ there is an isomorphism
\[
\RHom_{\cb^\text{op}}(T(A,?),T)(A')\arr{\sim}\RHom_{\cb^\text{op}}(T,T(A',?))(A)
\]
in $\cd k$. 
\item[3)] For each $A\in\ca$ we have isomorphisms  
\[
\RHom_\ca(A^\wedge ,T)\cong T(A,?)\cong T\otimes_\ca^\L A^\wedge
\] 
in $\cd(\cb^\text{op})$.
\item[4)] For each $B\in\cb$, we have isomorphisms 
\[
\RHom_{\cb^\text{op}}(B^\vee ,T)\cong T(?,B)\cong B^\vee\otimes_\cb^\L T
\]
in $\cd\ca$.
\end{itemize}
\end{lemma}
\begin{proof}
Statements 3) and 4) are well known (see \cite{KellerDDC}).  We just prove 1) since 2) is entirely similar. By definition, we have
\[
\RHom_{\ca}(T(?,B),T)(B')=\Hom_{\cc_{dg}\ca}({\bf p}_{\ca}T(?,B),T(?,B'))
\]
and
\[
\RHom_{\ca}(T,T(?,B'))(B)=\Hom_{\cc_{dg}\ca}(T(?,B),{\bf i}_{\ca}T(?,B')).
\]
But we have quasi-isomorphisms of dg $k$-modules
\[
\Hom_{\cc_{dg}\ca}({\bf p}_{\ca}T(?,B),T(?,B'))\ra \Hom_{\cc_{dg}\ca}({\bf p}_{\ca}T(?,B),{\bf i}_{\ca}T(?,B'))
\]
and
\[
\Hom_{\cc_{dg}\ca}(T(?,B),{\bf i}_{\ca}T(?,B'))\ra \Hom_{\cc_{dg}\ca}({\bf p}_{\ca}T(?,B),{\bf i}_{\ca}T(?,B')).
\]
\end{proof}

\begin{notation}
Consider the adjunctions at the beginning of this section. We denote by: 
\begin{itemize}
\item $\lambda(N):N\ra\RHom_{\ca}(T,N\otimes^\L_{\cb}T)$ the unit of the first adjunction.
\item $\varepsilon(M): \RHom_{\ca}(T,M)\otimes_{\cb}^\L T\ra M$ the counit of the first adjunction.
\item $\rho(M):M\ra\RHom_{\cb^\text{op}}(T,T\otimes^\L_{\ca}M)$ the unit of the second adjunction.
\item $\phi(N): T\otimes^\L_{\ca}\RHom_{\cb^\text{op}}(T,N)\ra N$ the counit of the second adjunction.
\item $\sigma(N): N\ra\RHom_{\ca}(\RHom_{\cb^\text{op}}(N,T),T) $ the unit of the third adjuction.
\item $\tau(M):M\ra\RHom_{\cb^\text{op}}(\RHom_{\ca}(M,T),T)$ the counit of the third adjunction (regarded in $\cd\ca$).
\end{itemize}
\end{notation}

\begin{proposition} \label{prop.relation between unit-counit isomorphism}
Consider the following statements:
\begin{itemize}
\item[1)] $\lambda(B^\we)$ is an isomorphism for each $B\in\cb$. 
\item[1')] $\lambda(X)$ is an isomorphism for each $X\in\per(\cb)$. 
\item[2)] $\varepsilon(T(?,B))$ is an isomorphism for each $B\in\cb$. 
\item[2')] $\varepsilon(M)$ is an isomorphism for each $M\in\thick_{\cd\ca}(T(?,B)\ko B\in\cb)$. 
\item[3)] $\sigma (B^\vee)$ is an isomorphism for each $B\in\cb$.
\item[3')] $\sigma (X)$ is an isomorphism for each $X\in\per(\cb^\text{op})$. 
\item[4)] $\tau(T(?,B))$ is an isomorphism for each $B\in\cb$. 
\item[4')] $\tau(M)$ is an isomorphism for each $M\in\thick_{\cd\ca}(T(?,B)\ko B\in\cb)$.
\end{itemize}
There is the following chain of implications:
\[
2)\Leftrightarrow 2')\Leftarrow 1')\Leftrightarrow 1)\Leftrightarrow 3)\Leftrightarrow 3')\Rightarrow 4')\Leftrightarrow 4).
\]
\end{proposition}
\begin{proof}
Applying Corollary~\ref{cor:point-stable form a thick subcategory} with $n=1$, we know that if $f$ is any of the morphisms of functors
$\lambda\ko \varepsilon\ko \sigma$ or $\tau$, then the class of objects $\cc$ for which $f_\cc$ is an isomorphism is a thick subcategory of the
corresponding triangulated category. That gives the equivalence $i)\Leftrightarrow i')$ for each $i=1,2,3,4$ since the perfect derived category is precisely the thick subcategory generated by the representable modules.

$1)\Rightarrow 2)$ Statement 1), together with Lemma~\ref{rem.relations between RHom and tensor}, says that we have an isomorphism 
\[B^\we\cong\RHom_{\ca}(T,T(?,B))
\]
in $\cd\cb$. After applying $?\otimes^\L_{\cb}T$, and taking into account again Lemma~\ref{rem.relations between RHom and tensor}, we get an isomorphism
\[T(?,B)\cong\RHom_{\ca}(T,T(?,B))\otimes^\L_{\cb}T
\]
in $\cd\ca$, which happens to be the inverse of the counit $\varepsilon(T(?,B))$.

$3)\Rightarrow 4)$ Statement 3), together with Lemma~\ref{rem.relations between RHom and tensor}, says that we have an isomorphism 
\[B^\vee\cong\RHom_{\ca}(T(?,B),T)
\]
in $\cd(\cb^\text{op})$. After applying $\RHom_{\cb^\text{op}}(?,T)$, and taking into account again Lemma~\ref{rem.relations between RHom and tensor}, we get an isomorphism
\[T(?,B)\cong\RHom_{\cb^\text{op}}(\RHom_{\ca}(T(?,B),T),T)
\]
in $\cd\ca$, which happens to be the counit $\tau(T(?,B))$.

$1)\Leftrightarrow 3)$ It is easy to check that for each $B\ko B'\in\cb$ we have a commutative square in the category of complexes over $k$,
\[\xymatrix{
\Hom_{\cb}(B',B)\ar[rr]^{\lambda(B^\we)(B')}\ar[d]_{\sigma(B'^\che)(B)} && \Hom_{\cc_{dg}\ca}(T(?,B'),{\bf i}_{\ca}T(?,B))\ar[d]^{\pi^\che} \\
\Hom_{\cc_{dg}\ca}({\bf p}_{\ca}T(?,B'),T(?,B))\ar[rr]^{\iota^\we}&& \Hom_{\cc_{dg}\ca}({\bf p}_{\ca}T(?,B'),{\bf i}_{\ca}T(?,B))
}
\]
where $\iota^{\we}$ and $\pi^\che$ are quasi-isomorphisms. Therefore, $\lambda(B^\we)$ is a quasi-isomorphism for every $B\in\cb$ if and only if $\sigma(B^\che)$ is a quasi-isomorphism for every $B\in\cb$.
\end{proof}

Similarly, we have:

\begin{proposition} \label{prop.relation between unit-counit isomorphism bis}
Consider the following statements:
\begin{itemize}
\item[1)] $\rho(A^\che)$ is an isomorphism for each $A\in\ca$. 
\item[1')] $\rho(X)$ is an isomorphism for each $X\in\per(\ca^{\text{op}})$. 
\item[2)] $\phi(T(A,?))$ is an isomorphism for each $A\in\ca$. 
\item[2')] $\phi(N)$ is an isomorphism for each $N\in\thick_{\cd(\cb^\text{op})}(T(A,?)\ko A\in\ca)$. 
\item[3)] $\tau(A^\we)$ is an isomorphism for each $A\in\ca$. 
\item[3')] $\tau(M)$ is an isomorphism for each $M\in\per(\ca)$.
\item[4)] $\sigma (T(A,?))$ is an isomorphism for each $A\in\ca$.
\item[4')] $\sigma (N)$ is an isomorphism for each $N\in\thick_{\cd(\cb^\text{op})}(T(A,?)\ko A\in\ca)$. 
\end{itemize}
There is the following chain of implications:
\[
2)\Leftrightarrow 2')\Leftarrow 1')\Leftrightarrow 1)\Leftrightarrow 3)\Leftrightarrow 3')\Rightarrow 4')\Leftrightarrow 4).
\]
\end{proposition}

\begin{definition}
\begin{itemize}
\item[1)] The objects of $\cd(\cb^\text{op})$ which are reflective (see Definition~\ref{def: (co)reflective}) with respect to the adjoint pair formed by $\RHom_{\cb^\text{op}}(?,T)$ and $\RHom_{\ca}(?,T)$ are said to be {\em homologically $T$-reflective}.
\item[2)] The objects of $\cd\ca$ which are coreflective (see Definition~\ref{def: (co)reflective}) with respect to the adjoint pair formed by $\RHom_{\cb^\text{op}}(?,T)$ and $\RHom_{\ca}(?,T)$ are said to be {\em homologically $T$-coreflective}. 
\end{itemize}
\end{definition}

Recall (see for example \cite[\S 4 of Chapter 1]{AF}) the following:

\begin{definition}
Let $A$ and $B$ be ordinary algebras and let $M$ be a $B$-$A$-bimodule. Consider the ring morphisms
\[\lambda: B\ra\End_{A}(M)\hspace{1cm}\text{ and }\hspace{1cm}\rho:A\ra\End_{B}(M),
\]
given by multiplying to the left and to the right by elements of $B$ and $A$, respectively. We say that $M$ is {\em faithfully balanced} if these morphisms are isomorphisms.
\end{definition}

\begin{examples}
\begin{itemize}
\item[1)] The regular $A$-$A$-bimodule $A$ is faithfully balanced.
\item[2)] If $k$ is a field and $M$ is a $k$-vector space, then $M$ is a faithfully balanced $\End_{k}(M)$-$k$-bimodule. (See Exercise 4.4 of Chapter 1 of \cite{AF}.)
\item[3)] If $M$ is a right $A$-module and $\op{BiEnd}(M_{A})$ is its ring of biendomorphism (see \cite{AF}), then $M$ is a faithfully balanced $\End(M_{A})$-$\op{BiEnd}(M_{A})$-bimodule.
\end{itemize}
\end{examples}

\begin{definition}
Let $\ca$ and $\cb$ be dg categories. A dg $\cb$-$\ca$ bimodule $T$ is {\em homologically faithfully balanced} if the units of the adjunctions
\[\lambda(B^\we): B^\we\arr{\sim}\RHom_{\ca}(T,T(?,B))\hspace{2cm}\rho(A^\che): A^\che\arr{\sim}\RHom_{\cb^\text{op}}(T,T(A,?)).
\]
are isomorphisms for each $B\in\cb$ and $A\in\ca$.
\end{definition}

A prototypical example of a homologically faithfully balanced bimodule is the following.

\begin{example} \label{good n-tilting}
Let $A$ and $B$ be ordinary algebras (they can be regarded as dg categories with only one object and such that the complex of endomorphisms is concentrated in degree $0$). An ordinary $B$-$A$-bimodule $T$ is homologically faithfully balanced if and only if it is faithfully balanced, and $\Ext_A^n(T,T)=0$ and $\Ext_B^n(T,T)=0$ for all $n>0$. In particular, if $T$ is a good $n$-tilting $A$-module (see Definition~\ref{tilting module}) and $B=\End_{A}(T)$, then $T$ is a homologically faithfully balanced $B$-$A$-bimodule \cite[Proposition 1.4]{BazzoniManteseTonolo}.
\end{example}

\begin{remark}
Propositions~\ref{prop.relation between unit-counit isomorphism} and \ref{prop.relation between unit-counit isomorphism bis} set a link between homologically $T$-(co)reflectivity and homologically faithfully balance of $T$. This is important for the proof of Theorem~\ref{special localization result}.
\end{remark}

\begin{notation}\label{notation: regular bimodule}
Given a small dg category $\cb$, the {\em regular bimodule} over $\cb$ is the dg $\cb$-$\cb$-bimodule 
\[
\cb^\text{op}\otimes\cb\ra\cc_{dg}k\ko (B,B')\mapsto\Hom_{\cb}(B,B').
\] 
Abusing of notation we shall denote it by $\cb(?,?)$ or simply by $\cb$.
\end{notation}

\begin{remark} \label{rem.Hom-tenso for regular bimodule}
It is a straightforward verification that we have isomophisms of covariant triangle functors
$\cd\cb\ra\cd\cb$
\[
\RHom_\cb(\cb,?)\cong \id_{\cd\cb}\cong ?\otimes_\cb^\L\cb.
\]
\end{remark}

\begin{definition}
We say that a dg category $\ca$ is {\em $k$-projective} if for each pair of objects $A'\ko A\in\ca$ the complex $\Hom_{\ca}(A',A)$ is $\ch$-projective in $\ch k$. We say that $\ca$ is {\em $k$-flat} if for each pair of objects $A'\ko A\in\ca$ tensoring with the complex $\Hom_{\ca}(A',A)$,
\[?\otimes_{k}\Hom_{\ca}(A',A): \ch k\ra \ch k,
\]
preserves acyclic complexes.
\end{definition}

\begin{lemma}\label{preserving H-projectivity}
Let $\ca$ and $\cb$ be a small dg categories. 
\begin{itemize}
\item[1)] Assume $\ca$ is $k$-projective. Then for each $A\in\ca$ the `restriction of scalars functor' 
\[
\ch(\ca^\text{op}\otimes\cb)\ra\ch\cb\ko M\mapsto M(?,A)
\]
preserves the property of being $\ch$-projective.
\item[2)] Assume $\ca$ is \em $k$-flat. Then for each $A\in\ca$ the `restriction of scalars functor' 
\[\ch(\ca^\text{op}\otimes\cb)\ra\ch\cb\ko M\mapsto M(?,A)
\]
preserves the property of being $\ch$-injective.
\end{itemize}
\end{lemma}
\begin{proof}
1) The `restriction of scalars functor' is isomorphic to
\[
A^\we\otimes_{\ca}?:\ch(\ca^\text{op}\otimes\cb)\ra\ch\cb,
\]
and it has a right adjoint:
\[
\Hom_{\cc_{dg}k}(A^\we,?):\ch\cb\ra\ch(\ca^\text{op}\otimes\cb).
\]
Let $P$ be an $\ch$-projective module in $\ch(\ca^\text{op}\otimes\cb)$ and let $N$ be an acyclic module in $\ch\cb$. Then
\[\Hom_{\ch\cb}(A^\we\otimes_{\ca}P,N)\cong\Hom_{\ch(\ca^\text{op}\otimes\cb)}(P,\Hom_{\cc_{dg}k}(A^\we,N)).
\]
Therefore, to prove $\Hom_{\ch\cb}(A^\we\otimes_{\ca}P,N)=0$ it suffices to prove that $\Hom_{\cc_{dg}k}(A^\we,N)$ is acyclic. But for each pair of objects $A'\in\ca\ko B\in\cb$ and integer $p\in\Z$ the complex $\Hom_{\cc_{dg}k}(A^\we,N)(B,A')=\Hom_{\cc_{dg}k}(\Hom_{\ca}(A',A),N(B))$ has the following $p$th homology $k$-module
\[H^{p}\Hom_{\cc_{dg}k}(\Hom_{\ca}(A',A),N(B))=\Hom_{\ch k}(\Hom_{\ca}(A',A),\Sigma^{p}N(B))=0.
\]

2) The restriction of scalars functor is isomorphic to
\[
\Hom_{\cc_{dg}(\ca^\text{op})}(A^\che,?):\ch(\ca^\text{op}\otimes\cb)\ra\ch\cb,
\]
which has a left adjoint
\[
A^\che\otimes_{k}?:\ch\cb\ra\ch(\ca^\text{op}\otimes\cb)
\]
Let $I$ be an $\ch$-injective module in $\ch(\ca^\text{op}\otimes\cb)$ and let $N$ be an acyclic module in $\ch\cb$. Then
\[\Hom_{\ch\cb}(N,\Hom_{\cc_{dg}(\ca^\text{op})}(A^\che,I))\cong\Hom_{\ch(\ca^\text{op}\otimes\cb)}(A^\che\otimes_{k}N,I)
\]
Therefore, to prove $\Hom_{\ch\cb}(N,\Hom_{\cc_{dg}(\ca^\text{op})}(A^\che,I))=0$ it suffices to prove that $A^\che\otimes_{k}N$ is an acyclic dg $\ca$-$\cb$-bimodule. But for each $A'\in\ca$ and $B\in\cb$ the complex $(A^\che\otimes_{k}N)(B,A')=A^\che(A')\otimes_{k}N(B)=\Hom_{\ca}(A,A')\otimes_{k}N(B)$ is still acyclic because $N(B)$ is acyclic and $\ca$ is $k$-flat.
\end{proof}

\begin{remark}\label{k-projective implies k-flat}
After Lemma~\ref{preserving acyclics under tensors} below, we know that if a dg category is $k$-projective then it is $k$-flat.
\end{remark}

\begin{lemma}\label{preserving acyclics under tensors}
Let $\cb$ be a small dg category. If $X$ is the colimit in $\cc\cb$ of a sequence
\[0=X_{-1}\ra X_{0}\ra X_{1}\ra\dots\ra X_{n-1}\ra X_{n}\ra\dots
\]
where each arrow $X_{n-1}\ra X_{n}$ is an inflation and each factor $X_{n}/X_{n-1}$ is a small coproduct of shifts of copies of $B^\we\ko B\in\cb$, then the functor
\[X\otimes_{B}?:\cc(\cb^\text{op})\ra\cc k
\]
preserves quasi-isomorphisms.
\end{lemma}
\begin{proof}
It is equivalent to check that $X\otimes_{\cb}?$ preserves acyclic objects. If $Y\in\cc(\cb^\text{op})$ is acyclic, then 
\begin{align}
H^m(X\otimes_{\cb}Y)\cong \nonumber \\
\Hom_{\cd k}(\Sigma^{-m}k, (\op{colim}_{n}X_{n})\otimes_{\cb}Y)\cong \nonumber \\
\Hom_{\cd k}(\Sigma^{-m}k, \op{colim}_{n}(X_{n}\otimes_{\cb}Y))\cong \nonumber \\
\cong\op{colim}_{n}\Hom_{\cd k}(\Sigma^{-m}k, X_{n}\otimes_{\cb}Y)\cong \nonumber \\
\cong\op{colim}_{n}H^m(X_{n}\otimes_{\cb}Y). \nonumber
\end{align}
The reader can find a proof of the third isomorphism in Lemma 6.3 of \cite{NicolasSaorinParametrizing}. Now it suffices to prove that each $X_{n}\otimes_{\cb}Y$ is acyclic. This is done by induction on $n$. If $n=0$, then $X_{0}\otimes_{\cb}Y$ is isomorphic to a small coproduct of shifts of dg $k$-modules of the form $Y(B)$ with $B\in\cb$, and so it is acyclic. If $n>0$, we can consider a conflation
\[X_{n-1}\ra X_{n}\ra X_{n}/X_{n-1}
\]
in $\cc(\cb^\text{op})$. This induces a conflation
\[X_{n-1}\otimes_{\cb}Y\ra X_{n}\otimes_{\cb}Y\ra (X_{n}/X_{n-1})\otimes_{\cb}Y
\]
in $\cc k$, and so a triangle
\[X_{n-1}\otimes_{\cb}Y\ra X_{n}\otimes_{\cb}Y\ra (X_{n}/X_{n-1})\otimes_{\cb}Y\ra \Sigma X_{n-1}\otimes_{\cb}Y
\]
in $\cd k$. By hypothesis of induction both $X_{n-1}\otimes_{\cb}Y$ and $(X_{n}/X_{n-1})\otimes_{\cb}Y$ are acyclic, and then so is $X_{n}\otimes_{\cb}Y$.
\end{proof}

\begin{notation}
Let $\ca\ko \cb$ and $\cc$ be small dg categories. If $T$ is a dg $\cb$-$\ca$-bimodule and $X$ is a dg $\cb$-$\cc$-bimodule, then we will commit an abuse of notation by writting $\Hom_{\cc_{dg}(\cb^\text{op})}(T,X)$ to refer to the dg $\ca$-$\cc$-bimodule defined by
\[
\cc^\text{op}\otimes\ca\ra\cc_{dg}k\ko (C,A)\mapsto\Hom_{\cc_{dg}(\cb^\text{op})}(T(A,?),X(C,?)).
\]
\end{notation}

\begin{lemma} \label{lemma:contravariant Hom and tensor}
Let $\ca\ko \cb$ and $\cc$ be small dg categories. 
\begin{itemize}
\item[1)] There exists a natural transformation between bifunctors from
$\cc_{dg}(\cb^\text{op}\otimes\ca)^\text{op}\times\cc_{dg}(\cb^\text{op}\otimes\cc)$ to $\cc_{dg}(\ca^\text{op}\otimes\cc)$,
which are dg functors on both variables,
\[
\psi: \Hom_{\cc_{dg}(\cb^\text{op})}(?,\cb)\otimes_\cb ? \ra \Hom_{\cc_{dg}(\cb^\text{op})}(?,?).
\]
\item[2)] There exists a natural transformation between bifunctors from
$\cc_{dg}(\ca^\text{op}\otimes\cb)\times\cc_{dg}(\cc^\text{op}\otimes\cb)^\text{op}$ to $\cc_{dg}(\ca^\text{op}\otimes\cc)$,
which are dg functors on both variables,
\[
\psi: ?\otimes_\cb\Hom_{\cc_{dg}\cb}(?,\cb) \ra \Hom_{\cc_{dg}\cb}(?,?).
\]
\end{itemize}
\end{lemma}
\begin{proof}
1) Consider a dg $\cb$-$\ca$-bimodule $T$ and a dg $\cb$-$\cc$-bimodule $X$. By definition of the dg $\ca$-$\cc$-bimodule $\Hom_{\cc_{dg}(\cb^\text{op})}(T,X):\cc^\text{op}\otimes\ca\longrightarrow\cc_{dg}k$, we have that $\Hom_{\cc_{dg}(\cb^\text{op})}(T,X)(C,A)=\Hom_{\cc_{dg}(\cb^\text{op})}(T(A,?),X(C,?))$, for all objects $C\in\cc$ and $A\in\ca$. Similarly, by definition of the tensor product of dg bimodules, the dg $\ca$-$\cc$-bimodule $\Hom_{\cc_{dg}(\cb^\text{op})}(T,\cb)\otimes_\cb X:\cc^\text{op}\otimes\ca\longrightarrow\cc_{dg}k$ takes $(C,A)\ra\Hom_{\cc_{dg}(\cb^\text{op})}(T(A,?),\cb)\otimes_\cb X(C,?)$. We then need to define a morphism of dg $k$-modules 
\[
\psi_{T,X}:\Hom_{\cc_{dg}(\cb^\text{op})}(T(A,?),\cb)\otimes_\cb X(C,?)\longrightarrow\Hom_{\cc_{dg}(\cb^\text{op})}(T(A,?),X(C,?)),
\]
for all $C$ and $A$ as above. For this we need in turn to define, for each object $B\in\cb$, a morphism (of zero degree) of dg $k$-modules 
\[
\psi_{T,X}(B):\Hom_{\cc_{dg}(\cb^\text{op})}(T(A,?),\cb(B,?))\otimes X(C,B)\longrightarrow\Hom_{\cc_{dg}(k)}(T(A,B),X(C,B)).
\] 
Indeed if $f:T(A,?)\longrightarrow\cb(B,?)$ is a graded morphism (of some degree) and $t\in T(A,B)$ and $x\in X(C,B)$ are homogeneous elements, we define $\psi_{T,X}(f\otimes x)(t)=(-1)^{|t| |x|}f_B(t)x$, where $f_B:T(A,B)\longrightarrow\cb(B,B)$ is the evaluation of $f$ at $B$. We leave to the reader the task of checking that, when $B$ varies on the objects of $\cb$, we have a well-defined morphism 
\[
\psi_{T,X}(C,A):\Hom_{\cc_{dg}(\cb^\text{op})}(T(A,?),\cb)\otimes_\cb X(C,?)\longrightarrow\Hom_{\cc_{dg}(\cb^\text{op})}(T(A,?),X(C,?))
\]
in $\cc_{dg}k$,  and then, when  $C$ and $A$ vary on the  objects of  $\cc$ and $\ca$, respectively,   we get a morphism 
\[
\psi_{T,X}:\Hom_{\cc_{dg}(\cb^\text{op})}(T,\cb)\otimes_\cb X\longrightarrow\Hom_{\cc_{dg}(\cb^\text{op})}(T,\cb)
\] 
in $\cc_{dg}(\ca^\text{op}\otimes\cc)$ which is easily seen to be natural on $T$ and $X$.

2) Proved similarly.
\end{proof}

\begin{notation}\label{one-functors}
Let $\ca\ko \cb$ and $\cc$ be dg $k$-categories. Let $Y$ be a dg $\cb$-$\ca$-bimodule. Assume that $\cc$ is $k$-projective. Consider the following functors. 
\begin{itemize}
\item The functor
\[
?\otimes^\L_{\cb}Y:\cd(\cc^\text{op}\otimes\cb)\ra\cd(\cc^\text{op}\otimes\ca)
\]
is defined as follows:
\[
(X\otimes^\L_{\cb}Y)(A,C)=({\bf p}_{\cc^\text{op}\otimes\cb}X)(?,C)\otimes_{\cb}Y(A,?).
\]
\item The functor
\[
\RHom_{\cb^\text{op}}(Y,?):\cd(\cb^\text{op}\otimes\cc)\ra\cd(\ca^\text{op}\otimes\cc)
\]
is defined as follows:
\[
\RHom_{\cb^\text{op}}(Y,Z)(A,C)=\Hom_{\cc_{dg}(\cb^\text{op})}(Y(A,?),({\bf i}_{\cb^\text{op}\otimes\cc}Z)(C,?)).
\]
\item The functor
\[
\RHom_{\cb^\text{op}}(?,Y):\cd(\cb^\text{op}\otimes\cc)^\text{op}\ra\cd(\cc^\text{op}\otimes\ca)
\]
is defined as follows:
\[
\RHom_{\cb^\text{op}}(Z,Y)(C,A)=\Hom_{\cc_{dg}(\cb^\text{op})}(({\bf p}_{\cb^\text{op}\otimes\cc}Z)(C,?),Y(A,?)).
\]
\end{itemize}
\end{notation}

\begin{remark}\label{coherence}
When $\cc=k$ these definitions agree with the usual ones thanks to Lemma~\ref{preserving H-projectivity} and Remark~\ref{k-projective implies k-flat}. Note also that:
\begin{itemize}
\item For each dg $\cc$-$\cb$-bimodule $X$ we have an isomorphism
\[
\left[(?\otimes^\L_{\cb}Y)X\right](?,C)\cong X(?,C)\otimes^\L_{\cb}Y
\]
in $\cd\ca$.
\item For each dg $\cb$-$\cc$-bimodule $Z$ we have 
\begin{itemize}
\item an isomorphism $\left[\RHom_{\cb^\text{op}}(?,Y)Z\right](?,C)\cong\RHom_{\cb^\text{op}}(Z(C,?),Y)$ in $\cd\ca$.
\item an isomorphism $\left[\RHom_{\cb^\text{op}}(Y,?)Z\right](C,?)\cong\RHom_{\cb^\text{op}}(Y,Z(C,?))$ in $\cd(\ca^\text{op})$.
\end{itemize}
\end{itemize}

\end{remark}

\begin{remark}
There are left-right symmetric versions of the above functors (\eg $Y\otimes^\L_{\ca}?$) whose definition is left to the reader. 
\end{remark}

\begin{notation}\label{two-functors}
Let $\ca\ko \cb$ and $\cc$ be dg $k$-categories. Assume both $\ca$ and $\cc$ are $k$-projective. We can define a triangulated $2$-functor
\[
?\otimes^\L_{\cb}?: \cd(\cc^\text{op}\otimes\cb)\times\cd(\cb^\text{op}\otimes\ca)\ra\cd(\cc^\text{op}\otimes\ca)
\]
by
\[
(X,Y)\mapsto X\otimes^\L_{\cb}Y={\bf p}_{\cc^\text{op}\otimes\cb}X\otimes_{\cb}{\bf p}_{\cb^\text{op}\otimes\ca}Y.
\]
We can also define another triangle $2$-functor
\[
\RHom_{\cb^\text{op}}(?,?):\cd(\cb^\text{op}\otimes\cc)^\text{op}\times\cd(\cb^\text{op}\otimes\ca)\ra\cd(\cc^\text{op}\otimes\ca)
\]
by doing
\[
\RHom_{\cb^\text{op}}(Z,Y)=\Hom_{\cc_{dg}(\cb^\text{op})}({\bf p}_{\cb^\text{op}\otimes\cc}Z,{\bf i}_{\cb^\text{op}\otimes\ca}Y).
\]
\end{notation}

The following result shows that Notation~\ref{two-functors} is coherent with Notation~\ref{one-functors}.

\begin{lemma}\label{defining triangulated 2-functors}
Let $\ca\ko \cb$ and $\cc$ be dg $k$-categories. Assume both $\ca$ and $\cc$ are $k$-projective. Then:
\begin{itemize}
\item[1)] Let $X$ be a dg $\cc$-$\cb$-bimodule and $Y$ a $\cb$-$\ca$-bimodule. We have an isomorphism
\[
(?\otimes^\L_{\cb}Y)X\cong X\otimes^\L_{\cb}Y\cong (X\otimes^\L_{\cb}?)Y
\]
in $\cd(\cc^\text{op}\otimes\ca)$.
\item[2)] Let $Z$ be a dg $\cb$-$\cc$-bimodule and $Y$ a $\cb$-$\ca$-bimodule. We have an isomorphism
\[
\RHom_{\cb^\text{op}}(?,Y)Z\cong \RHom_{\cb^\text{op}}(Z,Y)\cong \RHom_{\cb^\text{op}}(Z,?)Y
\]
in $\cd(\cc^\text{op}\otimes\ca)$
\end{itemize}
\end{lemma}

\begin{proposition}\label{relation among dualities}
Let $\ca\ko \cb$ and $\cc$ dg categories. Assume $\cc$ is $k$-projective. Let $Y$ be a dg $\cb$-$\ca$-bimodule. Consider the triangle functors $F$ and $G$ from $\cd(\cb^\text{op}\otimes\cc)^\text{op}$ to $\cd(\cc^\text{op}\otimes\ca)$ given by:
\[
\xymatrix{
F: \cd(\cb^\text{op}\otimes\cc)^\text{op}\ar[rr]^{\hspace{0.5cm}\RHom_{\cb^\text{op}}(?,\cb)} && \cd(\cc^\text{op}\otimes\cb)\ar[rr]^{?\otimes^\L_{\cb}Y} && \cd(\cc^\text{op}\otimes\ca)
}
\]
and
\[
G=\RHom_{\cb^\text{op}}(?,Y).
\]
The following assertions hold:
\begin{itemize}
\item[1)] There is a morphism of functors $\theta_{Y}:F\ra G$ such that $\theta_{Y}(X)$ is an isomorphism whenever $X$ is a dg $\cb$-$\cc$-bimodule such that $X(C,?)\in\per(\cb^\text{op})$ for all $C\in\cc$.
\item[2)] If also $\ca$ is $k$-projective, then the maps $\theta(X,Y)=\theta_{Y}(X)$ define a morphism of triangle $2$-functors
\[
\theta: \RHom_{\cb^\text{op}}(?,\cb)\otimes^\L_{\cb}?\ra\RHom_{\cb^\text{op}}(?,?).
\]
\end{itemize}
\end{proposition}
\begin{proof}
1) Using the map $\psi$ of Lemma~\ref{lemma:contravariant Hom and tensor}, we have a morphism of dg $\cc$-$\ca$-bimodules
\[
\xymatrix{F(X)\ar@{=}[d] \\
{\bf p}_{\cc^\text{op}\otimes\cb}\Hom_{\cc_{dg}(\cb^\text{op})}({\bf p}_{\cb^\text{op}\otimes\cc}X,\cb)\otimes_{\cb}Y\ar[d]^{\pi\otimes\id} \\
\Hom_{\cc_{dg}(\cb^\text{op})}({\bf p}_{\cb^\text{op}\otimes\cc}X,\cb)\otimes_{\cb}Y\ar[d]^{\psi} \\
\Hom_{\cc_{dg}(\cb^\text{op})}({\bf p}_{\cb^\text{op}\otimes\cc}X,Y)\ar@{=}[d] \\
GX,
}
\]
where $\pi$ is an $\ch$-projective resolution. We denote by $\theta_{Y}(X)$ this composition, which defines the desired morphism of triangle functors. Thanks to Remark~\ref{coherence}, we know that $\theta_{Y}(X)$ is an isomorphism if and only if $\theta_{Y}(X(C,?))$ is an isomorphism for each $C\in\cc$. Here $\theta_{Y}(X(C,?))$ is to be understood as the evaluation of the version of $\theta_{Y}$ obtained when $\cc=k$. But for this version we have that
\[
\theta_{Y}(B^\che): F(B^\che)=B^\we\otimes_{\cb}Y\cong Y(?,B)\ra \Hom_{\cc_{dg}(\cb^\text{op})}(B^\che,Y)=G(B^\che)
\]
is the canonical isomorphism. Moreover, the full subcategory of $\cd(\cb^\text{op})$ formed by those objects $X$ such that $\theta_{Y}(X)$ is an isomorphism is closed under shifts, extensions and direct summands. This implies the aimed result.

2) We use the previously defined $\theta_{Y}(X)$ and an $\ch$-injective resolution $Y\ra{\bf i}_{\cb^\text{op}\otimes\ca}Y$.
\end{proof}


\begin{remark}
Last proposition admits a left-right symmetric version whose statement is left to the reader. We shall freely use these two symmetric versions, specially in the case when $\cc=k$ or even $\ca=\cc=k$.
\end{remark}

\begin{notation}\label{B-dual}
Suppose that $\ca$ is a $k$-projective small dg category and that $T$ is a dg $\cb$-$\ca$-bimodule. Then, following Notation~\ref{one-functors}, we define a dg $\ca$-$\cb$-bimodule
\[T^{\ast}=\RHom_{\cb^\text{op}}(T,\cb)
\]
as follows: 
\[T^\ast(B,A)=\Hom_{\cc_{dg}(\cb^\text{op})}(({\bf p}_{\cb^\text{op}\otimes\ca}T)(A,?),\cb(B,?)).
\]
We can then consider
\[T^{\ast\ast}=\RHom_{\cb}(T^\ast,\cb),
\]
which is a dg $\cb$-$\ca$-bimodule as follows:
\[T^{\ast\ast}(A,B)=\Hom_{\cc_{dg}\cb}(({\bf p}_{\ca^\text{op}\otimes\cb}T^\ast)(?,A),\cb(?,B)).
\]
Denote by
\[\sigma: T\ra T^{\ast\ast}
\]
the morphism in $\cd(\cb^\text{op}\otimes\ca)$ using the following chain of morphisms (which involves $\ch$-projective and $\ch$-injective resolutions):
\[
\xymatrix{
T \\
{\bf p}_{\cb^\text{op}\otimes\ca}T\ar[u]_{\wr}\ar[d]^{\text{ev}} \\
\Hom_{\cc_{dg}\cb}(\Hom_{\cc_{dg}(\cb^\text{op})}({\bf p}_{\cb^\text{op}\otimes\ca}T,\cb),\cb)\ar[d] \\
\Hom_{\cc_{dg}\cb}({\bf p}_{\ca^\text{op}\otimes\cb}\Hom_{\cc_{dg}(\cb^\text{op})}({\bf p}_{\cb^\text{op}\otimes\ca}T,\cb),\cb)\ar@{=}[d] \\
T^{\ast\ast},
}
\]
where ${\bf p}_{\cb^\text{op}\otimes\ca}T\ra T$ is a $\ch$-projective resolution in $\ch(\cb^\text{op}\otimes\ca)$, the last map is induced by the $\ch$-projective resolution 
\[
{\bf p}_{\ca^\text{op}\otimes\cb}\Hom_{\cc_{dg}(\cb^\text{op})}({\bf p}_{\cb^\text{op}\otimes\ca}T,\cb)\ra \Hom_{\cc_{dg}(\cb^\text{op})}({\bf p}_{\cb^\text{op}\otimes\ca}T,\cb)
\]
in $\ch(\ca^\text{op}\otimes\cb)$, and the map `ev' is the evaluation map
\[
X\ra\Hom_{\cc_{dg}(\cb)}(\Hom_{\cc_{dg}(\cb^\text{op})}(X,\cb),\cb)\ko x\mapsto (\text{ev}(x): f\mapsto (-1)^{|x|\cdot |f|}f(x)).
\]
\end{notation}

\begin{lemma}\label{B-duality and isos}
Let $\ca$ and $\cb$ be dg categories. Assume $\ca$ is $k$-projective. If $T$ is a dg $\cb$-$\ca$-bimodule such that $T(A,?)\in\per(\cb^\text{op})$ for all $B\in\cb$, then the following assertions hold:
\begin{itemize}
\item[1)] $\sigma:T\ra T^{\ast\ast}$ is an isomorphism in $\cd(\cb^\text{op}\otimes\ca)$.
\item[2)] The triangle functors $?\otimes^\L_{\cb}T$ and $\RHom_{\cb}(T^\ast,?)$ from $\cd\cb$ to $\cd\ca$ are naturtally isomorphic.
\item[3)] The triangle functors $T^\ast\otimes^\L_{\cb}?$ and $\RHom_{\cb^\text{op}}(T,?)$ from $\cd(\cb^\text{op})$ to $\cd(\ca^\text{op})$ are naturtally isomorphic.
\end{itemize}
\end{lemma}
\begin{proof}
We just need to prove assertions 1) and 3), for assertion 2) follows from 3) by symmetry.

1) For each $A\in\ca$ the morphism $\sigma(?,A):T(A,?)\ra T^{\ast\ast}(A,?)$ gets identified with the unit
\[\sigma(T(A,?)):T(A,?)\ra\RHom_{\cb}(\RHom_{\cb^\text{op}}(T(A,?),\cb),\cb)
\]
of the adjunction
\[
\xymatrix{
(\cd\cb)^\text{op}\ar@<1ex>[d]^{\RHom_{\cb}(?,\cb)} \\
\cd(\cb^\text{op}).\ar@<1ex>[u]^{\RHom_{\cb^\text{op}}(?,\cb)}
}
\]
After Proposition~\ref{prop.relation between unit-counit isomorphism}, we know that it is an isomorphism if and only if the unit
\[
\lambda(B^\we):B^\we\ra\RHom_{\cb}(\cb,B^\we\otimes^\L_{\cb}\cb)
\]
of the adjunction
\[\xymatrix{
\cd\cb\ar@<1ex>[d]^{\RHom_{\cb}(\cb,?)} \\
\cd\cb\ar@<1ex>[u]^{?\otimes^\L_{\cb}\cb}
}
\]
is an isomorphism. But this last condition clearly holds.

3) After part 2) of Proposition~\ref{relation among dualities} we have a morphism of triangle functors
\[
\theta(T,?):T^\ast\otimes^\L_{\cb}?\ra\RHom_{\cb^\text{op}}(T,?).
\]
It is an isomorphism thanks to part 1) of Proposition~\ref{relation among dualities}.
\end{proof}

\section{Results for dg categories}\label{Results for dg categories}

\begin{corollary}\label{cor.RHom fully faithful with no extras}
Let $\ca$ and $\cb$ be small dg categories and let $T$ be a dg $\cb$-$\ca$-bimodule. The following assertions are
equivalent:
\begin{itemize}
\item[1)] $\RHom_\ca(T,?):\cd\ca\ra\cd\cb$ is fully faithful. 
\item[2)] The counit map $\delta:\RHom_\ca(T,?)\otimes_\cb^\L T\ra \id_{\cd\ca}$ is an isomorphism. 
\item[3)] $?\otimes_\cb^\L T:\cd\cb\ra\cd\ca$ induces a triangle equivalence
\[
\cd\cb/\ker(?\otimes_\cb^\L T)\arr{\sim}\cd\ca.
\]
\item[4)] The functor 
\[
\RHom_\ca(T,?)\otimes_\cb^\L T:\cd\ca\ra\cd\ca
\] 
preserves coproducts and 
\[
\delta(A^\we):\RHom_\ca(T,A^\we)\otimes_\cb^\L T\ra A^\we
\] 
is an isomorphism in $\cd\ca$, for each $A\in\ca$. 
\end{itemize}
\end{corollary}
\begin{proof}
The equivalence between 1), 2) and 3) follows from Proposition~\ref{GabrielZisman result}. Indeed, one can easily prove that the set $\cs$ of morphisms $u$ of $\cd\cb$ such that $u\otimes^\L_{\cb}T$ is an isomorphism equals the set of morphisms $u$ of $\cd\cb$ whose cone is in the kernel of $?\otimes^\L_{\cb}T$.

Clearly, 2) implies 4).

$4)\Rightarrow 2)$ The hypothesis implies that the full subcategory $\cc$ of $\cd\ca$ formed by those objects $M$ such that
$\delta(M)$ is an isomorphism is a triangulated subcategory of $\cd\ca$, closed under small coproducts and containing all $A^\we\ko A\in\ca$. It follows
that $\cc=\cd\ca$.
\end{proof}

The following remark is related to Proposition 2.6 of \cite{BazzoniManteseTonolo}.

\begin{remark} \label{rem.RHom fully faithful and smashing}
In the situation of Corollary~\ref{cor.RHom fully faithful with no extras}, $\cn=\ker(?\otimes_\cb^\L T)$ is a smashing subcategory
of $\cd\cb$ if, and only if, $\RHom_\ca(T,?)$ preserves coproducts if, and only if, $T(?,B)\in\per(\ca)$ for each $B\in\cb$.
\end{remark}

\begin{theorem} \label{teor.derived tensor fully faithful}
Let $\ca$ and $\cb$ be small dg categories and let $T$ be a dg $\cb$-$\ca$-bimodule. The following assertions are equivalent:
\begin{itemize}
\item[1)] $?\otimes_\cb^\L T:\cd\cb\ra\cd\ca$ is fully faithful. 
\item[2)] The unit map $\lambda:\id_{\cd\cb}\ra\RHom_\ca(T,?\otimes_\cb^\L T)$ is an isomorphism. 
\item[3)] $\lambda (B^\we)$ is an isomorphism for each $B\in\cb$, and the composition of the functors
\[
\xymatrix{
\cd\cb\ar[rr]^{?\otimes_\cb^\L T} &&\cd\ca\ar[rr]^{\RHom_\ca(T,?)} && \cd\cb
}
\] 
is a functor preserving small coproducts. 
\item[4)] $\lambda (B^\we)$ is an isomorphism for each $B\in\cb$, and the restriction of $\RHom_\ca(T,?)$ to $\Tria_{\cd\ca}(T(?,B)\ko B\in\cb)$ preserves small coproducts (and is fully faithful).
\item[5)] $\lambda (B^\we)$ is an isomorphism and $T(?,B)$ is a compact object of the category $\Tria_{\cd\ca}(T(?,B')\ko B'\in\cb)$,
for each $B\in\cb$.
\end{itemize}
It this case $?\otimes_\cb^\L T$ induces a triangle equivalence
\[
\cd\cb\arr{\sim}\op{im}(?\otimes_\cb^\L T)=\Tria_{\cd\ca}(T(?,B)\ko B\in\cb),
\] 
and $(\op{im}(?\otimes_\cb^\L T),\ker(\RHom_\ca(T,?)))$ is a semiortogonal decomposition of $\cd\ca$.
\end{theorem}
\begin{proof}
$1)\Leftrightarrow 2)$ follows from well-known properties of adjunctions (use the dual of Proposition~\ref{GabrielZisman result}).

$2)\Rightarrow 3)$ is clear.

$3)\Rightarrow 2)$ The full subcategory $\cc$ of $\cd\cb$ consisting of the objects $N$ such that $\lambda(N)$ is an isomorphism, is a triangulated subcategory closed under small coproducts and containing all the $B^\wedge\ko B\in\cb$. It follows that $\cc=\cd\cb$.

$1),3)\Rightarrow 4)$ By Corollary~\ref{cor.preservation of thickness} we have that $\op{im}(?\otimes^\L_{\cb}T)$ is contained in the category $\Tria_{\cd\ca}(T(?,B)\ko B\in\cb)$. But, since $?\otimes^\L_{\cb}T$ is fully faithful, it follows that $\op{im}(?\otimes^\L_{\cb}T)$ is a full triangulated subcategory of $\cd\ca$. Moreover, it containes the objects of the form $T(?,B)\ko B\in\cb$, and it is closed under coproducts. Therefore, we have
\[\op{im}(?\otimes^\L_{\cb}T)=\Tria_{\cd\ca}(T(?,B)\ko B\in\cb).
\]
Now the fact that $\RHom_\ca(T,?\otimes_\cb^\L T)$ preserves small coproducts implies that the restriction of $\RHom_\ca(T,?)$ to
$\op{im}(?\otimes_\cb^\L T)=\Tria_{\cd\ca}(T(?,B)\ko B\in\cb)$ preserves small coproducts. Moreover, this restriction is quasi-inverse of the equivalence
\[
?\otimes^\L_{\cb}T:\cd\cb\arr{\sim}\Tria_{\cd\ca}(T(?,B)\ko B\in\cb).
\]

$4)\Leftrightarrow 5)$ For each family $\{M_{i}\}_{i\in I}$ of objects of $\Tria_{\cd\ca}(T(?,B)\ko B\in\cb)$ and each $B\in\cb$ we have a commutative diagram
\[
\xymatrix{\coprod_{i}\RHom_{\ca}(T(?,B),M_{i})\ar[r]\ar@{=}[d] & \RHom_{\ca}(T(?,B),\coprod_{i}M_{i})\ar@{=}[d] \\
\coprod_{i}\RHom_{\ca}(T,M_{i})(B)\ar[r] & \RHom_{\ca}(T,\coprod_{i}M_{i})(B).
}
\]
Then the upper horizontal arrow is an isomorphism if and only if so is the bottom one.

$4)\Rightarrow 3)$ The functor
\[
\RHom_\ca(T,?\otimes_\cb^\L T):\cd\ca\ra\cd\ca
\]
is the composition of the following  functors
\[
\xymatrix{
\cd\ca\ar[rr]^{?\otimes_\cb^\L T\hspace{1cm}} && \Tria (T(?,B)\ko B\in\cb)\ar[rr]^{\hspace{1cm}\RHom_\ca(T,?)} &&\cd\ca. 
}
\]
The hypothesis guarantees that the second one in this composition preserves small coproducts and, hence, also the composition does.

We finally prove that $(\op{im}(?\otimes_\cb^\L T),\ker(\RHom_\ca(T,?)))$ is a semi-orthogonal decomposition of $\cd\ca$. Since we know that the inclusion
\[
\op{im}(?\otimes_\cb^\L T)\ra\cd\ca
\]
has a right adjoint, we
only need to prove the equality 
\[
\op{im}(?\otimes_\cb^\L T)^\perp =\ker(\RHom_\ca(T,?)),
\] 
which is easily checked. 
\end{proof}

\begin{theorem}\label{special localization result}
Let $\cb$ and $\ca$ be dg $k$-categories and let $T$ be a dg $\cb$-$\ca$-bimodule. Consider the following statements:
\begin{itemize}
\item[1)] $T$ is homologically faithfully balanced and for each $A\in\ca$ we have $T(A,?)\in\per(\cb^\text{op})$.
\item[2)] $T$ is homologically faithfully balanced and each $A^\we$ belongs to the category $\op{thick}_{\cd\ca}(T(?,B)\ko B\in\cb)$.
\item[3)] $\RHom_{\ca}(T,?):\cd\ca\ra\cd\cb$ is fully faithful, preserves compact objects and each $B^\we\ko B\in\cb$, is in its essential image.
\item[4)] For each $B\in\cb$ the unit $\lambda(B^\we):B^\we\ra\op{RHom}_{\ca}(T,T(?,B))$ is an isomorphism and each $A^\we$ belongs to $\thick_{\cd\ca}(T(?,B)\ko B\in\cb)$.
\item[5)] For each $B\in\cb$ the unit $\lambda(B^\we):B^\we\ra\RHom_\ca(T,T(?,B))$ is an isomorphism and $T\otimes^\L_\ca?:\cd(\ca^\text{op})\ra\cd(\cb^\text{op})$ is a fully faithful functor whose right adjoint preserves small coproducts.
\item[6)] For each $B\in\cb$ the unit $\lambda(B^\we):B^\we\ra\RHom_\ca(T,T(?,B))$ is an isomorphism and $?\otimes^\L_{\cb}T$ has a fully faithful left adjoint.
\item[7)] For each $B\in\cb$ the unit $\lambda(B^\we):B^\we\ra\RHom_\ca(T,T(?,B))$ is an isomorphism and there exists a dg $\ca$-$\cb$-bimodule $T'$ such that the couples $(?\otimes^\L_{\ca}T',?\otimes^\L_{\cb}T)$ and $(T\otimes^\L_{\ca}?,T'\otimes^\L_{\cb}?)$ are adjoint pairs with fully faithful left components.
\end{itemize}
Then assertions $1)$ - $6)$ are equivalent and implied by $7)$. Moreover, if $\ca$ is $k$-projective, then all the assertions $1)$ - $7)$ are equivalent.
\end{theorem}
\begin{proof}
$1)\Rightarrow 2)$ The fact that $T$ is homologically faithfully balanced implies that $\rho(A^\vee)$ is an isomorphism in $\cd(\ca^\text{op})$, for each
$A\in\ca$. By Proposition~\ref{prop.relation between unit-counit isomorphism bis}, it follows that
the map 
\[
\tau(A^\we):A^\we\ra\RHom_{\cb^\text{op}}(\RHom_\ca(A^\we,T),T)\cong\RHom_{\cb^\text{op}}(T(A,?),T)
\] 
is an isomorphism in $\cd\ca$, for each $A\in\ca$. The fact that $T(A,?)\in\per(\cb^\text{op})$ implies (see Corollary~\ref{cor.preservation of thickness}) that $A^\we$ is in $\thick_{\cd\ca}(\RHom_{\cb^\text{op}}(B^\vee,T)\ko B\in\cb)=\thick_{\cd\ca}(T(?,B)\ko B\in\cb)$, for each $A\in\ca$.

$2)\Rightarrow 1)$ The fact that $T$ is homologically faithfully balanced implies that $\lambda (B^\we)$ is an isomorphism in
$\cd\cb$ for each $B\in\cb$. Using Proposition~\ref{prop.relation between unit-counit isomorphism}, we get that 
\[
\sigma (B^\vee):B^\vee\ra\RHom_\ca(\RHom_{\cb^\text{op}}(B^\vee,T),T)=\RHom_\ca(T(?,B),T)
\] 
is an isomorphism in $\cd(\cb^\text{op})$, for each $B\in\cb$. But the hypothesis together with Corollary~\ref{cor.preservation of thickness} imply that $T(A,?)=\RHom_\ca(A^\we,T)$ is in the category 
\[
\thick_{\cd(\cb^\text{op})}(\RHom_\ca(T(?,B),T)\ko B\in\cb)=\thick_{\cd(\cb^\text{op})}(B^\vee\ko B\in\cb)=\per(\cb^\text{op}).
\]

$1)\ko 2)\Rightarrow 3)$ We need to prove that the counit map
\[
\delta: \RHom_{\ca}(T,?)\otimes^\L_{\cb}T\ra\id_{\cd\ca}
\]
is an isomorphism. This is equivalent to prove that for each object $A\in\ca$ the morphism
\[
\RHom_{\ca}(T,?)\otimes^\L_{\cb}T(A,?)=(\RHom_{\ca}(T,?)\otimes^\L_{\cb}T)(A)\ra ?(A)
\]
is an isomorphism of triangle functors $\cd\ca\ra\cd k$.

{\em Step 1: for each $A\in\ca$ we have $T(A,?)\cong T(A,?)^{\ast\ast}$ and $T(A,?)^\ast\in\per(\cb)$.} Note that $T(A,?)^\ast$ is the $\cb$-dual of the dg $\cb$-$k$-bimodule $T(A,?)$. That is to say, we are using Notation~\ref{B-dual} with $T$ replaced by $T(A,?)$ and $\ca$ replaced by $k$. Now, to prove the first part of the statement of this step apply part 1) of Lemma~\ref{B-duality and isos} to the dg $\cb$-$k$-bimodule $T(A,?)$. To prove the second part of the statement of this step note that $(B^\che)^\ast\cong B^{\we}\in\per\cb$ for each $B\in\cb$. Now apply Corollary~\ref{cor.preservation of thickness} using the fact that $T(A,?)\in\per(\cb^\text{op})=\thick_{\cd(\cb^\text{op})}(B^\che\ko B\in\cb)$.

{\em Step 2: for each $A\in\ca$ we have $T(A,?)^\ast\otimes^\L_{\cb}T\cong A^\we$.} We apply Proposition~\ref{relation among dualities} with $\cc=k$ and $Y=T$ and get an isomorphism of triangle functors
\[
T(A,?)^\ast\otimes^\L_{\cb}T\cong\RHom_{\cb^\text{op}}(T(A,?),T).
\]
Since $T$ is homologically faithfully balanced, $\rho(A^\we)$ is an isomorphism for all $A\in\ca$. Then, by Proposition~\ref{prop.relation between unit-counit isomorphism bis}, we get that
\[
\tau(A^\we):A^\we\ra\RHom_{\cb^\text{op}}(\RHom_{\ca}(A^\we,T),T)\cong \RHom_{\cb^\text{op}}(T(A,?),T)
\]
is an isomorphism. Thus we have the desired isomorphism
\[\xymatrix{
T(A,?)^\ast\otimes^\L_{\cb}T\ar[r] & \RHom_{\cb^\text{op}}(T(A,?),T)\ar[rr]^{\hspace{1.5cm}\tau(A^\we)^{-1}} && A^\we.
}
\]

{\em Step 3: $\delta$ is an isomorphism.}  Applying now a symmetric version of Proposition~\ref{relation among dualities}, for each $M\in\cd\ca$ we get an isomorphism in $\cd k$
\[
\xymatrix{
\RHom_{\ca}(T,M)\otimes^\L_{\cb}T(A,?)\ar[d]^\wr \\
\RHom_{\ca}(T,M)\otimes^\L_{\cb}T(A,?)^{\ast\ast}\ar[d]^\wr \\
\RHom_{\cb}(T(A,?)^\ast,\RHom_{\ca}(T,M))\ar[d]^\wr \\
\RHom_{\ca}(T(A,?)^\ast\otimes^\L_{\cb}T,M).
}
\]
Now, using step 2 we get an isomorphism in $\cd k$
\[
\RHom_{\ca}(T,M)\otimes^\L_{\cb}T(A,?)\cong\RHom_{\ca}(T(A,?)^\ast\otimes^\L_{\cb}T,M)\cong \RHom_{\ca}(A^\che,M)\cong M(A).
\]
It is routine to check that it is just the one given by the counit map $\delta(M)$.

{\em Step 4: each $B^\we$ is in the essential image of $\op{im}(\RHom_\ca(T,?))$.} This is a direct consequence of the fact that $T$ is homologically faithfully balanced since
\[
\lambda(B^\we):B^\we\ra\RHom_{\ca}(T,B^\we\otimes^\L_{\cb}T)=\RHom_{\ca}(T,T(?,B))
\]
is an isomorphism.

{\em Step 5: $\RHom_\ca(T,?)$ preserves compact objects.}  By hypothesis we have that
\[
A^\we\in\thick_{\cd\ca}(T(?,B)\ko B\in\cb).
\]
It follows that $\RHom_{\ca}(T,A^\we)$ belongs to
\[
\thick_{\cd\cb}(\RHom_{\ca}(T,T(?,B))\ko B\in\cb)=\thick_{\cd\cb}(B^\we\ko B\in\cb)=\per\cb,
\]
for each $A\in\ca$. By Corollary~\ref{cor.preservation of thickness}, we deduce that $\RHom_{\ca}(T,M)$ is compact in $\cd\cb$ whenever $M$ is compact in $\cd\ca$.

$3)\Rightarrow 4)$ By well-known properties of adjunctions (see Lemma~\ref{equivalence between reflective and coreflective objects}), the essential image of the functor $\RHom_{\ca}(T,?):\cd\ca\ra\cd\cb$ is the full subcategory of $\cd\cb$ formed by those modules $N$ for which the unit
\[N\arr{\sim}\RHom_{\ca}(T,N\otimes^\L_{\cb}T)
\]
is an isomorphism. In particular, we have
\[B^\we\arr{\sim}\RHom_{\ca}(T,T(?,B))
\]
in $\cd\cb$. On the other hand, since the restriction of $?\otimes^\L_{\cb}T$ to $\op{im}(\RHom_{\ca}(T,?))$ yields the quasi-inverse of $\RHom_{\ca}(T,?)$ and $\RHom_{\ca}(T,A^\we)\in\per\cb=\thick_{\cd\cb}(B^\we\ko B\in\cb)$, then
\[A^\we\cong\RHom_{\ca}(T,A^\we)\otimes^\L_{\cb}T\in\thick_{\cd\ca}(B^\we\otimes^\L_{B}T\ko B\in\cb)=\thick_{\cd\ca}(T(?,B)\ko B\in\cb).
\]

$4)\Rightarrow 2)$ We just need to check that 
\[
\rho(A^\vee):A^\vee\ra\RHom_{\cb^\text{op}}(T,T\otimes_\ca^\L A^\vee)
\] 
is an isomorphism in $\cd(\ca^\text{op})$, for each $A\in\ca$. By Proposition~\ref{prop.relation between unit-counit isomorphism bis}, that is
equivalent to prove that 
\[
\tau (A^\we):A^\we\ra\RHom_{\cb^\text{op}}(\RHom_\ca(A^\we,T),T)
\] 
is an isomorphism in $\cd\ca$, for each $A\in\ca$. For that it is enough to prove that 
\[
\tau(T(?,B)):T(?,B)\ra\RHom_{\cb^\text{op}}(\RHom_\ca(T(?,B),T),T)
\] 
is an isomorphism. But this follows from Proposition~\ref{prop.relation between unit-counit
isomorphism}.

$1)\Rightarrow 5)$ Note that 
\[
\RHom_{\cb^\text{op}}(T,?):\cd(\cb^\text{op})\ra\cd(\ca^\text{op})
\]
preserves small coproducts if, and only if, so does 
\[
\RHom_{\cb^\text{op}}(T(A,?),?):\cd(\cb^\text{op})\ra\cd k,
\]
for each $A\in\ca$. This is in turn equivalent to say that $T(A,?)\in\per(\cb^\text{op})$, for each $A\in\ca$. On the other hand, by the homological faithful balance of $T$, the unit map 
\[
\rho (A^\vee):A^\vee\ra\RHom_{\cb^\text{op}}(T,T\otimes_\ca^\L A^\vee)
\] 
is an isomorphism, for each $A\in\ca$. Then the symmetric of condition 5) in Theorem~\ref{teor.derived tensor fully faithful}
holds. It follows that
\[
T\otimes_\ca^\L?:\cd(\ca^\text{op})\ra\cd(\cb^\text{op})
\]
is fully faithful.

$5)\Rightarrow 1)$  By the fully faithful condition of 
\[
T\otimes_\ca^\L:\cd(\ca^\text{op})\ra\cd(\cb^\text{op}),
\]
we know that the unit map 
\[
\rho(A^\vee):A^\vee\ra\RHom_{\cb^\text{op}}(T,T\otimes_\ca^\L A^\vee)
\] 
is an isomorphism, for each $A\in\ca$. Then  $T$ is homologically faithfully balanced. That $T(A,?)\in\per(\cb^\text{op})$, for each $A\in\ca$ follows from the fact that the functor
\[
\RHom_{\cb^\text{op}}(T,?):\cd(\cb^\text{op})\ra\cd(\ca^\text{op})
\]
preserves small coproducts (see the first paragraph of the proof of the implication $1)\Rightarrow 5)$).

$1)$-$5)\Rightarrow 6)$ We claim that $?\otimes^\L_{\cb}T:\cd\cb\ra\cd\ca$ preserves small products, which, by Lemma~\ref{lemma:existence of left adjoints}, will imply that it has a left adjoint. Then Lemma~\ref{two-sided fully faithful adjoint} will finish the proof.

Indeed, if $\{X_{i}\}_{i\in I}$ is a small family of objects of $\cd\cb$, then we have a canonical morphism 
\[\varphi: (\prod_{i\in I}X_{i})\otimes^\L_{\cb}T\ra\prod_{i\in I}(X_{i}\otimes^\L_{\cb}T)
\]
in $\cd\ca$. This morphism is an isomorphism if and only if for each $A\in\ca$ the cochain map
\[
\xymatrix{
(\prod_{i\in I}X_{i})\otimes^\L_{\cb}T(A,?)\ar[r]^{\varphi(A)} & \prod_{i\in I}(X_{i}\otimes^\L_{\cb}T(A,?))
}
\]
is an isomorphism in $\cd k$. Consider the full subcategory $\ct$ of $\cd(\cb^\text{op})$ formed by those objects $N$ such that the canonical morphism
\[(\prod_{i\in I}X_{i})\otimes^\L_{\cb}N\ra\prod_{i\in I}(X_{i}\otimes^\L_{\cb}N)
\]
is an isomorphism in $\cd k$. Then $\ct$ is closed under shifts, extensions, direct summands, and contains $B^\che$ for each $B\in\cb$. It follows that $\per(\cb^\text{op})$ is contained in $\ct$. By assertion 1) of the theorem, we know that for each $A\in\ca$ we have $T(A,?)\in\per(\cb^\text{op})$, and so $T(A,?)\in\ct$. That is to say, for each $A\in\ca$ the map $\varphi(A)$ is an isomorphism as desired.

$6)\Rightarrow 4)$ Let us denote by $L:\cd\ca\ra\cd\cb$ the left adjoint of $?\otimes^\L_{\cb}T$. It is easy to check that $L$ preserves compact objects. This, together with the fact that the unit of the adjunction $(L,?\otimes^\L_{\cb}T)$ is an isomorphism, implies that each $A^\we$ is in the essential image of $\per\cb$ by the functor $?\otimes^\L_{\cb}T$. After Corollary~\ref{cor.preservation of thickness}, this implies that each $A^\we$ is in $\thick_{\cd\ca}(T(?,B)\ko B\in\cb)$.

$7)\Rightarrow 6)$ Clear.

$1)$-$6)\Rightarrow 7)$ We will take $T'$ to be $T^\ast$ (see Notation~\ref{B-dual}). Since $T(A,?)\in\per(\cb^\text{op})$ for each $A\in\ca$, then $T^\ast\otimes^\L_{\cb}?\cong\RHom_{\cb^\text{op}}(T,?)$ as triangle functors from $\cd(\cb^\text{op})\ra\cd(\ca^\text{op})$ (see Lemma~\ref{B-duality and isos}). Then $(T\otimes^\L_{\ca}?,T^\ast\otimes^\L_{\cb}?)$ is an adjoint pair whose left component is fully faithful. Moreover, we know the map $\sigma:T\ra T^{\ast\ast}$ is an isomorphism in $\cd(\cb^\text{op}\otimes\ca)$ (see Lemma~\ref{B-duality and isos}). We then get that
$?\otimes^\L_{\cb}T\cong ?\otimes^\L_{\cb}T^{\ast\ast}\cong\RHom_{\cb}(T^\ast,?)$ (see Lemma~\ref{B-duality and isos}) as triangle functors from $\cd\cb$ to $\cd\ca$. It follows that $(?\otimes^\L_{\ca}T^\ast,?\otimes^\L_{\cb}T)$ is an adjoint pair whose left component is fully faithful.
\end{proof}

\begin{lemma}\label{lemma:existence of left adjoints}
Let $\cd$ be a compactly generated triangulated category and let $F:\cd\ra\ct$ be a triangle functor. If $F$ preserves small products (resp. coproducts), then it has a left (resp. right) adjoint.
\end{lemma}
\begin{proof}
Given $X\in\ct$ we consider the functor
\[\Hom_{\ct}(X,F?):\cd^\text{op}\ra\Mod k,
\]
which takes triangles to exact sequences« and small coproducts to small products. The fact that $\cd$ is compactly generated implies that it has a set of symmetric generators (see Definition~\ref{def:symmetric generators}). But then so has $\cd^\text{op}$. In particular, $\cd^\text{op}$ is perfectly generated (see Definition 1 of \cite{Krause}) and so it satisfies Brown representability (see Theorem A of \cite{Krause}). Then
\[\Hom_{\ct}(X,F?)\cong\Hom_{\cd}(LX,?)
\]
for some $LX\in\cd$. The assignment $X\mapsto LX$ yields a functor which is clearly left adjoint to $F$.
\end{proof}

The following result is related to \cite[Proposition 2.6]{BazzoniManteseTonolo}.

\begin{corollary} \label{rem.smashing and BMT}
In the situation of Theorem~\ref{special localization result}, $\cn=\ker(?\otimes_\cb^\L T)$ is a localizing subcategory of $\cd\cb$ such that
$\cd\ca\cong\cd\cb/\cn$ (see Corollary~\ref{cor.RHom fully faithful with no extras}). It turns out that the following conditions are equivalent:
\begin{itemize}
\item[1)] $\cn$ is smashing. 
\item[2)] $\cn=0$. 
\item[3)] The functors $\RHom_\ca(T,?)$ and $?\otimes^\L_\cb T$ induce mutually quasi-inverse triangle equivalences between $\cd\ca$ and $\cd\cb$.
\end{itemize}
\end{corollary}
\begin{proof}
The pair $(\cn,\op{im}(\RHom_\ca(T,?)))$ is a semi-orthogonal decomposition of $\cd\cb$. Indeed, with an argument symmetric to that in the proof of the final
statement of Theorem~\ref{teor.derived tensor fully faithful}, we just need to check that $\cn=\ ^\perp(\op{im}(\RHom_\ca(T,)))$. But we have the following
sequence of implications, for each $Y$ in $\cd\cb$:
\begin{align}
Y\in\ ^\perp(\op{im}(\RHom_\ca(T,))) \text{ if and only if } \nonumber \\
\RHom_\cb(Y,\RHom_\ca(T,?))\cong\RHom_\ca(Y\otimes_\cb^\L T,?)=0  \text{ if and only if}  \nonumber \\
Y\otimes_\cb^\L T=0 \text{ if and only if }  \nonumber \\
Y\in\cn . \nonumber
\end{align}
To say that $\cn$ is smashing is then equivalent to say that $\op{im}(\RHom_\ca(T,))$ is (a triangulated subcategory) closed under taking small coproducts in $\cd\cb$. But this is equivalent to say that $\op{im}(\RHom_\ca(T,))=\cd\cb$ since $B^\we\in\op{im}(\RHom_\ca(T,))$, for all $B\in\cb$.
\end{proof}

\begin{corollary}\label{generalized tilting and recollements}
Let $\ca$ and $\cb$ be dg categories and let $T$ be a dg $\cb$-$\ca$-bimodule such that $\lambda(B^\we):B^\we\ra\RHom_\ca(T,T(?,B))$ is an
isomorphism in $\cd\cb$, for each $B\in\cb$. Consider the following assertions:
\begin{itemize}
\item[1)] $A^\we\in\thick_{\cd\ca}(T(?,B)\ko B\in\cb)$, for all $A\in\ca$. 
\item[2)] There is a dg category $\cc$ and  a recollement of triangulated categories 
\[\xymatrix{
\cd\ca\ar@/_2.5pc/[d]\ar@/_-2.5pc/[d] \\
\cd\cb\ar[u]^{j^*}\ar@/_2.5pc/[d]\ar@/_-2.5pc/[d] \\
\cd\cc\ar[u]^{i_{*}}
}
\]
such that $j^*$ is naturally isomorphic to $?\otimes^\L_\cb T$. 
\item[3)] There is a bijective on objects homological epimorphism $F:\cb\ra\cc$ of dg categories together with a recollement as above, such that $j^*\cong ?\otimes^\L_\cb T$ and $i_{*}$ is the restriction of scalars along $F$.
\end{itemize}
Assertions $1)$ and $2)$ are equivalent and implied by $3)$. Moreover, if $\cb$ is $k$-flat, the three assertions are equivalent.
\end{corollary}
\begin{proof}
$3)\Rightarrow 2)$ Clear.

$2)\Rightarrow 1)$  From the hypothesis one gets assertion 6) of Theorem~\ref{special localization result}. Then the implication is also clear.

$1)\Rightarrow 2)$ Assertions 4 and 6 of Theorem~\ref{special localization result} hold. Hence, $?\otimes^\L_\cb T$ has a fully faithful left adjoint $L$
and a fully faithful right adjoint $\RHom_\ca(T,?)$. Then we know that $(\op{im}(L),\ker(?\otimes^\L_\cb T),\op{im}(\RHom_\ca(T,?)))$ is
a triangulated TTF triple in $\cd\cb$ (see Lemma~\ref{two-sided fully faithful adjoint}).

The top part of the desired recollement is clear using \cite[\S 2.1]{NicolasSaorinParametrizing} and the fact that $L$ induces an equivalence of categories
$\cd\ca\arr{\sim}\op{im}(L)$. Then the right adjoint to the composition $L:\cd\ca\arr{\sim}\op{im}(L)\hookrightarrow\cd\cb$ is precisely 
$?\otimes^\L_\cb T$.

The existence of the dg category $\cc$  follows from \cite[Theorem 4.3]{KellerDDC}, because the central class $\cy$ of a triangulated TTF triple
$(\cx,\cy,\cz)$ on an algebraic compactly generated triangulated category $\cd$ with small coproducts is always an algebraic compactly generated triangulated category. Indeed, if $\tau^\cy$ is the left adjoint to the inclusion $\cy\ra\cd$ and $\cs$ is a set of compact generators of $\cd$, then 
$\{\tau^\cy(S)\}_{S\in\cs}$ is a family of compact generators of $\cy$.

$1)$-$2)\Rightarrow 3)$ Use \cite[Theorem 4]{NicolasSaorinParametrizing}.
\end{proof}

\begin{remark}
Implication $1)\Rightarrow 2)$ of Corollary~\ref{generalized tilting and recollements} partially recovers the main result of \cite{Yang}. Therein, a ground field $k$ is assumed, and starting from a `generalized tilting' right $\ca$-module $T$, the category $\cd\cb$ is proved to be a recollement of derived categories in which all the functors are expressed in terms of $T$. In our case, and without assuming any kind of flatness, we are able to find a recollement starting from a `generalized tilting' $\ca$-module. It is to express the functors involved in terms of $T$ that we use a certain flatness.
\end{remark}

\begin{proposition}\label{generalized tilting and cosmashing}
Let $\cd$ be a compactly generated algebraic triangulated category and let $\cb$ be a small dg category such that $\cd$ is equivalent to $\cd\cb$. The
following assertions hold:
\begin{itemize}
\item[1)] If $\ca$ is a small dg category and $T$ is a dg $\cb$-$\ca$-bimodule satisfying some of the conditions 1) - 6) of Theorem~\ref{special localization result}, then:
\begin{itemize}
\item[1.1)] $\cz=\op{im}\RHom_\ca(T,?)$ is a compactly generated co-smashing subcategory of $\cd$ which contains
all the compact objects. 
\item[1.2)] The restrictions of $\RHom_\ca(T,?)$ and the left adjoint $L$ of $?\otimes^\L_{\cb}T$ to the subcategory $\per\ca$ of compact objects are naturally isomorphic.
\end{itemize}
\item[2)] If $\cz$ is a compactly generated co-smashing subcategory of $\cd$ which contains the compact objects, then there is a small ($k$-flat) dg category $\ca$ and a dg $\cb$-$\ca$-bimodule $T$ satisfying the conditions of Theorem~\ref{special localization result}, such that $\cz=\op{im}\RHom_\ca(T,?)$.
\end{itemize}
\end{proposition}
\begin{proof}
1.1) We have seen in the proof of Theorem~\ref{special localization result} that the triple 
$(\op{im}(L),\ker(?\otimes^\L_\cb T,\op{im}\RHom_\ca(T,?))$ is a triangulated TTF triple in $\cd$, so that $\cz=\op{im}\RHom_\ca(T,?)$ is a co-smashing
subcategory. Moreover, the objects $L(A^\we)\ko A\in\ca$ are all compact and generate $\op{im}(L)$ as a triangulated category. Since $\cd\ca\simeq\op{im}(L)\cong\cz$, we get  that $\cz$ is compactly generated. By condition 3 of Theorem~\ref{special localization result}, $\cz$ contains the all representable $\cb$-modules $B^\we\ko B\in\cb$. Then $\cz$ contains $\thick_{\cd\cb}(B^\we\ko B\in\cb)=\per\cb$.

1.2) Let $\xi: \Hom_{\cd\cb}(L(?),?)\arr{\sim}\Hom_{\cd\ca}(?,?\otimes_\cb^\L T)$ be the adjunction isomorphism. We still denote by $\xi$ the map $\xi
(M,X):\Hom_{\cd\cb}(L(M),X)\arr{\sim}\Hom_{\cd\ca}(M,X\otimes^\L_\cb T)$, hoping that no confusion appears. In particular, for each $M\in\cd\ca$,  we have an isomorphism of $k$-modules
\[
\xi:\Hom_{\cd\cb}(L(M),\RHom_\ca(T,M))\arr{\sim}\Hom_{\cd\ca}(M,\RHom_\ca(T,M)\otimes^\L_\cb T).
\]
Due to the fully faithful condition of $\RHom_\ca(T,?)$, we know that $\delta: \RHom_\ca(T,?)\otimes^\L_\cb T\ra\id_{\cd\ca}$ is an isomorphism. We put
$f(M)=\xi^{-1}(\delta (M)^{-1})$, which is a morphism $L(M)\ra\RHom_\ca(T,M)$. It is easy to see that the maps $f(M)$ define a morphism of triangle functors $f:L\ra\RHom_\ca(T,?)$. Then, by Corollary~\ref{cor:point-stable form a thick subcategory}, we know that the full subcategory formed by those objects $M$ in $\cd\ca$ for which $f(M)$ is an isomorphism is a thick subcategory. Our task is hence reduced to prove that $f(A^\we)$ is an isomorphism, for each $A\in\ca$.

We shall prove that 
\[
f(A^\we)^*=\Hom_{\cd\ca}(f(A^\we),X):\Hom_{\cd\cb}(\RHom_\ca(T,A^\we),X)\ra\Hom_{\cd\cb}(L(A^\we),X)
\]
is an isomorphism, for each compact object $X$ of $\cd$. Bearing in mind that $L(A^\we)$ and $\Hom_\ca(T,A^\we)$ are both compact, Yoneda's lemma will give then that $f(A^\we)$ is an isomoprhism.

Since the full subcategory consisting of objects $X$ of $\cd$ such that $f(A^\we)^{*}=\Hom_{\cd\ca}(f(A^\we),\Sigma^p X)$ is an isomorphism,
for all $p\in\Z$, is a thick subcategory, we can assume without loss of generality that $X=\Sigma^{p}(B^\we)$ is a shift of a representable object. For simplicity, put $H=\RHom_\ca(T,?)$ and $G=?\otimes^\L_\cb T$. Consider the square:
\[
\xymatrix{
\Hom_{\cd\cb}(H(A^\we),HG(\Sigma^pB^\we))\ar[rr]^{\lambda (\Sigma^pB^\we)_{*}^{-1}} && \Hom_{\cd\cb}(H(A^\we),\Sigma^pB^\we)\ar[d]^{f(A^\we)^*} \\ 
\Hom_{\cd\ca}(A^\we,G(\Sigma^p B^\we))\ar[u]^{H} & & \Hom_{\cd\cb}(L(A^\we),\Sigma^pB^\we)\ar[ll]^{\xi}.
}
\]
Note that all the arrows in this diagram are isomorphisms, except possibly $f(A^\we)^*$. Moreover, looking at the explicit
definition of the maps involved, one can check that the composition of the arrows of the square gives the identity. Indeed, put
$\Psi (u):=\xi (\lambda(\Sigma^pB^\wedge)^{-1}\circ H(u)\circ f(A^\wedge ))$, for each $u\in\Hom_{\cd\ca}(A^\wedge ,G(\Sigma^pB^\wedge))$. Put then $\alpha :=\lambda (\Sigma^pB^\wedge)^{-1}\circ H(u)$, so that $\Psi
(u)=(\xi\circ\alpha_*)(f(A^\wedge ))$, where $\alpha_*:=\Hom_{\cd\cb}(L(A^\wedge ),\alpha )$. The naturality of $\xi$ gives that
$G(\alpha )_*\circ\xi=\xi\circ\alpha_*$. Then we get an equality:
\begin{center}
$\Psi(u)=(\xi\circ\alpha_*)(f(A^\wedge))=(G(\alpha )_*\circ\xi )(f(A^\wedge))=G(\alpha )\circ\xi (f(A^\wedge))=G(\lambda (\Sigma^pB^\wedge
)^{-1}\circ H(u))\circ\delta (A^\wedge )^{-1}=G(\lambda (\Sigma^pB^\wedge))^{-1}\circ GH(u)\circ\delta (A^\wedge )^{-1}$.
\end{center}
But the fact that $\delta$ is an isomorphism and the  equations of the adjunction $(G,H)$ give that $G(\lambda (\Sigma^pB^\wedge
))^{-1}=\delta (G(\Sigma^pB^\wedge))$. It follows that $\Psi (u)=\delta (G(\Sigma^pB^\wedge))\circ GH(u)\circ\delta (A^\wedge
)^{-1}=u$, using the naturality of $\delta$. Then $\Psi$ is the identity map. As a consequence, $f(A^\we)^*$ is an isomorphism.

2) By Corollary~\ref{smashing in bijection with cosmashing} and its proof, we have a triangulated TTF triple $(\cx,\cy,\cz)$ on $\cd$ where $\cx$ is a compactly generated triangulated category. Then the composition 
\[
\cx\arr{i_\cx}\cd\arr{\tau^\cz}\cz
\]
(where $i_{\cx}$ is the inclusion functor and $\tau^\cz$ is a left adjoint to the inclusion functor) is an equivalence of categories,
which induces by restriction another one 
\[
\cx^c\arr{\sim}\cz^c
\] 
between the corresponding subcategories of the compact objects. Note
that $\mathcal{X}^c=\mathcal{X}\cap\mathcal{D}^c$. But the fact that
$\mathcal{D}^c\subset\mathcal{Z}$ then implies that
$\tau^\mathcal{Z} (i_\mathcal{X}(X))=X$, for each
$X\in\mathcal{X}^c$. It follows that
$\mathcal{Z}^c=\mathcal{X}\cap\mathcal{D}^c$.

Let $\cp$ be a set of compact generators of $\cx$. Then, by \cite[Corollary 2.5]{NicolasSaorinParametrizing}, there is a small dg category $\ca$ together with a dg $\ca$-$\cb$-bimodule $\tilde{T}$ such that 
\[
?\otimes^\L_\ca\tilde{T}:\cd\ca\ra\cd\cb
\]
is fully faithful and its essential image is $\cx$. Note that  $A^\we\otimes^\L_\ca\tilde{T}=\tilde{T}(?,A)$ is then compact in $\cd\cb$, for each
$A\in\ca$.

On the other hand, after \cite{Tabuada2005a}, we can consider a quasi-isomorphism of dg categories 
$f:\hat{\ca}\ra\ca$, where $\hat{\ca}$ is $k$-projective. It is clear that we can replace $\ca$ by $\hat{\ca}$ and $\tilde{X}$ by $f_*(\tilde{X})$, where
$f_*:\cc_{dg}(\ca^\text{op}\otimes\cb)\ra\cc_{dg}(\hat{\ca}^\text{op}\otimes\cb)$ is the restriction of scalars. So, without loss of generality, we assume that $\mathcal{A}$ is $k$-projective. Then we define $T:=\RHom_\cb(\tilde{T},\cb)$. By Lemma~\ref{B-duality and isos}, we know that $\RHom_\cb(\tilde{T},?)$ and $?\otimes^\L_\cb T$ are naturally isomorphic triangulated functors $\cd\cb\ra\cd\ca$. It follows that $?\otimes^\L_\cb T$ has a fully faithful left
adjoint $?\otimes^\L_\ca\tilde{T}$ and a necessarily fully faithful (see Lemma~\ref{two-sided fully faithful adjoint}) right adjoint $\RHom_\ca(T,?)$. It follows that
\[
(\op{im}(?\otimes^\L_\ca\tilde{T}),\ker(?\otimes^\L_\cb T),\op{im}(\RHom_\ca(T,?))
\] 
is a triangulated TTF triple whose first component is $\cx$. Then we have $\cz=\op{im}(\RHom_\ca(T,?))$. In particular $B^\we\in\op{im}(\RHom_\ca(T,?))$, for each $B\in\cb$, and then the unit map $\lambda(B^\we):B^\we\ra\RHom_\ca(T,B^\we\otimes^\L_\cb T)$ is an isomorphism since $\RHom_\ca(T,?)$ and
$?\otimes^\L_\cb T$ are mutually inverse equivalences of categories between $\cd\ca$ and $\op{im}(\RHom_\ca(T,?))$. So $T$ satisfies condition 6 of Theorem~\ref{special localization result}.
\end{proof}

\begin{remark}\label{different but isomorphic duals}
Let $\ca$ and $\cb$ be small dg categories. Assume $\ca$ is $k$-projective. Let $T$ be a dg $\cb$-$\ca$-bimodule satisfying conditions 1)-6) of Theorem~\ref{special localization result}. After the proof of $1)$-$6)\Rightarrow 7)$ we know that the left adjoint $L$ to $?\otimes^\L_{\cb}T$ is isomorphic to $?\otimes^\L_{\ca}T^\ast$, where $T^\ast=\RHom_{\cb}(T,\cb)$ (see Notation~\ref{B-dual}). Moreover, after Proposition~\ref{generalized tilting and cosmashing}, we know that $L(A^\we)\cong\RHom_{\ca}(T,A^\we)$ for each $A\in\ca$. Now, consider the dg $\ca$-$\cb$-bimodule $\RHom_{\ca}(T,\ca)$ as we did in Notation~\ref{B-dual} but replacing the regular dg $\cb$-$\cb$-bimodule $\cb$ by the regular dg $\ca$-$\ca$-bimodule $\ca$ (see Notation~\ref{notation: regular bimodule}). Then for each $A\in\ca$ we have that
\begin{align}
\RHom_{\ca}(T,\ca)(?,A)\cong \nonumber \\
\RHom_{\ca}(T,A^\we)\cong \nonumber \\
L(A^\we)\cong \nonumber \\
A^\we\otimes^\L_{\ca}T^{\ast}\cong \nonumber \\
T^\ast(?,A)= \nonumber \\
\RHom_{\cb}(T,\cb)(?,A). \nonumber
\end{align}
We have just proved that the dg $\ca$-$\cb$-bimodules $T^\ast=\RHom_{\cb^\text{op}}(T,\cb)$ and $\RHom_{\ca}(T,\ca)$ become isomorphic in $\cd\cb$ after restricting scalars.
\end{remark}

\section{Results for ordinary algebras}\label{Results for ordinary algebras}

\subsection{The thick subcategory generated by an exceptional module}

\begin{lemma}\label{coproduct-product}
Let $\ca$ be a small dg category and let $\{M_{i}\}_{i\in I}$ be a family of right dg $\ca$-modules such that, for each integer $p\in\Z$ and each object $A\in\ca$, the set $I(p,A)$ formed by those indexes $i\in I$ such that $H^pM_{i}(A)\neq 0$ is finite. Then the canonical morphism
\[\coprod_{i\in I}M_{i}\ra\prod_{i\in I}M_{i}
\]
is a quasi-isomorphism.
\end{lemma}
\begin{proof}
For each integer $p\in\Z$ and each object $A\in\ca$, the homology functor
\[\cc\ca\ra\Mod\Z\ko M\mapsto H^{p}M(A)
\]
preserves products and coproducts, since products and coproducts are exact in $\Mod\Z$. Then we get a map
\[\coprod_{i\in I}H^{p}M_{i}(A)\cong H^p(\coprod_{i\in I}M_{i})(A)\ra H^p(\prod_{i\in I}M_{i})(A)\cong \prod_{i\in I}H^pM_{i}(A).
\]
But the hypothesis implies that this map is bijective.
\end{proof}

\begin{proposition}\label{description thickT}
Let $A$ be an ordinary algebra and $T$ an $A$-module such that $\Ext^{p}_{A}(T,T)=0$ for all $p>0$. Then:
\begin{itemize}
\item[1)] $\thick_{\cd A}(T)$ consists of those complexes which are isomorphic in $\cd A$ to direct summands of bounded complexes of modules in $\op{sum}_{\Mod A}(T)$.
\item[2)] The quotient functor $\ch A\ra\cd A$ induces a triangle equivalence between the full subcategory of $\ch A$ formed by complexes isomorphic to bounded complexes over $\op{sum}_{\Mod A}(T)$ and $\tria_{\cd A}(T)$.
\end{itemize}
\end{proposition}
\begin{proof}
By \cite[Lemma 2.2.2]{BondalVanDenBergh} we know that $\thick_{\cd A}(T)=\op{add}_{\cd A}(\tria_{\cd A}(T))$. Thus we just need to prove 2). Although the situation here is more general, (the proof of) \cite[Lemma 1.1]{Happel} gives that the restriction of the the quotient functor to the subcategory mentioned in assertion 2 is fully faithful. Its essential image is then a triangulated subcategory of $\tria_{\cd A}(T)$ which contains $T$, and so the functor induces a triangulated equivalence as stated. 
\end{proof}

\begin{lemma}\label{T reflexibility}
Let $A$ be an ordinary algebra, $T$ an $A$-module, $B=\End_{A}(T)$ and $0\ra Y\ra T^0\arr{f} T^1$ an exact sequence where $T^p\in\op{add}_{\Mod A}(T)$ and $\Ext^1_{A}(\op{im}(f),T)=0$. Then 
\[
\sigma(Y): Y\ra\Hom_{B^\text{op}}(\Hom_{A}(Y,T),T)\ko y\mapsto (f\mapsto f(y))
\] 
is an isomorphism.
\end{lemma}
\begin{proof}
We have a commutative diagram
\[
\xymatrix{
0\ar[d] && 0\ar[d] \\ 
Y\ar[d]\ar[rr]^{\sigma(Y)\hspace{1.8cm}} && \Hom_{B^\text{op}}(\Hom_{A}(Y,T),T)\ar[d] \\
T^0\ar[d]\ar[rr]^{\sigma(T^0)\hspace{1.8cm}} && \Hom_{B^\text{op}}(\Hom_{A}(T^0,T),T)\ar[d]\\
\op{im}(f)\ar[d]\ar[rr]^{\sigma(\op{im}(f))\hspace{1.8cm}} && \Hom_{B^\text{op}}(\Hom_{A}(\op{im}(f),T),T) \\
0 &
}
\]
with exact columns, where $\sigma(T^{0})$ is an isomorphism and $\sigma(\op{im}(f))$ a monomorphism. By the Snake Lemma, we conclude that $\sigma(Y)$ is an isomorphism.
\end{proof}

\begin{proposition}\label{A in thickT}
Let $A$ be an ordinary algebra and let $T$ be an $A$-module such that $\Ext^{p}_{A}(T,T)=0$ for all $p>0$. The following assertions are equivalent:
\begin{itemize}
\item[1)] $A\in\thick_{\cd A}(T)$.
\item[2)] There is an exact sequence
\[0\ra A\ra T^0\ra T^1\ra\dots\ra T^n\ra 0
\]
with $T^{i}\in\op{add}_{\Mod A}(T)$ for all $i=0,1,\dots,n$.
\end{itemize}
\end{proposition}
\begin{proof}
Implication $2)\Rightarrow 1)$ is clear. Let us prove $1)\Rightarrow 2)$. By Proposition~\ref{description thickT} we know that there is a section
\[A\ra T^\bullet
\]
in $\cd A$, where $T^\bullet$ is a bounded complex
\[\dots\ra 0\ra T^{-m}\ra\dots\ra T^{-1}\ra T^0\ra T^1\ra\dots\ra T^n\ra 0\ra\dots
\]
in which each $T^p$ is in $\op{sum}_{\Mod A}(T)$. Since $A$ is homologically projective in $\ch A$, this morphism in $\cd A$ is represented by a cochain map $f\in\Hom_{\cc A}(A,T^\bullet)$. In $\cc A$ we have a factorization of $f$
\[f: A\arr{\tilde{f}} \sigma^{\geq 0}(T^\bullet)\ra T^\bullet
\]
where $\sigma^{\geq 0}(T^\bullet)$ is the stupid truncation 
\[\dots\ra 0\ra T^0\ra T^1\ra\dots\ra T^n\ra 0\ra\dots
\]
Hence $\tilde{f}$ also induces a section in $\cd A$. Thus we can assume that in the morphism $f\in\Hom_{\cd A}(A,T^\bullet)$ the codomain $T^\bullet$ is concentrated in non-negative degrees. We fix a decomposition $T^\bullet=A\oplus X$ in $\cd A$, and let $e\in\End_{\cd A}(T^\bullet)$ the idempotent corresponding to the matrix
\[\left[\begin{array}{cc}0 & 0 \\ 0 & \id\end{array}\right]: A\oplus X\ra A\oplus X.
\]
By part 2) of Proposition~\ref{description thickT} we know that $e$ is represented by a cochain map that we still denote by $e$. Let
\[
\tilde{T}: \dots\ra 0\ra \tilde{T}^{-1}\ra\tilde{T}^0\ra\tilde{T}^1\ra\dots\ra\tilde{T}^{n}\ra 0\ra\dots
\]
be its mapping cone. We then have an isomorphism of triangles
\[\xymatrix{
A\oplus X\ar[rr]^{\left[\begin{array}{cc}0&0 \\ 0&\id\end{array}\right]}\ar[d]^{\wr} && A\oplus X\ar[rr]^{\left[\begin{array}{cc}\id & 0 \\ 0 & 0\end{array}\right]}\ar[d]^{\wr} && A\oplus \Sigma A\ar[d]^{\wr}\ar[rr]^{\left[\begin{array}{cc}0 & \id \\ 0 & 0\end{array}\right]} && \Sigma A\oplus \Sigma X\ar[d]^{\wr} \\
T^\bullet\ar[rr]^{e} && T^\bullet\ar[rr] && \tilde{T}\ar[rr] && \Sigma T^\bullet
 }
\]
where the triangle at the top is the direct sum of the two split triangles
\[\xymatrix{A\ar[r]^{0} & A\ar[rr]^{\left[\begin{array}{c}\id \\ 0\end{array}\right]} && A\oplus\Sigma A\ar[rr]^{\left[\begin{array}{cc}0 & \id\end{array}\right]} && \Sigma A
}
\]
and
\[X\arr{\id} X\ra 0\ra\Sigma X.
\]
We then get that $\tilde{T}\cong A\oplus\Sigma A$ in $\cd A$, and so:
\begin{itemize}
\item[i)] $H^p(\tilde{T})=0$ for $p\neq -1,0$.
\item[ii)] There is a section $A\ra\tilde{T}$ in $\cd A$.
\end{itemize}
Taking again the stupid truncation $\sigma^{\geq 0}(\tilde{T})$, we can deduce that we have a section $A\ra \widehat{T}$ in $\cd A$ where
\[\widehat{T}:\dots\ra 0\ra \widehat{T}^0\ra\dots\ra\widehat{T}^n\ra 0\ra\dots
\]
is a bounded complex over $\op{sum}_{\Mod A}(T)$ with homology concentrated in degree $0$. Note that $H^0(\widehat{T})=A\oplus X$ and that we have a quasi-isomorphism $g: H^0(\widehat{T})\ra\widehat{T}$. Let us denote by $(?)^*$ the contravariant functor
\[\Hom_{A}(?,T):\cc A\ra\cc(B^\text{op}),
\]
where $B=\End_{A}(T)$. Since $\Ext^p_{A}(T,T)=0$ for $p>0$, the cochain map
\[\widehat{T}^*\arr{g^*}H^0(\widehat{T})^*=A^*\oplus X^*
\] 
yields a projective resolution in $\Mod B^\text{op}$ consisting of finitely generated terms. Choose now projective resolutions
\[P_{(A,T)}\ra A^*
\]
and
\[P_{(X,T)}\ra X^*
\]
in $\Mod B^\text{op}$. The Comparison Theorem (see \cite[Theorem 2.2.6]{Weibel}) tells us that $\widehat{T}^*$ is isomorphic to $P_{(A,T)}\oplus P_{(X,T)}$ in $\ch(B^\text{op})$. This means that there are contractible complexes $C$ and $C'$ in $\cc(B^\text{op})$ such that
\[\widehat{T}^*\oplus C\cong P_{(A,T)}\oplus P_{(X,T)}\oplus C'
\]
in $\cc(B^\text{op})$. Let us denote also by $(?)^*$ the contravariant functor
\[\Hom_{B^\text{op}}(?,T): \cc(B^\text{op})\ra\cc A.
\]
After applying it we get an isomorphism
\[
\widehat{T}^{**}\oplus C^*\cong P_{(A,T)}^*\oplus P_{(X,T)}^*\oplus C'^*
\]
in $\cc A$. Note that $C^*$ and $C'^*$ are again contractible complexes, and that $\widehat{T}^{**}\cong \widehat{T}$ in $\cc A$. Thus, $P_{(A,T)}^*$ is isomorphic in $\ch A$ to a bounded complex $Q$ over $\op{add}_{\Mod A}(T)$ concentrated in non-negative degrees and with homology concentrated in degree $0$ isomorphic to $A^{**}$. Finally, since $A$ is a direct summand of a kernel, $H^0(\widehat{T})$, of a morphism $\widehat{T}^0\ra\widehat{T}^{1}$ between objects of $\op{sum}_{\Mod A}(T)$, Lemma~\ref{T reflexibility} implies that $A^{**}\cong A$ in $\Mod A$. Thus 
\[0\ra A\ra Q^0\ra Q^1\ra\dots\ra Q^n\ra 0
\]
is the required sequence.
\end{proof}

\begin{remark}
Proposition~\ref{A in thickT} is obviously false if we replace $T$ by an arbitrary $A$-module. Indeed, if $A$ is a hereditary Artin algebra and $M=A/J(A)$, where $J(A)$ is the Jacobson radical, then $\op{thick}_{\cd A}(M)=\per A$. However, condition 2 of Proposition~\ref{A in thickT} does not hold unless $A$ is semisimple.
\end{remark}


\subsection{Consequences of the results on dg categories}

The following consequence of Theorem~\ref{teor.derived tensor fully faithful} for ordinary algebras shows that in that context, the fact that $?\otimes_B^\L T$ is fully faithful is related to the notion of tilting module. Recall that a (right) module $M$ over an ordinary algebra $A$ is called {\em self-small} if the
canonical map 
\[
\Hom_A(M,M)^{(\alpha)}\ra\Hom_A(M,M^{(\alpha)})
\] 
is bijective, for all cardinals $\alpha$.

\begin{corollary}\label{cor.tensor fully faithful and tilting}
Let $A$ and $B$ be ordinary algebras and $T$ be a $B$-$A$-bimodule. Consider the following assertions:
\begin{itemize}
\item[1)] $?\otimes_B^\L T:\cd B\ra\cd A$ is fully faithful. 
\item[2)] The structural algebra homomorphism $B\ra\End_A(T)$ is an isomorphism, $T$ is self-small as a right $A$-module and $\Ext_A^i(T,T^{(\alpha)})=0$, for all integer $i>0$ and all cardinals $\alpha$.
\end{itemize}
Then $1)\Rightarrow 2)$ and, in case $\RHom_A(T,?):\cd A\ra \cd B$ preserves small coproducts of objects in $\Tria_{\cd A}(T)$, the converse is also true.
\end{corollary}
\begin{proof}
For each cardinal $\alpha$, the unit map 
\[
\lambda(B^{(\alpha)}):B^{(\alpha )}\ra \RHom_A(T,B^{(\alpha)}\otimes_B^\L T)=\RHom_A(T,T^{(\alpha)})
\] 
is an isomorphism if, and only if, $\Ext_A^i(T,T^{(\alpha )})=0$, for all integers $i>0$, and the canonical map 
\[B^{(\alpha)}\cong\Hom_A(T,T)^{(\alpha)}\ra\Hom_A(T,T^{(\alpha )})
\] 
is an isomorphism. Then the result is a direct consequence of Theorem~\ref{teor.derived tensor fully faithful}.
\end{proof}

The following result proves at once both the first part of \cite[Lemma 4.2]{Chen-Xi2} and its converse.

\begin{corollary}\label{recollement and bimodule}
Let $A$ and $B$ be ordinary algebras and $T$ be a $B$-$A$-bimodule. The following assertions are equivalent:
\begin{itemize}
\item[1)] There is a recollement
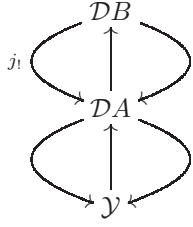
\[\xymatrix{
\cd B\ar@/_2.5pc/[d]_{j_{!}}\ar@/_-2.5pc/[d] \\
\cd A\ar[u]\ar@/_2.5pc/[d]\ar@/_-2.5pc/[d] \\
\cy\ar[u]
}
\]
such that $j_{!}=?\otimes^\L_{B}T$.
\item[2)] The following conditions hold:
\begin{itemize}
\item[2.1)] The structural map $B\ra\End_{A}(T)$ is an isomorphism of algebras. 
\item[2.2)] As a right $A$-module, $T$ admits a finite projective resolution with finitely generated terms.
\item[2.3)] $\Ext_A^i(T,T)=0$, for all $i>0$
\end{itemize}
\end{itemize}
\end{corollary}
\begin{proof}
Due to Brown representability theorem and to Lemma~\ref{two-sided fully faithful adjoint}, the recollement of
assertion 1) exists if, and only if, $?\otimes^\L_B T$ is fully faithful and $\RHom_A(T,?):\cd A\ra\cd B$ preserves small coproducts. But $\RHom_A(T,?)$ preserves small coproducts exactly when $T\in\per A$, i.e., when $2.2)$ holds. In this case the right $A$-module $T$ is self-small and we have an isomorphism 
\[
\Ext_A^i(T,T)^{(\alpha)}=H^i(\RHom_A(T,T)^{(\alpha)})\arr{\sim}H^i(\RHom_A(T,T^{(\alpha)})=\Ext_A^i(T,T^{(\alpha )}),
\] 
for each integer $i>0$ and each cardinal $\alpha$.  The result is now a direct consequence of Corollary~\ref{cor.tensor fully faithful and tilting}.
\end{proof}


As a consequence of Corollary~\ref{recollement and bimodule} we have the following result, which generalizes the recollement of \cite[\S 6.1]{Chen-Xi2} and \cite[Theorem 1.3]{Chen-Xi2}.

\begin{corollary}\label{generalizing ChenXi}
Let $A$ be an ordinary algebra, let $W$ be an injective cogenerator of $\Mod A$ and let $U$ be a right $A$-module satisfying the following two condictions:
\begin{itemize}
\item[1)] $\Ext_A^i(U,U)=0$, for all $i>0$. 
\item[2)] There is a natural number $n$ and an exact sequence $0\ra U^{-n}\ra\dots\ra U^{-1}\ra U^0\ra W\ra 0$, with $U^i\in\op{add}(U)$ for $i=-n,...,-1,0$.
\end{itemize}
If $\Lambda=\End_A(W)\ko B=\End_A(U)$ and $T=\Hom_A(U,W)$, which is a $\Lambda$-$B$-bimodule, then there is a recollement 
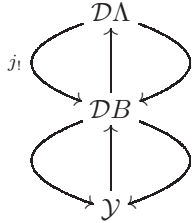
\[\xymatrix{
\cd\Lambda\ar@/_2.5pc/[d]_{j_{!}}\ar@/_-2.5pc/[d] \\
\cd B\ar[u]\ar@/_2.5pc/[d]\ar@/_-2.5pc/[d] \\
\cy\ar[u]
}
\]
such that $j_!=?\otimes^\L_\Lambda T$.
\end{corollary}
\begin{proof}
We shall prove that conditions 2.1)-2.3) of Corollary~\ref{recollement and bimodule} hold (with the pair of algebras $(B,A)$ replaced by $(\Lambda ,B)$).

The key point here is that the functor $\Hom_A(U,?):\Mod A\ra\Mod B$ keeps exact the sequence in 2). On the other hand, the trace map
$\Hom_A(U,U^i)\otimes_BU\ra U^i$ is an isomorphism since $U^i\in\op{add}_{\Mod A}(U)$, for each $i=-n,...,-1,0$. The right exactness of
the functor $?\otimes_BU$ then implies that also the trace map $T\otimes_BU=\Hom_A(U,W)\otimes_BU\ra W$ is an isomorphism. From that it follows easily that the map $\lambda:\Lambda=\End_A(W)\ra\End_B(T)$ given by the functor $\Hom_A(U,?)$ is an isomorphism. Then condition 2.1) in Corollary~\ref{recollement and bimodule} holds.

On the other hand, applying the functor $\Hom_A(U,?)$ to the exact sequence in 2), we obtain a finite projective resolution of $T$ as a
right $B$-module, which consist of finitely generated terms since $U^i\in\op{add}(U)$ for all $i$. Therefore condition 2.2) in Corollary~\ref{recollement and bimodule} holds.

It remains to check that $\Ext_B^i(T,T)=0$, for all $i>0$. We need to prove that, for each $i=-n,...,-1$, the sequence
\[
\Hom_B[(U,U^{i+1}),(B,W)]\ra\Hom_B[(U,U^{i}),(B,W)]\ra\Hom_B[(U,U^{i-1}),(B,W)]\ (\ast)
\]
is exact at the central point, where $(U,X)=\Hom_A(U,X)$ for each right $A$-module $X$. But the argument in the second paragraph of
this proof essentially proves that the functors $\Hom_A(U,?)$ and $?\otimes_BU$ induces mutually inverse maps
\[
\Hom_A(U',W)\arr{\sim}\Hom_B[(U,U'),(U,W)],
\]
for each $U'\in\op{add}(U_A)$. It follows that the exactness of the sequence $(\ast)$ is tantamount to the exactness at the central point of
the sequence
\[
\Hom_A(U^{i+1},W)\ra\Hom_A(U^{i},W)\ra\Hom_A(U^{i-1},W),
\]
which is clear since $W$ is an injective right $A$-module.
\end{proof}

\begin{remark}
Note that in Corollaries~\ref{recollement and bimodule} and \ref{generalizing ChenXi} the triangulated category $\cy$ is equivalent to $\cd C$ for some dg algebra $C$. If, in addition, $B$ is a $k$-flat then such $C$ can be chosen together with a homological epimorphism of dg algebras $f: B\ra C$ suh that 
\[
i_{*}:\cd C\simeq \cy\ra \cd \cb
\]
is the restriction of scalars along $f$. The proof of these facts goes as in the proof of Corollary~\ref{generalized tilting and recollements}.
\end{remark}

\begin{remark}
Note that if $A$ is a ring with Morita selfduality (\eg an Artin algebra) and $W$ is the minimal injective cogenerator, then $\Lambda$ and $A$ are Morita equivalent in the above example. Compare with \cite[Theorem 1.3]{Chen-Xi2}.
\end{remark}

The following consequence of Theorem~\ref{special localization result} is related to \cite[Theorem 2.2]{BazzoniManteseTonolo}.

\begin{corollary} \label{cor.tilting buenos via RHom}
Let $A$ and $B$ be ordinary algebras and $T$ be a $B$-$A$-bimodule. The next assertions are equivalent:
\begin{itemize}
\item[1)] The following conditions hold:
\begin{itemize}
\item[a)] $T$ is a faithfully balanced bimodule. 
\item[b)] $\Ext_A^p(T,T)=0=\Ext_{B^\text{op}}^p(T,T)$, for all $p>0$. 
\item[c)] $T$ admits a finite projective resolution with finitely generated terms as left $B$-module.
\end{itemize}
\item[2)] $\RHom_A(T,?):\cd A\ra\cd B$ is fully faithful, preserves compact objects and $B$ is in its essential image. 
\item[3)] The following conditions hold:
\begin{itemize}
\item[a)] The structural algebra morphism $B\ra\End_A(T)$ is an isomorphism. 
\item[b)] $\Ext_A^p(T,T)=0$ for all $p>0$.
\item[c)] There exists an exact sequence in $\Mod A$
\[0\ra A\ra T^0\ra T^1\ra \dots\ra T^n\ra 0, 
\]
where the $T^i$ are direct summands of finite direct sums of copies of $T$.
\end{itemize}
\item[4)] The following conditions hold:
\begin{itemize}
\item[a)] The structural algebra morphism $B\ra\End_{A}(T)$ is an isomorphism.
\item[b)] $T\otimes_A^\L?:\cd(A^\text{op})\ra\cd(B^\text{op})$ is fully faithful and has a right adjoint which preserves coproducts.
\end{itemize}
\item[5)] The following conditions hold:
\begin{itemize}
\item[a)] The structural algebra morphism $B\ra\End_A(T)$ is an isomorphism. 
\item[b)] $\Ext_A^p(T,T)=0$ for all $p>0$.
\item[c)] The functor $?\otimes^\L_BT:\cd B\ra\cd A$ has a fully faithful left adjoint.
\end{itemize}
\end{itemize}
\end{corollary}
\begin{proof}
Let us regard the ordinary algebras $A$ and $B$ as dg categories $\ca$ and $\cb$. Then
statement 1 (resp. 2, 3, 4, 5) correspond to statement 1 of Theorem~\ref{special localization result} (resp. 3, 4, 5, 6). 
\end{proof}

The next result is also a consequence of Theorem~\ref{special localization result} (see also Example~\ref{good n-tilting}):

\begin{corollary}\label{some characterization of good tilting}
Let $A$ be an ordinary algebra, $T$ a right $A$-module and $B=\End_{A}(T)$. Assume $T$ has finite projective dimension over $A$ and $\Ext^p_{A}(T,T^{(\alpha)})=0$ for each cardinal $\alpha$ and each integer $p>0$. The following conditions are equivalent:
\begin{itemize}
\item[1)] $T$ is a good tilting module over $A$.
\item[2)] $\RHom_{A}(T,?):\cd A\ra\cd B$ is fully faithful and preserves compact objects.
\item[3)] $?\otimes^\L_{B}T$ has a fully faithful left adjoint.
\end{itemize}
\end{corollary}

Finally we deduce:

\begin{corollary}
Let $A$ be a $k$-projective ordinary algebra,  $T$ a right $A$-module satisfying the following conditions:
\begin{itemize}
\item[a)] There exists an exact sequence $0\ra A\ra T^0\ra\dots\ra T^n\ra 0$ in $\Mod A$, with $T^i\in\op{add}(T)$, for all $i\geq 0$.
\item[b)] $\Ext_A^i(T,T)=0$ for all $i>0$. 
\end{itemize}
If $B=\End(T_A)$ then $\Ext_A^i(T,A)$ and $\Ext_{B^\text{op}}^i(T,B)$ are isomorphic $k$-modules, for all $i\geq 0$.
\end{corollary}
\begin{proof}
Just take homologies in $\RHom_{A}(T,A)$ and $\RHom_{B^\text{op}}(T,B)$ and use Corollary~\ref{cor.tilting buenos via RHom} and Remark~\ref{different but isomorphic duals}.
\end{proof}

\subsection{Hereditary case}


\begin{lemma}\label{relating the two counits}
Let $A$ and $B$ be ordinary algebras, and $T$ an $A$-$B$-bimodule such that $\op{pd}_{A}(T)\leq 1$. Let $\delta$ be the counit of the adjoint pair
\[\xymatrix{ \cd A\ar@<1ex>[d]^{\RHom_{A}(T,?)}\\
\cd B,\ar@<1ex>[u]^{?\otimes^L_{B}T}
}
\]
and let $\varepsilon$ be the counit of the adjoint pair
\[\xymatrix{ \Mod A\ar@<1ex>[d]^{\Hom_{A}(T,?)}\\
\Mod B.\ar@<1ex>[u]^{?\otimes_{B}T}
}
\]
The following assertions hold:
\begin{itemize}
\item[1)] There are morphisms 
\[\tau_{i}: \op{Tor}^B_{i+2}(\Ext^1_{A}(T,?),T)\ra\op{Tor}^B_{i}(\Hom_{A}(T,?),T)\ko i\geq 0,
\]
of  endofunctors of $\Mod A$.
\item[2)] There is a commutative diagram with exact columns formed by endofunctors of $\Mod A$ and morphisms between them
\[\xymatrix{
\op{Tor}^B_{2}(\Ext^1_{A}(T,?),T)\ar[d]^{\tau_{0}} & & \\
\Hom_{A}(T,?)\otimes_{B}T\ar[rr]^{\hspace{0.5cm}\varepsilon}\ar[d] && \id_{\Mod A}\ar@{=}[d]\\
H^0(\RHom_{A}(T,?)\otimes^\L_{B}T)\ar[rr]^{\hspace{1cm}H^0(\delta)}\ar[d] && \id_{\Mod A}\\
\op{Tor}^B_{1}(\Ext^1_{A}(T,?),T) &&
}
\]
\item[3)] For an $A$-module $M$ we have that $\delta_{M}$ is an isomorphism if and only if the following conditions hold:
\begin{itemize}
\item[a)] $(\tau_{i})_{M}$ is an isomorphism for all $i>0$,
\item[b)] $(\tau_{0})_{M}$ is a monomorphism,
\item[c)] $H^0(\delta_{M})$ is an isomorphism,
\item[d)] $\Ext^1_{A}(T,M)\otimes_{B}T=0$.
\end{itemize}
\end{itemize}
\end{lemma}
\begin{proof}
Let $M$ be a right $A$-module. Consider the triangle
\[\tau_{\leq 0}\RHom_{A}(T,M)\ra\RHom_{A}(T,M)\ra\tau_{\geq 1}\RHom_{A}(T,M)\arr{+}
\]
in $\cd B$ formed by using canonical truncations. Since the right $A$-module $T$ has projective dimension at most $1$, the complex of $B$-modules $\RHom_{A}(T,M)$ has homology concentrated in degrees $0$ and $1$. Thus
\[\tau_{\leq 0}\RHom_{A}(T,M)\cong H^0\RHom_{A}(T,M)\cong\Hom_{\cd A}(T,M)\cong\Hom_{A}(T,M)
\]
and
\[\tau_{\geq 1}\RHom_{A}(T,M)\cong\Sigma^{-1}H^1\RHom_{A}(T,M)\cong\Sigma^{-1}\Hom_{\cd A}(T,\Sigma M)\cong\Sigma^{-1}\op{Ext}^1_{A}(T,M). 
\]
Therefore, there is a triangle in $\cd B$ of the form
\[\Hom_{A}(T,M)\ra\RHom_{A}(T,M)\ra \Sigma^{-1}\op{Ext}^{1}_{A}(T,M)\arr{+}
\]
After applying $?\otimes^\L_{B}T$ we get a triangle
\[\Hom_{A}(T,M)\otimes^\L_{B}T\ra\RHom_{A}(T,M)\otimes^\L_{B}T\ra\Sigma^{-1}\op{Ext}^1_{A}(T,M)\otimes^\L_{B}T\arr{+}
\]
in $\cd A$ which is functorial in $M$. Bearing in mind that, for each $B$-module $N$, we have $H^{-i}(N\otimes^\L_{B}T)\cong \op{Tor}^{B}_{i}(N,T)$ for each $i\geq 0$, and $H^{i}(N\otimes^\L_BT)=0$ for each $i\geq 1$,   the long exact sequence of homologies for the last triangle, when taken  for degrees $\leq 1$,  gives the following exact sequence, which is then natural in $M$:

\[
\xymatrix{ 
H^{-2}(\RHom_A(T,M)\otimes^\L_BT) \ar[r] & \op{Tor}^B_3(\op{Ext}^{1}_A(T,M),T)\ar[r]^{\tau_1} & \op{Tor}^{B}_1(\Hom_A(T,M),T)\ar@{->} `r/3pt[d] `/10pt[l] `^dl[ll] `^d/10pt[dl] [dll] \\                
 H^{-1}(\RHom_A(T,M)\otimes^\L_BT) \ar[r] & \op{Tor}^{B}_2(\op{Ext}^{1}_A(T,M),T) \ar[r]^{\tau_0} & \Hom_A(T,M)\otimes_BT\ar@{->} `r/3pt[d] `/10pt[l] `^dl[ll] `^d/10pt[dl] [dll] \\                
 H^{0}(\RHom_A(T,M)\otimes^\L_BT) \ar[r] & \op{Tor}^{B}_1(\op{Ext}^{1}_A(T,M),T) \ar[r] & 0 \ar@{->} `r/3pt[d] `/10pt[l] `^dl[ll] `^d/10pt[dl] [dll] \\
H^{1}(\RHom_A(T,M)\otimes^\L_BT) \ar[r] & \op{Ext}^{1}_A(T,M)\otimes_BT \ar[r] & 0
    }
\]
From this assertions 1 and 2 clearly follow. On the other hand $\delta_M$ is an isomorphism in $\mathcal{D}A$ if, and only if, $H^i(\RHom_A(T,M)\otimes_B^LT)=0$, for $i\neq 0$, and $H^0(\delta_M)$ is an isomorphism in $\Mod A$. The mentioned long exact sequence proves that the equalities $H^i(\RHom_A(T,M)\otimes_B^LT)=0$, for $i\neq 0$, hold exactly when conditions 3.a, 3.b and 3.d are satisfied. Therefore assertion 3 is also true. 
\end{proof}

\begin{proposition}\label{ordinary case:proposition}
Let $A$ be a right hereditary ordinary algebra, $B$ a dg algebra and $T$ a dg $B$-$A$-bimodule with bounded homology. The following assertions are equivalent:
\begin{itemize}
\item[1)] $\RHom_{A}(T,?):\cd A\ra\cd B$ is fully faithful.
\item[2)] The counit 
\[
\delta_{A^{(\alpha)}}:\RHom_{A}(T,A^{(\alpha)})\otimes^\L_{B}T\ra A^{(\alpha)}
\] 
is an isomorphism in $\cd A$ for each cardinal $\alpha$.
\item[3)] The counit 
\[\delta_{I}: \RHom_{A}(T,I)\otimes^\L_{B}T\ra I
\]
is an isomorphism in $\cd A$ for each injective $A$-module $I$.
\item[4)] The counit
\[\delta_{T^{(\alpha)}}: \RHom_{A}(T,T^{(\alpha)})\otimes^L_{B}T\ra T^{(\alpha)}
\]
is an isomorphism for all cardinals $\alpha$, and there is a cardinal $\beta$ such that $A\in\thick_{\cd A}(T^{(\beta)})$.
\end{itemize}
\end{proposition}
\begin{proof}
Implications 1) $\Rightarrow$ 2) and 1) $\Rightarrow$ 3) follow from part 2) of Lemma~\ref{equivalence between reflective and coreflective objects}.

2) $\Rightarrow$ 1) Let $\op{Coref}$ be the full subcategory of $\cd A$ formed by the coreflective objects with respect to the adjoint pair
\[\xymatrix{\cd A\ar@<1ex>[d]^{\RHom_{A}(T,?)} \\
\cd B\ar@<1ex>[u]^{?\otimes^\L_{B}T}
}
\]
Condition 2) says that $A^{(\alpha)}$ belongs to $\op{Coref}$ for each cardinal $\alpha$. Using that $\op{Coref}$ is thick (see Lemma~\ref{(co)reflective objects form a thick subcategory}) and that every $A$-module admits a finite projective resolution, we deduce that each $A$-module $X$ satisfies that $\Sigma^mX\in\op{Coref}$ for any integer $m\in\Z$. The fact that $A$ is right hereditary implies that for any complex $M$ of $A$-modules we have
\[\coprod_{n\in\Z}\Sigma^{-n}H^nM\cong M\cong\prod_{n\in\Z}\Sigma^{-n}H^nM.
\]
The task is then reduced to prove that if $\{M_{n}\}_{n\in\Z}$ is a family of $A$-modules, then $\coprod_{n\in\Z}\Sigma^nM_{n}\in\op{Coref}$. Indeed, thanks to Lemma~\ref{coproduct-product}, we have
\[
\RHom_{A}(T,\coprod_{n\in\Z}\Sigma^nM_{n})\cong\RHom_{A}(T,\prod_{n\in\Z}\Sigma^nM_{n})\cong\prod_{n\in\Z}\RHom_{A}(T,\Sigma^nM_{n}).
\]
Using that $T$ is isomorphic in $\cd A$ to a bounded complex of projective modules, we get that the family $\{\RHom_{A}(T,\Sigma^nM_{n})\}_{n\in\Z}$ satisfies the hypothesis of Lemma~\ref{coproduct-product}. Thus we still have
\[\prod_{n\in\Z}\RHom_{A}(T,\Sigma^nM_{n})\cong\coprod_{n\in\Z}\RHom_{A}(T,\Sigma^nM_{n}).
\]
Summarizing, we have
\[
\RHom_{A}(T,\coprod_{n\in\Z}\Sigma^nM_{n})\cong \coprod_{n\in\Z}\RHom_{A}(T,\Sigma^nM_{n}).
\]
It immediatly follows that $\coprod_{n\in\Z}\Sigma^{n}M_{n}$ is in $\op{Coref}$, as desired.

3) $\Rightarrow$ 1) The fact that each $A$-module $M$ admits an injective resolution 
\[0\ra M\ra I^0\ra I^1\ra 0
\]
implies that $\delta_{M}$ is an isomorphism and, hence, $M\in\op{Coref}$. Now use the same argument of the implication 2) $\Rightarrow$ 1).

1) $\Rightarrow$ 4) 

{\em Step 1: $\RHom_{A}(T,A)$ is a special Milnor colimit:} let us prove that we can express $\RHom_{A}(T,A)$ as a Milnor colimit in $\cd B$ of a sequence
\[P_{0}\ra P_{1}\ra P_{2}\ra\dots
\]
where $P_{0}$ and the cone over each morphism is a finite coproduct of objects of the form $\Sigma^m B^{(\alpha_{m})}$ for some integer $m$ and some cardinal $\alpha_{n}$. Indeed, to express $\RHom_{A}(T,A)$ as a Milnor colimit of finite extensions of coproducts of shifts of copies of $B$ we proceed as in the proof of \cite[Theorem 5.2]{KellerDDC} (see also \cite[Theorem 12.2]{KellerNicolas}). We start by considering a right approximation
\[\pi_{0}:P_{0}\ra\RHom_{A}(T,A)
\]
of $\RHom_{A}(T,A)$ with respect to the full subcategory formed by small coproducts of shifts of copies of $B$. But using that $T$ has bounded homology and that $A$ is right hereditary, we prove that $T$ is isomorphic in $\cd A$ to a bounded complex of projective $A$-modules. This implies that $\RHom_{A}(T,A)$ has bounded homology, and so $P_{0}$ can be taken to be a finite coproduct of objects of the form $\Sigma^m B^{(\alpha_{m})}$ for some integer $m$ and some cardinal $\alpha_{n}$. Consider the triangle
\[\xymatrix{
C_{0}\ar[r]^{\alpha_{0}} & P_{0}\ar[r]^{\pi_{0}\hspace{1cm}} & \RHom_{A}(T,A)\ar[r] & \Sigma C_{0}.
}
\]
The next step is to consider a right approximation
\[\beta_{0}: Z_{0}\ra C_{0}
\]
with respect to the full subcategory formed by small coproducts of shifts of copies of $B$. We define $P_{1}$ by using the triangle
\[
\xymatrix{
Z_{0}\ar[r]^{\alpha_{0}\beta_{0}} & P_{0}\ar[r] & P_{1}\ar[r] & \Sigma Z_{0}.
}
\]
The fact that both $P_{0}$ and $\RHom_{A}(T,A)$ have bounded homology, implies that $C_{0}$ also has bounded homology. Hence, we can take $Z_{0}$ to be finite coproduct of objects of the form $\Sigma^m B^{(\alpha_{m})}$ for some integer $m$ and some cardinal $\alpha_{n}$. Continuing in this way, we get the required sequence.

{\em Step 2:} After step 1, and taking into account the isomorphism
\[A\cong \RHom_{A}(T,A)\otimes^\L_{B}T,
\]
we conclude that $A$ is the Milnor colimit in $\cd A$ of a sequence
\[P_{0}\otimes^\L_{B}T\ra P_{1}\otimes^\L_{B}T\ra\dots
\]
where $P_{0}\otimes^\L_{B}T$ and all successive cones are finite coproducts of objects of the form $\Sigma^m T^{(\alpha_{m})}$ for some integer $m$ and some cardinal $\alpha_{n}$. The compactness of $A$ implies that the isomorphism
\[A\arr{\sim}\op{Mcolim}P_{n}\otimes^\L_{B}T
\]
factors through the canonical map
\[P_{k}\otimes^\L_{B}T\ra\op{Mcolim}P_{n}\otimes^\L_{B}T
\]
for some $k\in\N$. Therefore, $A$ is a direct summand of $P_{k}\otimes^\L_{B}T$, which belongs to $\thick_{\cd A}(T^{(\beta)})$ for some cardinal $\beta$.

4) $\Rightarrow$ 2) We know $A\in\thick_{\cd A}(T^{(\beta)})$ for some cardinal $\beta$. After \cite[Lemma 2.2.2]{BondalVanDenBergh} we have 
\[ \thick_{\cd A}(T^{(\beta)})=\op{add}_{\cd A}(\tria_{\cd A}(T^{(\beta)})).
\]
Thus $A$ is a direct summand of an object of $\tria_{\cd A}(T^{(\beta)})$. Hence, for any cardinal $\alpha$, we have that $A^{(\alpha)}$ is the direct summand of an object of $\tria_{\cd A}(T^{(\alpha\cdot \beta)})$. On the other hand, we are assuming that $T^{(\alpha\cdot\beta)}$ belongs to the set $\op{Coref}$ of coreflective objects with respect to the adjunction $(?\otimes^\L_{B}T,\RHom_{A}(T,?))$, which is  a thick subcategory of $\cd A$ (see Lemma~\ref{(co)reflective objects form a thick subcategory}). Thus $A^{(\alpha)}\in\op{Coref}$.
\end{proof}

\begin{corollary}\label{1tilting inducing ff}
Let $A$ and $B$ be ordinary algebras. Assume $A$ is right hereditary. Let $T$ be a $B$-$A$-bimodule. With the terminology as in Lemma 7.17, consider the following assertions:
\begin{itemize}
\item[1)] $\RHom_{A}(T,?):\cd A\ra\cd B$ is fully faithful.
\item[2)]  For each cardinal $\alpha$ the module $A^{(\alpha)}$ satisfies the following conditions:
\begin{itemize}
\item[a)] $(\tau_{i})_{A^{(\alpha)}}$ is an isomorphism for all $i>0$,
\item[b)] $(\tau_{0})_{A^{(\alpha)}}$ is a monomorphism,
\item[c)] $H^0(\delta_{A^{(\alpha)}})$ is an isomorphism,
\item[d)] $\Ext^1_{A}(T,A^{(\alpha)})\otimes_{B}T=0$,
\end{itemize}
where the morphisms $\tau_{i}$ are defined in Lemma~\ref{relating the two counits}.
\item[3)] For each injective right $A$-module $I$ the following conditions hold:
\begin{itemize}
\item[a)] $\op{Tor}^B_{i}(\Hom_{A}(T,I),T)=0$ for each $i>0$,
\item[b)] The counit
\[
\varepsilon_{I}:\Hom_{A}(T,I)\otimes_{B}T\ra I
\]
is an isomorphism.
\end{itemize}
\item[4)] $T$ satisfies:
\begin{itemize}
\item[a)] There is a short exact sequence $0\ra A\ra T_{0}\ra T_{1}\ra 0$ in $\Mod A$ where $T_{i}\in\op{Add}(T)$.
\item[b)] $\op{Tor}^B_{i}(\Hom_{A}(T,T^{(\alpha)}), T)=0$ for each $i>0$ and each cardinal $\alpha$,
\item[c)] The counit
\[\varepsilon_{T^{(\alpha)}}:\Hom_{A}(T,T^{(\alpha)})\otimes_{B}T\ra T^{(\alpha)}
\]
is an isomorphism for each cardinal $\alpha$.
\end{itemize}
\end{itemize}
Then 1), 2) and 3) are equivalent, and they are implied by 4). Moreover, if $\Ext^1_{A}(T,T^{(\alpha)})=0$ for each cardinal $\alpha$, the four conditions are equivalent.
\end{corollary}
\begin{proof}
After Lemma~\ref{relating the two counits} and Proposition~\ref{ordinary case:proposition}, the equivalence between 1), 2) and 3) is clear, and it is also clear that 4) implies 1). Assume now 1) holds and that $\op{Ext}^1_{A}(T,T^{(\alpha)})=0$ for each cardinal, and let us prove 4).  By Proposition~\ref{ordinary case:proposition} we know there exists a cardinal $\beta$ such that $A\in\thick_{\cd A}(T^{(\beta)})$. Since $\Ext^p_{A}(T^{(\beta)},T^{(\beta)})=0$ for all $p>0$, Proposition~\ref{A in thickT} says that there is an exact sequence
\[0\ra A\ra T^0\arr{d^0} T^1\arr{d^1}\dots\arr{d^{n-1}} T^n\ra 0
\]
with $T^{i}\in\op{Add}_{\Mod A}(T)$ for all $i=0,1,\dots,n$. Now put $Z^{i}=\ker(d^{i})$. If $n>1$ then for any $A$-module $M$ we have an epimorphism 
\[\Ext^1_{A}(M,T^{n-1})\ra\Ext^1_{A}(M,Z^{n-1})
\]
since $\Ext^2_{A}(?,?)=0$. Therefore, $\Ext^1_{A}(?,Z^{n-1})$ vanishes on $\op{Add}_{\Mod A}(T)$ because so does $\Ext^1_{A}(?,T^{n-1})$. It follows that the sequence
\[0\ra Z^{n-1}\ra T^{n-1}\ra T^n\ra 0
\]
splits and so $Z^{n-1}\in\op{Add}_{\Mod A}(T)$. Repeating the argument we get that the objects $Z^{n-1},\dots, Z^1\in\op{Add}_{\Mod A}(T)$ and so condition a) holds.

Conditions b) and c) follow easily after applying $H^{i}$ to the counit
\[\delta_{T^{(\alpha)}}:\RHom_{A}(T,T^{(\alpha)})\otimes^\L_{B}T\ra T^{(\alpha)}.
\]
\end{proof}

\begin{remark}
If in Corollary~\ref{1tilting inducing ff} condition $\Ext^1_{A}(T,T^{(\alpha)})=0$ holds, then $T$ is $1$-tilting (see Definition~\ref{tilting module}).
\end{remark}

\subsection{Study of some natural families of modules}

\begin{notation}\label{notation of types of modules}
Let $A$ be an ordinary algebra. Let $\cp^{<\infty}(A)$ be the full subcategory of $\Mod A$ formed by those modules with finite projective dimension. Consider the following full subcategories of $\cp^{<\infty}(A)$:
\begin{itemize}
\item $\ce$, formed by those modules $T$ such that $\Ext^{p}_{A}(T,T)=0$ for $p>0$.
\item $\ct$, formed by those modules which are $(n\text{-})$tilting modules.
\item $\ch^b$, formed by those modules such that as modules over their endomorphism algebra have a finite projective resolution with finitely generated terms.
\item $\cR$, formed by those modules $T$ such that the functor $\RHom_{A}(T,?):\cd A\ra\cd B$ is fully faithful, where $B=\End_{A}(T)$.
\item $\widehat{\cR}$, formed by those modules $T$ such that the functor $\RHom_{A}(T,?):\cd A\ra\cd B$ is fully faithful, preserves compact objects and $B$ is in its essential image, where $B=\End_{A}(T)$.
\end{itemize}
\end{notation}

\begin{theorem}\label{main result ordinary case}
Let $A$ be an ordinary algebra. If $T$ is a right $A$-module, we will use the letter $B$ to refer to $\End_{A}(T)$. The following assertions hold:
\begin{itemize}
\item[1)] Using the Notation~\ref{notation of types of modules} we have: $\widehat{\cR}=\ce\cap\cR\cap\ch^b$ and $\ct\cap\widehat{\cR}$ is precisely the class of good tilting modules.
\item[2)] If $\Ext^{p}_{A}(T,T)=0$ for each $p>0$ and $\RHom_{A}(T,?):\cd A\ra \cd B$ is fully faithful, then $T$, as a left $B$-module, has a projective resolution consisting of finitely generated left $B$-modules.
\item[3)] There are right $A$-modules $T$ satisfying the following conditions:
\begin{itemize}
\item[a)] $\op{pd}_{A}T\leq 1$,
\item[b)] the functor $\RHom_{A}(T,?):\cd A\ra\cd B$ is fully faithful, preserves compact objects and $B$ is in its essential image,
\item[c)] $\Ext_A^1(T,T^{(\mathbf{N})})\neq 0$, and so $T$ is not a tilting module.
\end{itemize}
\item[4)] There are right $A$-modules $T$ satisfying the following conditions:
\begin{itemize}
\item[a)] $\op{pd}_{A}T\leq 1$,
\item[b)] the functor $\RHom_{A}(T,?):\cd A\ra\cd B$ is fully faithful,
\item[c)] $T$, regarded as a left $B$-module, has a finite projective resolution with finitely generated terms,
\item[d)] $\Ext^{1}_{A}(T,T)\neq 0$.
\end{itemize}
\item[5)] There are right $A$-modules $T$ satisfying the following conditions:
\begin{itemize}
\item[a)] $T$ is a $1$-tilting module,
\item[b)] $T$, regarded as a left $B$-module, does not have a projective resolution with finitely generated terms (and hence $\RHom_{A}(T,?):\cd A\ra\cd B$ is not fully faithful).
\end{itemize}
\item[6)] There are right $A$-modules $T$ satisfying the following conditions:
\begin{itemize}
\item[a)] $\op{pd}_{A}T\leq 1$,
\item[b)] $T$, regarded as a left $B$-module, has a finite projective resolution with finitely generated terms,
\item[c)] $\Ext^{p}_{A}(T,T)=0$ for each $p>0$,
\item[d)] $\RHom_{A}(T,?):\cd A\ra\cd B$ is not fully faithful.
\end{itemize}
\end{itemize}
\end{theorem}
\begin{proof}
1) If $T\in\widehat{\cR}$, then Theorem~\ref{special localization result} says that $T\in\per(B^\text{op})$, and that $\REnd_{A}(T)\cong B$, from which we deduce that $T\in\ce\cap\ch^b$. Conversely, if $T\in\ce\cap\cR\cap\ch^b$, then $T\in\per(B^\text{op})$ and so $?\otimes^\L_{B}T:\cd B\ra\cd A$ preserves products (see the implication $1)$-$5)\Rightarrow 6)$ in the proof of Theorem~\ref{special localization result}). Then, by Lemma~\ref{lemma:existence of left adjoints}, we know that $?\otimes^\L_{B}T$ has a left adjoint which, by Lemma~\ref{two-sided fully faithful adjoint}, is fully faithful. On the other hand, the fact that $B=\End_{A}(T)$ and $T\in\ce$ imply that the unit map $\lambda: B\ra\RHom_{A}(T,B\otimes^\L_{B}T)$ is an isomorphism. It follows that condition $6)$ in Theorem~\ref{special localization result} holds, so that $T\in\widehat{\cR}$.

The fact that $\ct\cap\widehat{\cR}$ is the class of good tilting modules is a direct consequence of Corollary~\ref{some characterization of good tilting}.

2) From the hypothesis we have $\op{Ext}^{i}_{A}(T,T^{\alpha})=0$ for each $i>0$ and each cardinal $\alpha$. Thus
\[\RHom_{A}(T,T^{\alpha})\cong\Hom_{A}(T,T^{\alpha})\cong B^\alpha
\]
in $\cd B$. Now, using the counit of the adjunction $(?\otimes^\L_{B}T,\RHom_{A}(T,?))$, we have an isomorphism
\[B^\alpha\otimes^\L_{B}T\cong\RHom_{A}(T,T^\alpha)\otimes^\L_{B}T\arr{\sim}T^\alpha.
\]
Using Proposition~\ref{tensoring preserves products} below we finish the proof.

3) Assume $A$ is a non-N\oe therian right hereditary algebra. Let us take
\[T=E(A)\oplus E(A)/A\oplus I,
\]
where $E(A)$ is an injective envelope of $A$ and $I$ is any injective cogenerator containing a copy of each cyclic $A$-module as a submodule. Then $T$ satisfies condition 3) of Corollary~\ref{cor.tilting buenos via RHom}, 
and so condition b) of 3) holds. However, by a result of Faith (see \cite[Theorems 25.1 and 25.3]{AF}) we know that $T^{(\N)}$ is not injective. Now, we apply the argument of the second step in \cite[Lemma 7]{NicolasSaorinLiftingRestricting} to see that $\op{Ext}^{1}_{A}(T,T^{(\N)})\neq 0$.

4) Let $A$ be a right hereditary algebra and let $T$ be a generator of $\Mod A$. Then we have a decomposition $A\oplus X=T^{n}$ in the category of right $A$-modules, for some natural number $n$. From here we get a decomposition
\[T\oplus\Hom_{A}(X,T)=B^n
\]
in the category of left $B$-modules. In particular, $T$ is a finitely generated and projective left $B$-module, and so it is flat over $B$. In particular, conditions  a) and b) of part 3) of Lemma~\ref{relating the two counits} hold, and $H^0(\delta_{M})$ gets identified with $\varepsilon_{M}$ for any module $M$. But Gabriel-Popescu theorem (see \cite[Theorem X.4.1]{Stenstrom}) implies that $\varepsilon_{M}$ is an isomorphism. According to assertion 2) of Corollary~\ref{1tilting inducing ff}, our task reduces to find the generator $T$ such that $\op{Ext}^{1}_{A}(T,T)\neq 0$ and $\op{Ext}^{1}_{A}(T,A^{(\alpha)})\otimes_{B}T=0$ for each cardinal $\alpha$. To do that we take any $A$-module $X$ such that $\Hom_{A}(X,A)=0=\op{Ext}^{1}_{A}(X,X)$ and $\op{Ext}^{1}_{A}(X,A)\neq 0$ (for instance, $X$ can be any non-projective simple module, assuming $A$ to be an Artin algebra or, more generally, a right artinian algebra). Our choice is 
\[T=A\oplus X, 
\]
in which case
\[B=\left[\begin{array}{cc} A & 0 \\ X & B'\end{array}\right],
\]
with $B'=\End_{A}(X)$. Note that in the category of left $B$-modules we have
\[T\cong B\cdot e_{1}
\]
where
\[e_{1}=\left[\begin{array}{cc}1 & 0 \\ 0 & 0\end{array}\right].
\]
When we view right $B$-modules as `rows' in the customary way (see \eg \cite[Proposition III.2.2]{ARS}), we get
\begin{align}
\op{Ext}^1_{A}(T,A^{(\alpha)})\otimes_{B}T= \nonumber \\
\op{Ext}^1_{A}(T,A^{(\alpha)})\otimes_{B}Be_{1}= \nonumber \\
\op{Ext}^1_{A}(T,A^{(\alpha)})e_{1}= \nonumber \\
\left[\begin{array}{cc} \op{Ext}^{1}_{A}(A,A^{(\alpha)}) & \op{Ext}^{1}_{A}(X,A^{(\alpha)})\end{array}\right]e_{1}= \nonumber \\
\left[\begin{array}{cc} 0 & \op{Ext}^1_{A}(X,A^{(\alpha)})\end{array}\right]\cdot\left[\begin{array}{cc} 1 & 0 \\ 0 & 0\end{array}\right]=0. \nonumber
\end{align}

5) Let $A$ be a right N\oe therian and right hereditary algebra, let $E(A)$ be an injective envelope of $A$ in the category of right $A$-modules. Assume that 
\[
\Hom_{A}(E(A)/A,E(A))=0
\] 
and there is an indecomposable direct summand $I_{0}$ of $E(A)/A$ with infinite multiplicity. Let $I$ be the direct sum of one isomorphic copy of each indecomposable direct summand of $E(A)/A$. Note that the fact that $A$ is right hereditary implies that $E(A)/A$ is injective, and so $I$ is also injective. The injective module
\[T=E(A)\oplus I
\]
is clearly $1$-tilting. Moreover, we have
\[B=\End_{A}(T)\cong\left[\begin{array}{cc} \End_{A}(E(A)) & 0 \\ \Hom_{A}(E(A),I) & \End_{A}(I)\end{array}\right],
\]
and the exact sequence
\[0\ra A\ra E(A)\ra E(A)/A\ra 0
\]
of right $A$-modules yields an exact sequence
\[0\ra \Hom_{A}(E(A)/A,I)\ra \Hom_{A}(E(A),T)\ra T\ra 0
\]
of left $B$-modules, with
\[\Hom_{A}(E(A),T)\cong Be_{1},
\]
where
\[e_{1}=\left[\begin{array}{cc} 1 & 0 \\ 0 & 0\end{array}\right].
\]
In particular, $Be_{1}$ is a projective left $B$-module, and if $T$, regarded as a left $B$-module, had a projective resolution with finitely generated terms, then the left $B$-module $\Hom_{A}(E(A)/A,T)$ would be finitely generated. Fix now an epimorphism
\[p: B^{n}\ra \Hom_{A}(E(A)/A,T)
\]
of left $B$-modules. Fix also a decomposition
\[E(A)/A\cong I_{0}^{(\N)}\oplus E'.
\]
Note that $I_{0}$ is a direct summand of $T$, and let $e: T: \twoheadrightarrow I_{0}\hookrightarrow T$ be the corresponding idempotent of $B$. Thus, for each finite subset $F$ of $\N$, we have that
\[(Be)^{(F)}\cong\Hom_{A}(I_{0}^{(F)},T)
\]
is a projective direct summand of $\Hom_{A}(E(A)/A,T)$, and hence, also a direct summand of $B^n$. Therefore, 
\[I_{0}^{(F)}\cong\Hom_{B^\text{op}}((Be)^{(F)},T)
\]
is a direct summand of $\Hom_{B^\text{op}}(B^n,T)\cong T^n$ in the category of right $A$-modules. But this is a contradiction for it is well-known that $T$ is a direct sum of modules with local endomorphism ring (see \cite[Theorem 3.48 and 3.52]{Lam}). Then due to Azumaya's theorem (see \cite[12.6]{AF} and \cite[Remark after Corollary 3.53]{Lam}), the decomposition of $T^n$ into a direct sum of indecomposables is unique and we know that $I_{0}$ appears in it with multiplicity exactly $n$.

It remains to prove the existence of a right N\oe therian right hereditary algebra as the one required in the previous proof. Consider the (first) Weyl algebra 
\[A=A_{1}(k)=k\langle p, q\rangle/(pq-qp-1)
\]
over a field $k$ with characteristic $0$, the vector space $S=k[x]$ of polinomials in one variable with coefficients in $k$, and the $k$-algebra morphism (known as the {\em standard representation})
\[A\ra\End_{k}(S)
\]
which sends $p$ to $d/dx$ and $q$ to multiplication by $x$. It is well-known that $S$ becomes a simple $A$-module with $\End_{A}(S)\cong k$ (see \cite[\S 4.6.3]{Dixmier}). From the fact that $A$ has no non-zero divisor (see \cite[Corollary 7.11.3]{McConnell-Robson}) we deduce that 
\[
\Hom_{A}(S,E(A))=0=\Hom_{A}(E(A)/A,E(A)), 
\]
and so 
\[\Hom_{A}(S,E(A)/A)\arr{\sim}\Ext^1_{A}(S,A).
\]
It remains to check that $\Ext^1_{A}(S,A)$ is infinite dimensional over $k$, for that will mean that $I_{0}=E(S)$ is an indecomposable direct summand of $E(A)/A$ with infinite multiplicity. Indeed, a projective resolution of $S$ is given by
\[\xymatrix{
0\ar[r] & A\ar[r]^{\lambda} & A\ar[r] & S\ar[r] & 0
}
\]
where $\lambda(f)=qf$. Then we get an exact sequence of left $A$-modules
\[\xymatrix{
0=\Hom_{A}(S,A)\ar[r] & A\ar[r]^{\lambda^{\che}} & A\ar[r] & \Ext^{1}_{A}(S,A)\ra 0
}
\]
where $\lambda^{\che}(f)=fq$. Therefore,
\[\Ext^1_{A}(S,A)\cong A/Aq\cong k[x]
\]
is infinite-dimensional.

6) Take any hereditary Artin algebra $A$ and let $T$ be a finitely generated projective $A$-module which is not a generator of the category of right $A$-modules. Then condition a) of part 4) of Corollary~\ref{1tilting inducing ff} is not satisfied and so $\RHom_{\cd A}(T,?):\cd A\ra\cd B$ is not fully faithful. On the other hand, $B$ is also a hereditary Artin algebra and $T$ is a finitely generated left $B$-module. Therefore, as a left $B$-module, $T$ has a projective resolution consisting of finitely generated terms.
%


\end{proof}

\begin{proposition}\label{tensoring preserves products}
Let $B$ be an ordinary algebra and let $T$ be a complex of left $B$-modules such that $H^pT=0$ for $p\gg 0$. The following assertions are equivalent:
\begin{itemize}
\item[1)] $T$ is quasi-isomorphic to a right bounded complex of finitely generated projective left $B$-modules.
\item[2)] The canonical map $B^\alpha\otimes^\L_{B}T\ra T^\alpha$ is an isomorphism in $\cd(B^\text{op})$ for each cardinal $\alpha$.
\end{itemize}
\end{proposition}
\begin{proof}
$1)\Rightarrow 2)$ Without loss of generality, we can take the complex $P$ of finitely generated projective left $B$-modules in 1) vanishing in degrees $\geq 1$. Let
\[f:P\ra T
\]
be a quasi-isomorphism. Since $P$ is homotopically projective, we have a commutative square in the category of complexes of $B$-modules
\[\xymatrix{
\hspace{1.5cm}B^{\alpha}\otimes_{B}P=B^\alpha\otimes^\L_{B}P\ar[r]\ar[d]_{\id\otimes^\L f} & P^\alpha\ar[d]^{f^\alpha} \\
B^\alpha\otimes^\L_{B}T\ar[r] & T^\alpha
}
\]
It is clear that $f^\alpha$ and $\id\otimes^\L f$ are isomorphisms in $\cd(B^\text{op})$. Hence, the horizontal morphism in the bottom is an isomorphism if and only if so is the horizontal one in the top, which is the case as one can deduce from \cite[Lemma I.13.2]{Stenstrom}.

$2)\Rightarrow 1)$ Without loss of generality, we can assume that $H^pT=0$ for $p>0$. Thus $T$ is quasi-isomorphic to its canonical truncation $\tau_{\leq 0}T$, which, thanks to Lemma~\ref{preserving acyclics under tensors}, also satisfies that the map
\[B^\alpha\otimes^\L_{B}(\tau_{\leq 0}T)\ra (\tau_{\leq 0}T)^\alpha
\]
is a quasi-isomorphism for each cardinal $\alpha$. Note that $\tau_{\leq 0}T$ is quasi-isomorphic to a complex $P$ of projective left $B$-modules such that $P^n=0$ for $n>0$, and thus $P$ also satisfies that $B^\alpha\otimes^\L_{B}P\ra P^\alpha$ is a quasi-isomorphism for each cardinal $\alpha$. Since $P$ is an $\ch$-projective left dg $B$-module, then $B^\alpha\otimes^\L_{B}P=B^\alpha\otimes_{B}P$. Since $B^\alpha\otimes_{B}P\ra P^\alpha$ is a quasi-isomorphism, we have that a commutative square in the category of left $B$-modules,
\[\xymatrix{
H^0(B^\alpha\otimes_{B}P)\ar[r]^{\sim} & H^0(P^\alpha)\ar[d]^{\wr} \\
B^\alpha\otimes_{B}H^0(P)\ar[r]\ar[u]^{\wr} & H^0(P)^\alpha
}
\]
This implies that $B^\alpha\otimes_{B}H^0(P)\ra H^0(P)^\alpha$ is an isomorphism for each cardinal $\alpha$, and so, after \cite[Lemma I.13.2]{Stenstrom}, we deduce that $H^0(P)$ is a finitely presented left $B$-module. Therefore, we can replace $P$ by a quasi-isomorphic complex $P_{0}$ of projective left $B$-modules such that $P_{0}^n=0$ for $n>0$ and $P_{0}^0$ is finitely generated.

In what follows we inductively construct a sequence
\[\dots\ra P_{2}\arr{f_{2}}P_{1}\arr{f_{1}}P_{0}
\]
of quasi-isomorphisms of complexes of left $B$-modules satisfying:
\begin{itemize}
\item[a)] $P_{n}^{i}=0$ for all $i>0$,
\item[b)] $P_{n}^{i}$ is finitely generated and projective for each $-n\leq i\leq 0$,
\item[c)] $P_{n}^{i}=P_{n-1}^{i}$ and $f_{n}^{i}:P_{n}^{i}\ra P_{n-1}^{i}$ is the identity map for $-n+1\leq i\leq 0$.
\end{itemize}
Once this sequence has been constructed, we take the limit $Q=\op{lim}_{n\geq 0}P_{n}$ in the category of complexes of left $B$-modules. In fact, we have an explicit description of it:
\[
\xymatrix{
Q: \dots\ra P_{2}^{-2}\ar[r]^{\hspace{0.7cm}d_{2}^{-2}} & P_{1}^{-1}\ar[r]^{\hspace{0.2cm}d_{1}^{-1}} & P_{0}^0\ar[r] & 0\ar[r] & 0\ar[r] & \dots
}
\]
Note that the maps $Q\ra P_{n}\ko n\geq 0$, are quasi-isomorphisms, because for each $i\leq 0$ the map $Q^{i}_{n+1}\ra Q^{i}_{n}$ is the identity for all $n\geq -i$. In particular, $Q$ is a right bounded complex of finitely generated left $B$-modules which is quasi-isomorphic to $P_{0}$, and so quasi-isomorphic to $T$.

Let us construct the sequence of quasi-isomorphisms. Let $n>0$ and suppose the sequence
\[
\xymatrix{
P_{n-1}\ar[r]^{f_{n-1}} & P_{n-2}\ar[r] & \dots\ar[r]^{f_{1}} & P_{0}
}
\]
has been already defined. Put $X=P_{n-1}$ for simplicity. Taking stupid truncations, the induction hypothesis says that $\sigma^{\geq -n+1}X$ is a perfect complex. Considering the triangle
\[\sigma^{\geq -n+1}X\ra X\ra\sigma^{\leq -n}X\ra\Sigma\sigma^{\geq -n+1}X,
\]
we get the following morphism of triangles in the derived category of left $B$-modules:
\[\xymatrix{
B^\alpha\otimes^\L_{B}\sigma^{\geq -n+1}X\ar[d]^{\wr}\ar[r] & B^\alpha\otimes^\L_{B}X\ar[d]^{\wr}\ar[r] & B^\alpha\otimes^\L_{B}\sigma^{\leq -n}X\ar[r]^{\hspace{1cm}+}\ar[d] & \\
(\sigma^{\geq -n+1}X)^\alpha\ar[r] & X^\alpha\ar[r] & (\sigma^{\leq -n}X)^\alpha\ar[r]^{\hspace{0.5cm}+} &
}
\]
Since the implication 1) $\Rightarrow$ 2) of this proposition has been already proved, we know that the leftmost vertical map is an isomorphism. On the other hand, since $X$ is quasi-isomorphic to $T$, the central map is also an isomorphism in $\cd(B^{\text{op}})$. Using the 5-Lemma we can prove that the rightmost vertical map is an isomorphism. In particular, we have isomorphisms
\begin{align}
B^\alpha\otimes_{B}H^{-n}(\sigma^{\leq -n}X)\cong \nonumber \\
\cong H^{-n}(B^\alpha\otimes_{B}\sigma^{\leq -n}X)\cong \nonumber \\
\cong H^{-n}(B^\alpha\otimes^\L_{B}\sigma^{\leq -n}X)\cong  \nonumber \\
\cong H^{-n}((\sigma^{\leq -n}X)^\alpha)\cong \nonumber \\
\cong H^{-n}((\sigma^{\leq -n}X))^\alpha \nonumber
\end{align}
for each cardinal $\alpha$, which implies that
\[
H^{-n}((\sigma^{\leq -n}X))=\op{coker}(X^{-n-1}\ra X^{-n})
\]
is a finitely presented left $B$-module. We can now construct a quasi-isomorphism 
\[g:Y\ra X
\]
where $Y^{i}=X^{i}$ and $g^{i}=\id$ for $-n+1\leq i\leq 0$, the left $B$-module $Y^{i}$ is projective for each $i\leq 0$, and $Y^{-n}$ is finitely generated. Indeed, we fix an epimorphism $p: Y^{-n}\ra H^{-n}X$ with $Y^{-n}$ finitely generated projective, and take a lift $f: Y^{-n}\ra X^{-n}$:
\[\xymatrix{
&&& Y^{-n}\ar@{->>}[dd]^{p}\ar@{.>}[dl]_{f}& \\
X^{-n-1}\ar[rr]\ar@{->>}[dr] && X^{-n}\ar[rr]\ar@{->>}[dr] && X^{-n+1} \\
& B^{-n}\ar@{^(->}[ur] && H^{-n}X &
}
\]
Consider now the following commutative diagram with bicartesian squares
\[\xymatrix{M\ar[d]\ar@{->>}[r] & f^{-1}(B^{-n})\ar[d]\ar@{^(->}[r] & f^{-1}(Z^{-n})\ar[d]\ar@{^(->}[r] & Y^{-n}\ar@{->>}[d]^{f} \\
X^{-n-1}\ar@{->>}[r] & B^{-n}\ar@{^(->}[r] & Z^{-n}\ar@{^(->}[r] & X^{-n}
}
\]
Since
\[Z^{-n-1}=\ker(d^{-n-1})\cong\ker(M\ra f^{-1}(B^{-n}))
\]
we get another bicartesian square
\[\xymatrix{\tilde{Z}^{-n-1}\ar[r]\ar@{->>}[d] & Y^{-n-1}\ar@{->>}[d]^{\varepsilon} \\
Z^{-n-1}\ar@{^(->}[r] & M
}
\]
where $\varepsilon$ is an epimorphism and $Y^{-n-1}$ is projective. All this information get assembled in a commutative diagram
\[\xymatrix{
\tilde{Z}^{-n-1}\ar@{^(->}[dr]\ar@{->>}[ddd]&&&&& \\
& Y^{-n-1}\ar[rr]\ar@{->>}[dr]_{\varepsilon} && Y^{-n}\ar[ddd]^{f}\ar[r] & X^{-n+1}\ar[r]\ar[ddd]^{\id} & \dots \\
&& M\ar[ddl]\ar[ur] &&& \\
Z^{-n-1}\ar@{^(->}[dr] &&&&& \\
& X^{-n-1}\ar[rr] && X^{-n}\ar[r] & X^{-n+1}\ar[r] & \dots
}
\]
It is clear how to continue taking pullbacks and going to the left in order to construct the desired quasi-isomorphism $g:Y\ra X$. We then put $P_{n}=Y$ and $f_{n}=g$.
\end{proof}

\end{document}